\documentclass[12pt]{amsbook}
\usepackage{amsmath}
\usepackage{amssymb}
\usepackage{amsthm}

\input xy
\xyoption{matrix}
\xyoption{arrow}
\xyoption{curve}

\theoremstyle{plain}
\newtheorem{theorem}{Theorem}[section]
\newtheorem{corollary}[theorem]{Corollary}
\newtheorem{proposition}[theorem]{Proposition}
\newtheorem{lemma}[theorem]{Lemma}

\theoremstyle{definition}
\newtheorem{definition}[theorem]{Definition}

\newcommand{\br}[1]{\langle#1\rangle}
\newcommand{\C}{{\mathcal{C}}}

\newcommand{\Cat}{{\mathbf{Cat}}}

\newcommand{\DD}{{\mathcal{D}}}
\newcommand{\D}{\mathbb{D}}

\newcommand{\e}{\varepsilon}

\newcommand{\etahat}{\hat{\eta}}

\newcommand{\id}{{\mathop{\textnormal{id}}\nolimits}}
\long\def\ignore#1\endignore{}

\newcommand{\JJ}{{\mathcal{J}}}
\newcommand{\J}{{\mathbb{J}}}
\newcommand{\Lhat}{{\hat{L}}}
\newcommand{\M}{{\mathbf{M}}}

\newcommand{\op}{{\textnormal{op}}}

\def\S{{\mathcal{S}}}
\newcommand{\Set}{{\mathbf{Set}}}
\newcommand{\Shat}{\hat{S}}

\newcommand{\That}{\hat{T}}

\begin{document}

\title{Left Adjoints for Generalized Multicategories}

\author{A.\ D.\ Elmendorf}

\address{Department of Mathematics\\
Purdue University Calumet\\
 Hammond, IN 46323}
 \thanks{The author was partially supported by a grant from the Simons Foundation
 (Collaboration Grant for Mathematicians number 209092 to Anthony Elmendorf)}

\email{aelmendo@purduecal.edu}

\date{February 14, 2015}

\begin{abstract}
We construct generalized 
multicategories associated to an arbitrary operad in 
$\Cat$ that is $\Sigma$-free.  The construction generalizes the passage
to symmetric multicategories from permutative categories, 
which is the case when the operad is the categorical version
of the Barratt-Eccles operad.  The main
theorem is that there is an adjoint pair relating algebras over the
operad to this sort of generalized multicategory.  The construction 
is flexible enough to allow for equivariant generalizations.
\end{abstract}

\maketitle

\tableofcontents

\part{Introduction and constructions}
\section{Introduction}
This paper describes a new construction for multicategories that generalizes
the notions of both symmetric and non-symmetric multicategory in the literature.
While not as general as Cruttwell and Shulman's construction in \cite{CS}, we
are able to give a left adjoint construction to algebras over the operad
that controls the construction.  This generalizes the free strict monoidal category 
construction for non-symmetric multicategories and the free permutative
category construction for symmetric multicategories.  The only hypothesis 
necessary for the new construction 
and the left adjoint
is the $\Sigma$-free property on the controlling
categorical operad; that is, the $n$th category $\DD_n$ of the operad must have a $\Sigma_n$-action
that is free.  The special case in which $\DD_n$ is the discrete category $\Sigma_n$
gives the usual construction of non-symmetric multicategories, and the 
special case of the categorical version of the Barratt-Eccles operad gives
the usual construction of symmetric multicategories.  

Our construction is motivated by Leinster's work in \cite{L}, where he
showed that monads on certain categories such as
sets or categories, subject to a few hypotheses, could give rise to generalized
multicategories; however, he did not fit symmetric multicategories into his
framework.  It is also motivated by the work of Guillou and May in \cite{GM}, 
in which they described how algebras
over categorical operads in an equivariant setting could give rise to genuine equivariant
spectra.  However, they did not consider any associated theory of 
multicategories.  Such a theory would be a generalization of the theory
given in \cite{EM2} by Mandell and the author in the non-equivariant setting,
in which we showed that the spectrum determined by a permutative category
depends only on its underlying (symmetric) multicategory.  The present construction
does generalize Leinster's to the symmetric setting, as well as accounting for 
the context considered by Guillou and May; it also allows for many other examples.
Applications to stable homotopy are a subject for future work.

The crucial feature that we exploit in all our constructions is
is the presence of an operad of categories that is $\Sigma$-free, meaning that the groups
$\Sigma_n$ not only act, but act \emph{freely} on the associated components of the operad.  This
freeness is critical: it explains the failure of Leinster's construction when
considering strictly commutative monoidal categories, and makes clear that
the most natural generalization of his construction to the symmetric case
is to permutative categories, since they  are precisely the algebras over the $\Sigma$-free
categorical Barratt-Eccles operad.
We make heavy use of the $\Sigma$-free hypothesis, but need to 
assume nothing more about our categorical operad.  In particular, we make no use of other
properties enjoyed by the Barratt-Eccles operad, such as its $E_\infty$
structure.  

Our main results are as follows.  Given an operad $\DD$ of 
categories that is $\Sigma$-free, we construct a category of generalized
multicategories associated to the monad $\D$ given by the operad $\DD$.
We call these $\D$-multicategories.
We show that every $\D$-algebra has an underlying $\D$-multicategory,
and that there is a left adjoint to this forgetful functor back to $\D$-algebras.
In the case in which the operad $\DD$ is
the categorical Barratt-Eccles operad, the algebras are permutative 
categories, which in turn all have underlying (symmetric) multicategories;
our construction recovers these multicategories and the left adjoint
back to permutative categories.  

The most novel feature of the paper is a construction that comes within
a whisker of being the left adjoint, but fails on exactly one count: the unit
map fails to preserve the presheaf structure we impose on the morphisms of our
sort of multicategory.  We call this the ``provisional'' left adjoint, since it
does most of the heavy lifting in the construction of the actual left adjoint, which
is a coequalizer (really a coend) constructed from the provisional version.  

The paper is organized as follows.  Since the proofs consist almost entirely
of lengthy diagram chases, these are deferred to the second part of the paper,
which can be ignored by the casual reader.  The definition and main constructions, 
as well as the statements of the main results, are in the first part, consisting
of the first six sections.  However, these sections, though rather short, contain
no proofs at all.  The diligent reader will find these by perusing the corresponding
sections in the second part.

The results of this paper should be considered as proof of concept theorems, 
since there are many other questions about the construction
that this paper does not address.  Perhaps chief among these is the question
of producing spectra, especially equivariant spectra, from particular instances of the construction.  However,
this will definitely require more hypotheses than those we use in this paper, 
which is long enough as is without investigation of further questions.

It is a pleasure to acknowledge conversations
with Mike Shulman on these topics.  Conversations with Anna Marie 
Bohmann and Bj\o rn Dundas have also been stimulating and useful.

\section{Preliminary Constructions}

We begin by describing 
some preliminary
constructions which are necessary for our definition of $\D$-multicategories,
deferring all proofs to later sections.  We first need some
results about categorical operads, and the consequences we will need
that follow from our $\Sigma$-free hypothesis.  

Given an operad $\DD$ in $\Cat$, we can form the associated 
monad $\D$ in $\Cat$ by means of the standard construction
\[
\D\C:=\coprod_{n\ge0}\DD_n\times_{\Sigma_n}\C^n.
\]
It will be most useful, however, to break this monad up into pieces, using
the nerve construction $N:\Cat\to\mathord{\textbf{Ssets}}$, where the target
is the category of simplicial sets.  Since $N$ is a right adjoint, it preserves products
and therefore operad structure, so $N\DD$ is an operad of simplicial sets, and since
evaluation at any simplicial degree is also a right adjoint, we get the sequence
of (set) operads $N\DD_0,N\DD_1,N\DD_2,\dots$.  Associated to each we have 
a monad $\D_0,\D_1,\D_2,\dots$ in $\Set$.  Now given a (small) category $\C$, we also denote
by $\C_0,\C_1,\C_2,\dots$ the simplices of the nerve of $\C$, so in particular $\C_0$
is the set of objects of $\C$, and $\C_1$ its set of morphisms.   It is now straightforward
to see that the objects of $\D\C$ are $\D_0\C_0$ and the morphisms are $\D_1\C_1$; 
higher nerve degrees are also given by $\D_n\C_n$.  It will be useful to consider mixed
uses of these monads, however; for example, we will need to consider both $\D_1\C_0$
and $\D_0\C_1$.  We will make much use of the following as a calculational tool; its
proof makes crucial use of the $\Sigma$-free hypothesis.

\begin{lemma}\label{category object}
If $\DD$ is $\Sigma$-free, the monads $\{\D_0,\D_1\}$ give a category object in
monads on $\Set$, which in turn determines the remaining monads $\D_2,\D_3,\dots$.
\end{lemma}

We also need the fact that all these monads are what Leinster calls \emph{Cartesian}.
We recall the definition.

\begin{definition} A monad $J$ is \emph{Cartesian} if it satisfies the following properties:
\begin{enumerate}
\item
$J$ preserves pullbacks.
\item
The naturality diagrams for the product $\mu:J^2\to J$ and the unit $\eta:\id\to J$
are all pullbacks.
\end{enumerate}
\end{definition}

Our first basic result is the following theorem; note that since the categorical operad
$\DD$ is assumed to be $\Sigma$-free, so are all the set operads $N\DD_n$.

\begin{theorem}\label{cartesianmonad}
The monad associated to a $\Sigma$-free operad in $\Set$ is Cartesian.
Consequently, all the monads $\D_n$ are Cartesian.
\end{theorem}

Our construction of generalized multicategories associated to the monad $\D$
involves the use of presheaves.  
We therefore need to introduce our conventions and notation for presheaves, and 
give the fundamental lemmas about them that will be used in the rest of the paper.

When considering a fiber product $X\times_ZY$, we will consider the 
structure map from $X$ to $Z$ to be of ``source'' type, and the structure map from $Y$
to $Z$ to be of ``target'' type, where the use of ``source'' and ``target'' should be clear
from context.  
The motto is ``source to the right, target to the left.''
Our motivating example is to be able to consider the composition map in a small category
$\C$ to be given by a map
\[
\gamma:\C_1\times_{\C_0}\C_1\to\C_1;
\]
the structure map for the first $\C_1$ is the source map, and for the second one the
target map.  With this in mind, we have the following characterization of presheaves.

\begin{definition}
Let $X$ be a set, and $\C$ a small category.  A \emph{presheaf structure} on $X$ consists
of a map $\e:X\to\C_0$, together with an action map
\[
\xi: X\times_{\C_0}\C_1\to X.
\]
\end{definition}

Of course, by an action map we mean the usual coherence conditions of associativity
and unit must be satisfied.  
Further, we consider $X\times_{\C_0}\C_1$ to have a structure map given
by the source on $\C_1$, and we require the action to preserve structure maps in the sense
that the diagram
\[
\xymatrix{
X\times_{\C_0}\C_1\ar[rr]^-{\xi}\ar[dr]_-{S}&&X\ar[dl]^-{\e}
\\&\C_0}
\]
commutes; this is necessary in order to make sense of associativity.
The fibers in $X$ over the elements of $\C_0$ then form the target objects
of a functor from $\C^\op$ to $\Set$, and conversely, given such a functor, the disjoint
union of the target objects has a presheaf structure as specified above.  

The category theorists refer to the following concept as a ``discrete fibration.''  Since we will
also refer to its dual concept, we will use alternate terminology.

\begin{definition}
A functor $F:\C\to\C'$ is a \emph{target cover} if, given $f':a'\to b'$ in $\C'_1$ and 
$b\in\C_0$ such that
$F(b)=b'$,
there is a unique $f\in\C_1$ such that $F(f)=f'$ and $T(f)=b$, where $T$ refers to the
target function.
\end{definition}

An alternate way of describing this is to say that the following square is a pullback:
\[
\xymatrix{
C_1\ar[r]^-{T}\ar[d]_-{F_1}&\C_0\ar[d]^-{F_0}
\\\C'_1\ar[r]_-{T}&\C'_0.}
\]
Replacing $T$ with $S$ (the source map), we get the dual notion of a \emph{source cover}.
If a functor is both a source and target cover, we will say it is simply a \emph{cover}.

The following theorem is our primary structural tool.

\begin{theorem}\label{preservescovers}
Let $\D$ be the monad associated to a $\Sigma$-free $\Cat$-operad.  Then
$\D$ preserves both target covers and source covers (and consequently covers.)
\end{theorem}

It will be surprisingly useful to apply this theorem to functors between discrete
categories, which are all covers.  However, our most fundamental use is
to prove the following theorem, whose following corollary is basic to  our
definition of generalized multicategory.  Again, $\D$ is the monad associated
to a $\Sigma$-free $\Cat$-operad.

\begin{theorem}\label{Dpresheaf}
Let $X$ have the structure of presheaf over a category $\C$ with structure map
$\e:X\to\C_0$.  
Then $\D_0X$ 
naturally supports the structure of presheaf over $\D\C$
with structure map $\D_0\e:\D_0X\to\D_0\C_0$.
\end{theorem}

\begin{corollary}\label{D_0presheaf}
If $X$ supports the structure of presheaf over $\D^k(*)$,
then $\D_0X$ supports the structure of presheaf over $\D^{k+1}(*)$.
\end{corollary}

\section{The Definition of $\D$-multicategories}

We are now in a position to define generalized multicategories associated 
to a $\Sigma$-free operad $\DD$, which we call $\D$-multicategories after
the associated monad $\D$.  
To define a $\D$-multicategory $M$, 
we require the following data:
\begin{enumerate}
\item
A set $M_0$ of objects.
\item
A set $M_1$ of morphisms, which must come equipped with a specified presheaf
structure over $\D(*)$.
\item
A target map $T:M_1\to M_0$, which must be a presheaf map over the 
terminal functor $\e:\D(*)\to*$.
\item
A source map $S:M_1\to\D_0M_0$, which must be a map of presheaves over $\D(*)$.  
The presheaf structure on $\D_0M_0$ is given by Corollary \ref{D_0presheaf}.
\item
A unit map $I:M_0\to M_1$ which must be a presheaf map over the monadic unit map
$*\to\D(*)$.
\item
A composition map, to be described in detail below.
\end{enumerate}

It may be helpful at this point to explain how this works in the case of the categorical
Barratt-Eccles operad, whose algebras are the permutative categories.  Each permutative
category has an underlying symmetric multicategory, which this formalism is designed
to capture.  The component categories of the operad are the categories $E\Sigma_n$, whose
objects are the elements of $\Sigma_n$ and with exactly one morphism for each choice
of source and target.  The category $\D(*)$ then has objects the natural numbers (including 0),
and the only morphisms are elements of automorphism groups, which are 
isomorphic to $\Sigma_n$ for each $n$.  
The presheaf structure on the morphisms of a symmetric multicategory then consists of actions
of the $\Sigma_n$ on the morphisms with sources lists of length $n$, with the action
permuting the order in which the source objects are listed.  Composition involves
concatenation of lists.  The monad $\D_0$ associates to a set the set of lists of elements
of the set; that is, it is the free monoid monad.  It is now clear how $\D_0X$ is a presheaf
over $\D(*)$ in this instance.

The only one of our general data still to be described explicitly is the composition map.  We define
the \emph{composables} $M_2$ as the pullback in the following diagram:
\[
\xymatrix{
M_2\ar[r]^-{T}\ar[d]_-{S}&M_1\ar[d]^-{S}
\\\D_0M_1\ar[r]_-{\D_0T}&\D_0M_0,}
\]
so we can also write $M_2\cong M_1\times_{\D_0M_0}\D_0M_1$, with the maps
$T$ and $S$ in the diagram given by projection to the first and second factors, respectively.
Since Corollary \ref{D_0presheaf} tells us that $\D_0M_1$ is a presheaf over $\D^2(*)$, 
$M_2$ inherits that structure, and $T:M_2\to M_1$,
like $\D_0T:\D_0M_1\to\D_0M_0$,
is a map of presheaves over 
$\D\e:\D^2(*)\to\D(*)$.  The composition map we require as part of the data of the
multicategory $M$ is a map
\[
\gamma:M_2\to M_1
\]
that is a map of presheaves over the monad multiplication $\mu:\D^2(*)\to\D(*)$.  This 
completes the specification of the data for a $\D$-multicategory $M$.

Of course, we also require properties for these data.  For the data other than the 
composition map, these amount to the commutativity of the diagram
\[
\xymatrix{
&M_0\ar[dl]_-{\eta}\ar[d]_-{I}\ar[dr]^-{=}
\\\D_0M_0&M_1\ar[l]^-{S}\ar[r]_-{T}&M_0,}
\]
where $\eta$ is the unit for the monad $\D_0$.

The remaining properties involve the composition map $\gamma$.  First, we require
that it preserve sources and targets, in the sense that the two diagrams
\[
\xymatrix{
M_2\ar[r]^-{T}\ar[d]_-{\gamma}&M_1\ar[d]^-{T}
\\M_1\ar[r]_-{T}&M_0}
\quad\text{and}\quad
\xymatrix{ 
M_2\ar[r]^-{S}\ar[d]_-{\gamma}&\D_0M_1\ar[r]^-{\D_0S}&\D_0^2M_0\ar[d]^-{\mu}
\\M_1\ar[rr]_-{S}&&\D_0M_0}
\]
must commute.  Note that the second diagram captures the idea that
sources of composites in a multicategory are given by concatenation of sources:
this is the role of the monad multiplication map $\mu$ in the diagram.

Next, we require that $\gamma$ be unital, and we 
can write this most easily by using the expression $M_1\times_{\D_0M_0}\D_0M_1$
for $M_2$.  The unit conditions now can be expressed by requiring that
the diagram 
\[
\xymatrix{
M_1
\ar[r]^-{\cong}
\ar[ddrr]_-{=}
&M_0\times_{M_0}M_1
\ar[r]^-{I\times_\eta\eta}
&M_1\times_{\D_0M_0}\D_0M_1
\ar[dd]^-{\gamma}
&M_1\times_{\D_0M_0}\D_0M_0
\ar[l]_-{1\times\D_0I}
&M_1
\ar[l]_-{\cong}
\ar[ddll]^-{=}
\\\relax
\\&&M_1}
\]
commute.

It remains to specify the meaning of associativity for $\gamma$.  To do so, we first introduce
the pullback
\[
\xymatrix{
M_3\ar[r]^-{T}\ar[d]_-{S}&M_2\ar[d]^-{S}
\\\D_0M_2\ar[r]_-{\D_0T}&\D_0M_1,}
\]
where we think of $M_3$ as giving the associativity data for $M$.  It is
a presheaf over $\D^3(*)$ by construction, $S:M_3\to\D_0M_2$ is a map
of presheaves over $\D^3(*)$, and $T:M_3\to M_2$ is a map of
presheaves over the functor $\D^2\e:\D^3(*)\to\D^2(*)$.  We have two
induced composition maps $M_3\to M_2$ which we think of as composing
in either the first two (target) slots, or the last two (source) slots.  For composition
in the target slots, the commutative diagram
\[
\xymatrix{
M_3\ar[r]^-{T}\ar[d]_-{S}&M_2\ar[r]^-{\gamma}\ar[d]^-{S}&M_1\ar[ddd]^-{S}
\\\D_0M_2\ar[r]^-{\D_0T}\ar[d]_-{\D_0S}&\D_0M_1\ar[d]^-{\D_0S}
\\\D_0^2M_1\ar[r]_-{\D_0^2T}\ar[d]_-{\mu}&\D_0^2M_0\ar[dr]^-{\mu}
\\\D_0M_1\ar[rr]_-{\D_0T}&&\D_0M_0}
\]
gives us the induced map $\gamma_T:M_3\to M_2$, which the left vertical maps
in the diagram tell us is a map of presheaves over the functor $\mu:\D^3(*)\to\D^2(*)$. 
Next, the commutative diagram 
\[
\xymatrix{
M_3\ar[r]^-{T}\ar[d]_-{S}&M_2\ar[r]^-{T}\ar[d]^-{S}&M_1\ar[dd]^-{S}
\\\D_0M_2\ar[r]_-{\D_0T}\ar[d]_-{\D_0\gamma}&\D_0M_1\ar[dr]^-{\D_0T}
\\\D_0M_1\ar[rr]_-{\D_0T}&&\D_0M_0}
\]
gives us the induced map $\gamma_S:M_3\to M_2$, in which we compose
in the last two slots;  the left vertical maps tell us that it is a map of presheaves
over the functor $\D\mu:\D^3(*)\to\D^2(*)$.
Our associativity condition is now the requirement that
the diagram
\[
\xymatrix{
M_3\ar[r]^-{\gamma_T}\ar[d]_-{\gamma_S}&M_2\ar[d]^-{\gamma}
\\M_2\ar[r]_-{\gamma}&M_1}
\]
must commute.  This completes the definition of a $\D$-multicategory.

\section{The Underlying $\D$-multicategory of a $\D$-algebra}\label{underlying}

In this section we give the construction of the underlying $\D$-multicategory
of a $\D$-algebra.  We defer the proofs that it actually satisfies the necessary
properties.  

Suppose given a $\D$-algebra $\C$, so $\C$ is a category together with
an action map $\xi:\D\C\to\C$.  Since $\D$ determines (and is determined by)
the category object $\{\D_0,\D_1\}$ in monads on $\Set$, this is equivalent to 
having a $\D_0$-algebra structure $\xi_0:\D_0\C_0\to\C_0$ and a
$\D_1$-algebra structure $\xi_1:\D_1\C_1\to\C_1$ which determine 
a functor.  We use this structure to define an underlying $\D$-multicategory $U\C$.

For the objects $(U\C)_0$, we just use the objects $\C_0$ of $\C$.  For the
morphisms, we exploit the fact that we have the action map $\xi_0:\D_0\C_0\to\C_0$
to define $(U\C)_1$ by means of the following pullback square,
which also defines a comparison map $\kappa_1:(U\C)_1\to\C_1$:
\[
\xymatrix{
(U\C)_1\ar[r]^-{\kappa_1}\ar[d]_-{S}&\C_1\ar[d]^-{S}
\\\D_0\C_0\ar[r]_-{\xi_0}&\C_0.}
\]
For the target map $T:(U\C)_1\to(U\C)_0=\C_0$, we compose $\kappa_1$
with the target map $T:\C_1\to\C_0$ in $\C$.  In the case of permutative categories,
this definition simply says that a morphism in the underlying multicategory 
consists of a morphism in the category together with a specified decomposition
of the source as a direct sum.

We must provide a presheaf structure on $(U\C)_1$ over $\D(*)$,
and it is not merely the pullback of the presheaf structure on $\D_0M_0$,
since that would not preserve the comparison map $\kappa_1$.
We exploit
the following basic connection between presheaves and target covers,
whose proof appears in Section \ref{presheafproofs}.

\begin{theorem}\label{equivalentpresheaves}
Let $F:\C\to\C'$ be a target cover.  Then a presheaf structure on $X$ over $\C$
with structure map $\e:X\to\C_0$ is equivalent to a presheaf structure on $X$
over $\C'$ with structure map $\e':X\to\C'_0$ together with an explicit factorization
$\e'=F_0\circ\e$.
\end{theorem}

The presheaf structure on $U\C_1$ can now be described using the characterization
of presheaves from Theorem \ref{equivalentpresheaves}.  
We start with the discrete category $\C_0^\delta$ generated by the set $\C_0$,
and observe that the terminal map $\e:\C_0^\delta\to*$ is a cover.  Consequently,
by Theorem \ref{preservescovers}, so is the functor $\D\e:\D(\C_0^\delta)\to\D(*)$.
In particular, a presheaf
structure on $U\C_1$ over $\D(*)$ is equivalent to a presheaf structure over $\D(\C_0^\delta)$,
which we describe as follows.  
Note that the objects of $\D(\C_0^\delta)$ are $\D_0\C_0$ and the
morphisms are $\D_1\C_0$.
Since $U\C_1:=\C_1\times_{\C_0}\D_0\C_0$, we
have
\[
U\C_1\times_{\D_0\C_0}\D_1\C_0\cong\C_1\times_{\C_0}\D_1\C_0.
\]
In the following composite defining the presheaf action map, we exploit the
fact that the monad action is a functor; in particular, the diagram
\[
\xymatrix{
\D_1\C_1\ar[r]^-{S_\D S}\ar[d]_-{\xi_1}
&\D_0\C_0\ar[d]^-{\xi_0}
\\\C_1\ar[r]_-{S}&\C_0}
\]
induces  a map
\[
\xymatrix@C+10pt{
\D_1\C_1\ar[r]^-{(\xi_1,S_D S)}&\C_1\times_{\C_0}\D_0\C_0}
\]
that preserves the structure map of target type to $\C_0$.  We now express
the presheaf action map on $U\C_1$ as the composite
\begin{gather*}
\xymatrix{
U\C_1\times_{\D_0\C_0}\D_1\C_0\cong\C_1\times_{C_0}\D_1\C_0
\ar[r]^-{1\times\D_1I}
&\C_1\times_{\C_0}\D_1\C_1}
\\\xymatrix{
\ar[rr]^-{1\times(\xi_1,S_\D S)}
&&\C_1\times_{\C_0}\C_1\times_{\C_0}\D_0\C_0
\ar[r]^-{\gamma_\C\times1}
&\C_1\times_{\C_0}\D_0\C_0\cong U\C_1.}
\end{gather*}

For the identity map $(U\C)_0=\C_0\to(U\C)_1$, we note that the diagram
\[
\xymatrix{
\C_0\ar[r]^-{I_\C}\ar[d]_-{\eta}&\C_1\ar[d]^-{S}
\\\D_0\C_0\ar[r]_-{\xi_0}&\C_0}
\]
commutes, since both composites coincide with the identity on $\C_0$.  We get
an induced map $I_{U\C}:\C_0\to\C_1\times_{\C_0}\D_0\C_0\cong(U\C)_1$.

Before constructing the composition map $\gamma_{U\C}$, we pause to introduce
the comparison maps $\kappa_n:(U\C)_n\to\C_n$; we will need 
the case $n=2$ for the definition of $\gamma_{U\C}$ and the case $n=3$
for the verification of associativity.  Note first that since
$\{\D_0,\D_1\}$ is a category object in monads on $\Set$, we have a composite identity
natural transformation $I^n_\D:\D_0\to\D_n$ for each $n$.  
We assume for induction that the diagrams
\[
\xymatrix{
U\C_n\ar[r]^-{T}\ar[d]_-{\kappa_n}
&U\C_{n-1}\ar[d]^-{\kappa_{n-1}}
\\\C_n\ar[r]_-{T}&\C_{n-1}}
\quad\text{and}\quad
\xymatrix{
U\C_n\ar[rrr]^-{\kappa_n}\ar[d]_-{S}
&&&\C_n\ar[d]^-{S}
\\\D_0U\C_{n-1}\ar[r]_-{\D_0\kappa_{n-1}}
&\D_0\C_{n-1}\ar[r]_-{I_\D^{n-1}}
&\D_{n-1}\C_{n-1}\ar[r]_-{\xi_{n-1}}&\C_{n-1}}
\]
both commute; this is true for $n=1$ with the convention that
$\kappa_0=\id_{\C_0}$.
Now assuming the
previous $\kappa_j$ have been defined for $j\le n$, and supposing $n\ge1$, we
define $\kappa_{n+1}$ by means of the following commutative diagram, which
induces a map to the pullback $\C_{n+1}=\C_{n}\times_{\C_{n-1}}\C_{n}$:
\[
\xymatrix{
U\C_{n+1}\ar[r]^-{T}\ar[d]_-{S}
&U\C_{n}\ar[r]^-{\kappa_{n}}\ar[d]^-{S}
&\C_{n}\ar[dddd]^-{S}
\\\D_0U\C_{n}\ar[r]^-{\D_0T}\ar[d]_-{\D_0\kappa_{n}}
&\D_0U\C_{n-1}\ar[d]^-{\D_0\kappa_{n-1}}
\\\D_0\C_{n}\ar[r]^-{\D_0T}\ar[d]_-{I_\D^{n}}
&\D_0\C_{n-1}\ar[d]^-{I_\D^{n-1}}
\\\D_{n}\C_{n}\ar[r]_-{T^2}\ar[d]_-{\xi_{n}}
&\D_{n-1}\C_{n-1}\ar[dr]^-{\xi_{n-1}}
\\\C_{n}\ar[rr]_-{T}&&\C_{n-1}.}
\]
The right part of the diagram commutes by induction, and the left stack from top to bottom
as follows: the top square defines $U\C_{n+1}$, the next one
combines an inductive hypothesis
with the fact that $\D_0$ is Cartesian, the one below
it follows from targets of identities being the object, and the bottom because 
$\xi:\D\C\to\C$ is a functor.  It is straightforward to see that the inductive hypotheses are
preserved.

Now to define the composition map $\gamma_{U\C}:(U\C)_2\to(U\C)_1$, 
we use the following commutative diagram, which defines a map from $(U\C)_2$
to $\C_1\times_{\C_0}\D_0\C_0=(U\C)_1$: 
\[
\xymatrix{
(U\C)_2\ar[rrr]^-{\kappa_2}\ar[d]_-{S}&&&\C_2\ar[r]^-{\gamma_\C}\ar[d]^-{S}&\C_1\ar[ddd]^-{S}
\\\D_0(U\C)_1\ar[r]^-{\D_0\kappa_1}\ar[d]_-{\D_0S}&\D_0\C_1\ar[r]^-{I_\D}\ar[d]^-{\D_0S}
&\D_1\C_1\ar[r]^-{\xi_1}\ar[d]^-{S^2}&\C_1\ar[ddr]^-{S}
\\\D_0^2\C_0\ar[r]_-{\D_0\xi_0}\ar[d]_-{\mu}&\D_0\C_0\ar[r]_-{=}&\D_0\C_0\ar[drr]^-{\xi_0}
\\\D_0\C_0\ar[rrrr]_-{\xi_0}&&&&\C_0.}
\]
This completes the specification of the data for the underlying $\D$-multicategory $U\C$.

\section{The provisional left adjoint construction}

In this section we introduce the construction of a $\D$-algebra $\Lhat M$ associated to 
a given $\D$-multicategory $M$ that is almost, but not quite, a left adjoint to the
forgetful functor.  The construction is functorial, and there are unit and counit
maps satisfying the triangle identities.  The only flaw is that the unit map fails to 
preserve the presheaf structure.  Our actual left adjoint, $LM$, will be constructed
as a quotient of $\Lhat M$, and will inherit all the left adjoint properties as well 
as satisfying the requirement that the unit actually be a map of $\D$-multicategories.  

The idea of $\Lhat M$ in the case of ordinary (symmetric) multicategories and
their associated free permutative categories is that the objects of $LM$ (and $\Lhat M$)
in this case are given by the free monoid on the objects of $M$:
\[
\Lhat M_0=LM_0:=\coprod_{n\ge0}M_0^n=\D_0M_0.
\]
We adopt the principle that for any $\D$ that we are considering, we also want 
$LM_0:=\D_0M_0$.  The problem is specifying the morphisms.  In the case of 
classical symmetric multicategories, a morphism from one list $\br{a_1,\dots,a_m}$ to
another $\br{b_1,\dots,b_n}$ is specified by a function 
$f:\{1,\dots,m\}\to\{1,\dots,n\}$ together with an $n$-tuple of morphisms $\br{\phi_i}_{i=1}^n$
of $M$, where 
\[
\phi_i:\br{a_j}_{f(j)=i}\to b_i.
\]
The problem is that the entries in the list $\br{a_j}_{f(j)=i}$ are given in 
natural  number order, and there is no way to capture this in our formalism.  
Another way to look at the problem is that we can certainly capture the idea of
a list of morphisms as an element of $\D_0M_1$, but we also need to permute
the elements of the source in a compatible way; what the function $f$ does is permute in such a way
that the inputs to the individual morphisms in the list have their order preserved.
Instead, we will allow the elements of the source to be permuted arbitrarily
as part of the data for a morphism, which means that each morphism in the actual left adjoint will
be represented in many different ways.  We will then identify the ones that do represent
the same morphism by means of a coequalizer construction.  The construction of $\Lhat M$,
however, simply allows the entries to be permuted arbitrarily without identifying morphisms, and this is the essential
difference between $\Lhat M$ and $LM$.  We turn now to the actual construction of
$\Lhat M$.

We have already specified the objects $\Lhat M_0$ as $\D_0M_0$.  The morphisms
are given as the pullback in the following diagram:
\[
\xymatrix{
\Lhat M_1\ar[r]^-{\That}\ar[dd]_-{\Shat}&\D_0M_1\ar[d]^-{S}
\\&\D_0^2M_0\ar[d]^-{\mu}
\\\D_1M_0\ar[r]_-{T_\D}&\D_0M_0,}
\]
so we can write
\[
\Lhat M_1\cong\D_0M_1\times_{\D_0M_0}\D_1M_0.
\]
The intuition here is that the elements of $\D_0M_1$ correspond to lists of morphisms,
and the elements of $\D_1M_0$ attach permutations to them in a compatible way.  This
may be clearer if we realize that we have a pullback diagram
\[
\xymatrix{\D_1M_0\ar[r]^-{T_\D}\ar[d]_-{\D_1\e}&\D_0M_0\ar[d]^-{\D_0\e}
\\\D_1(*)\ar[r]_-{T_\D}&\D_0(*)}
\]
that can be pasted onto the bottom of the given diagram, giving a pure 
permutation in the lower left corner. (The diagram
is a pullback because we can think of $\e$ as giving a cover functor from
$M_0^\delta$ to $*$, and then using the fact that $\D$ preserves covers.)  
However, the actual diagram given is
much more useful for specifying the source of a morphism.  In particular,
we may extend the diagram as follows,
\[
\xymatrix{
\Lhat M_1\ar[r]^-{\That}\ar[dd]_-{\Shat}&\D_0M_1\ar[d]^-{S}\ar[r]^-{\D_0T}&\D_0M_0
\\&\D_0^2M_0\ar[d]^-{\mu}
\\\D_1M_0\ar[r]_-{T_\D}\ar[d]_-{S_\D}&\D_0M_0
\\\D_0M_0,}
\]
in which the top horizontal composite defines the target map for $\Lhat M$
and the left vertical composite defines the source map.

We specify an identity map $I:\D_0M_0\to\Lhat M_1$ by means of the pullback
construction of $\Lhat M_1$: examining the diagram
\[
\xymatrix{\D_0M_0\ar[r]^-{\D_0I}\ar[dd]_-{I_\D}\ar[dr]_-{\D_0\eta}
&\D_0M_1\ar[d]^-{\D_0S}
\\&\D_0^2M_0\ar[d]^-{\mu}
\\\D_1M_0\ar[r]_-{T_\D}&\D_0M_0,}
\]
we see that both composites coincide with the identity on $\D_0M_0$, 
so the diagram commutes and we get a map to $\Lhat M_1$.  

To construct the composition, we start again with
\[
\Lhat M_1\cong\D_0M_1\times_{\D_0M_0}\D_1M_0,
\]
from the pullback diagram defining $\Lhat M_1$.  We then have 
\[
\Lhat M_2:=\Lhat M_1\times_{\D_0M_0}\Lhat M_1
\cong\D_0M_1\times_{\D_0M_0}\D_1M_0\times_{\D_0M_0}\D_0M_1\times_{\D_0M_0}\D_1M_0.
\]
The first step in the composition is the identification of the two terms
in the middle as $\D_1M_1$; this is a consequence of the pullback diagram
\[
\xymatrix{
\D_1M_1\ar[r]^-{\D_1T}\ar[d]_-{S_\D}&\D_1M_0\ar[d]^-{S_\D}
\\\D_0M_1\ar[r]_-{\D_0T}&\D_0M_0.}
\]
This diagram is a pullback because we can think of $T:M_1\to M_0$
as a functor $\M_1^\delta\to M_0^\delta$, which is a cover since the 
categories are discrete, and then use the fact that $\D$ preserves covers. 
We get an isomorphism 
\[
\xymatrix@C+10pt{
\D_1M_0\times_{\D_0M_0}\D_0M_1
&\D_1M_1\ar[l]_-{(\D_1T,S_\D)}^-{\cong}}
\]
that preserves the structure maps to $\D_0M_0$ of both source and target type.
Next, we exploit the commutative diagram
\[
\xymatrix{
\D_1M_1\ar[r]^-{T_\D}\ar[d]_-{\D_1S}&\D_0M_1\ar[dd]^-{\D_0S}
\\\D_1\D_0M_0\ar[d]_-{\D_1I_\D}\ar[dr]^-{T_\D}
\\\D_1^2M_0\ar[r]_-{T_\D^2}\ar[d]_-{\mu}&\D_0^2M_0\ar[d]^-{\mu}
\\\D_1M_0\ar[r]_-{T_\D}&\D_0M_0.}
\]
We abbreviate the left vertical composite as $\theta$, and we obtain
a map $(T_\D,\theta):\D_1M_1\to\D_0M_1\times_{\D_0M_0}\D_1M_0$.
This map also preserves the structure maps of both source and target type; 
the preservation of target type structure maps is a consequence of 
the commutative diagram
\[
\xymatrix{
\D_1M_1\ar[r]^-{T_\D}\ar[d]_-{\D_1T}&\D_0M_1\ar[d]^-{\D_0T}
\\\D_1M_0\ar[r]_-{T_\D}&\D_0M_0,}
\]
and preservation of source type structure maps follows from the 
commutative diagram
\[
\xymatrix{
\D_1M_1\ar[r]^-{S_\D}\ar[d]_-{\D_1S}&\D_0M_1\ar[dd]^-{\D_0S}
\\\D_1\D_0M_0\ar[d]_-{\D_1I_\D}\ar[dr]^-{S_\D}
\\\D_1^2M_0\ar[r]_-{S_\D^2}\ar[d]_-{\mu}&\D_0^2M_0\ar[d]^-{\mu}
\\\D_1M_0\ar[r]_-{S_\D}&\D_0M_0.}
\]
THe net result of these two steps is to produce what we think of as
an interchange map $\chi:\D_1M_0\times_{\D_0M_0}\D_0M_1
\to\D_0M_1\times_{\D_0M_0}\D_1M_0$ defined by the zig-zag
involving an isomorphism as follows:
\[
\xymatrix@C+10pt{
\D_1M_0\times_{\D_0M_0}\D_0M_1
&\D_1M_1\ar[l]_-{(\D_1T,S_\D)}^-{\cong}\ar[r]^-{(T_\D,\theta)}
&\D_0M_1\times_{\D_0M_0}\D_1M_0.}
\]

We will also need the fact that both squares in the diagram
\[
\xymatrix{
\D_0M_2\ar[r]^-{\D_0T}\ar[d]_-{\D_0S}&\D_0M_1\ar[d]^-{\D_0S}
\\\D_0^2M_1\ar[r]^-{\D_0^2T}\ar[d]_-{\mu}&\D_0^2M_0\ar[d]^-{\mu}
\\\D_0M_1\ar[r]_-{\D_0T}&\D_0M_0}
\]
are pullbacks, since $\D_0$ is Cartesian, so 
\[
D_0M_1\times_{\D_0M_0}\D_0M_1\cong\D_0M_2.
\]
We also need the fact that the diagram
\[
\xymatrix{
\D_2M_0\ar[r]^-{T_\D}\ar[d]_-{S_\D}&\D_1M_0\ar[d]^-{S_\D}
\\\D_1M_0\ar[r]_-{T_\D}&\D_0M_0}
\]
is a pullback, since it just expresses the composables in $\D(M_0^\delta)$,
so we have
\[
\D_1M_0\times_{\D_0M_0}\D_1M_0\cong\D_2M_0.
\]
We can now express the composition in $\Lhat M$ as the composite map
\begin{gather*}
\Lhat M_2
\cong\D_0M_1\times_{\D_0M_0}\D_1M_0\times_{\D_0M_0}\D_0M_1\times_{\D_0M_0}\D_1M_0
\\
\xymatrix{\relax
&&\D_0M_1\times_{\D_0M_0}\D_1M_1\times_{\D_0M_0}\D_1M_0\ar[ll]_-{1\times(\D_1T,S_\D)\times1}^-{\cong}}
\\
\xymatrix{
\ar[rr]^-{1\times(T_\D,\theta)\times1}&&\D_0M_1\times_{\D_0M_0}\D_0M_1\times_{\D_0M_0}\D_1M_0\times_{\D_0M_0}\D_1M_0}
\\
\xymatrix{
\cong\D_0M_2\times_{\D_0M_0}\D_2M_0
\ar[rr]^-{\D_0\gamma_M\times\gamma_\D}&&\D_0M_1\times_{\D_0M_0}\D_1M_0\cong\Lhat M_1.}
\end{gather*}
Note that the middle two arrows can be also written as $1\times\chi\times1$.
It is not supposed to be obvious that this captures the idea of composition
in the intuitive description of $\Lhat M$ given above; however, it is the most
convenient description for verification of its formal properties.

We now claim the following as the major result about $\Lhat M$:

\begin{theorem}
The category $\Lhat M$ supports the structure of a $\D$-algebra.  There is a natural
map of $\D$-algebras $\e:\Lhat U\C\to\C$ for any $\D$-algebra $\C$, and there is
a map $\eta: M\to U\Lhat M$ for any $\D$-multicategory $M$ that satisfies 
all the properties of a map of $\D$-multicategories except preservation of 
the $\D(*)$-presheaf structure on morphisms.  Both adjunction triangles commute.
\end{theorem}

We defer the proof, as with those of all previous claims. 

\section{The actual left adjoint}
The actual left adjoint, which we denote by $LM$, is a quotient of $\Lhat M$.
In particular, we define $LM_0:=\Lhat M_0=\D_0M_0$, and we define
$LM_1$ by means of a coequalizer diagram with target $\Lhat M_1$, displayed
schematically as follows:
\[
\D_0M_1\times_{\D_0^2M_0}\D_0\D_1M_0\times_{\D_0M_0}\D_1M_0
\rightrightarrows\D_0M_1\times_{\D_0M_0}\D_1M_0,
\]
where the target expresses $\Lhat M_1$.  We then define $LM_1$ as the
coequalizer of the diagram.  The intuition 
in the case of classical symmetric multicategories
is that the middle term 
in the source of the two arrows
displays tuples
of permutations that can be attached either to the tuples of morphisms in 
the front, or the total permutation in the back, and the result should be the
same either way they are attached.  The two arrows record these two ways
of attaching the lists of permutations.

To describe the first arrow, in which we attach the list of permutations to the
list of morphisms, we 
give a
presheaf structure on $M_1$ over $\D(M_0^\delta)$ equivalent
to the one assumed over $\D(*)$ by
exploiting Theorem \ref{equivalentpresheaves} and the fact
that $S:M_1\to\D_0M_0$ is a map of presheaves over $\D(*)$.  Since $S$ is
a map of presheaves, we have a commutative triangle
\[
\xymatrix{
M_1\ar[rr]^-{S}\ar[dr]_-{\e}&&\D_0M_0\ar[dl]^-{\D_0\e}
\\&\D_0(*).}
\]
But $\D_0\e:\D_0M_0\to\D_0(*)$ is the map on objects of the target cover
$\D\e:\D(M_0^\delta)\to\D(*)$, so Theorem \ref{equivalentpresheaves} tells us
that the presheaf structure on $M_1$ over $\D(*)$ is equivalent to one over
$\D(M_0^\delta)$ with structure map $S$.  In particular, we get a presheaf
action map
\[
\xymatrix{
M_1\times_{\D_0M_0}\D_1M_0\ar[r]^-{\psi}&M_1.}
\]
This gives us a composite
\[
\xymatrix{
\D_0M_1\times_{\D_0^2M_0}\D_0\D_1M_0
\cong\D_0(M_1\times_{\D_0M_0}\D_1M_0)
\ar[r]^-{\D_0\psi}&\D_0M_1}
\]
that induces the first map to be coequalized.  

The second arrow is easier to describe: it is 
induced by
the composite
\begin{gather*}
\xymatrix{
\D_0\D_1M_0\times_{\D_0M_0}\D_1M_0\ar[r]^-{I_\D\times1}
&\D_1^2M_0\times_{\D_0M_0}\D_1M_0}
\\\xymatrix{
\relax\ar[r]^-{\mu\times1}
&\D_1M_0\times_{\D_0M_0}\D_1M_0\ar[r]^-{\gamma_\D}
&\D_1M_0.}
\end{gather*}

We now define $LM_1$ to be the coequalizer of the two arrows
described, and our main theorem is as follows.

\begin{theorem}\label{maintheorem}
The coequalizer $LM_1$ gives the morphisms for a $\D$-multicategory
$LM$, and the construction $L$ is left adjoint to the underlying
$\D$-algebra construction $U$.
\end{theorem}

\part{Proofs}
\section{$\Sigma$-free Categorical Operads and Their Associated Monads}
The proofs of our claims start in this section.
We begin by explaining the consequences of the $\Sigma$-free 
assumption on a categorical operad that we will use.  The $\Sigma$-free
hypothesis is always used in the context of the following fundamental proposition,
which is no doubt well-known to the experts.

\begin{proposition}\label{orbitpullbacks}
Let $G$ be a discrete group.  Then passage to orbits from the category of $G$-sets
to the category of sets sends pullbacks of free $G$-sets to pullbacks of sets.
\end{proposition}

\begin{proof}
Suppose given a pullback diagram of free $G$-sets
\[
\xymatrix{
A\ar[r]^-{f}\ar[d]_-{j}&B\ar[d]^-{h}
\\C\ar[r]_-{k}&D.}
\]
We wish to show that the diagram of orbits
\[
\xymatrix{
A/G\ar[r]^-{f/G}\ar[d]_-{j/G}&B/G\ar[d]^-{h/G}
\\C/G\ar[r]_-{k/G}&D/G}
\]
is also a pullback.  So suppose given orbit classes $[c]\in C/G$ and $[b]\in B/G$
such that $[k(c)]=[h(b)]$ in $D/G$;  
we wish to show that there is a unique orbit class $[a]\in A/G$ such that $[f(a)]=[b]$
and $[j(a)]=[c]$.
Since $D$ is a free $G$-set, there is a unique
$g\in G$ such that $k(c)=g\cdot h(b)=h(g\cdot b)$.  Since the first square is a pullback,
there is a unique $a\in A$ such that $f(a)=g\cdot b$ and $j(a)=c$, and consequently
$[f(a)]=[b]$ and $[j(a)]=[c]$.  This establishes existence.  For uniqueness, suppose there
is also $a'\in A$ such that $[f(a')]=[b]$ and $[j(a')]=[c]$.  We need to show that $[a]=[a']$.
Since $[f(a)]=[f(a')]$, there is a unique $g_1\in G$ such that $f(a)=g_1\cdot f(a')
=f(g_1\cdot a')$, and since $[j(a)]=[j(a')]$, there is a unique $g_2\in G$ such that
$j(a)=g_2\cdot j(a')=j(g_2\cdot a')$.  Now we have
\[
g_1\cdot kj(a')=g_1\cdot hf(a') =hf(g_1\cdot a')=hf(a)=kj(a)=kj(g_2\cdot a')=g_2\cdot kj(a').
\]
Since the action of $G$ on $D$ is free, it follows that $g_1=g_2$.  We now have
\[
f(a)=f(g_1\cdot a')\text{ and }j(a)=j(g_2\cdot a')=j(g_1\cdot a'),
\]
so since $a$ and $g_1\cdot a'$ have the same images under both $f$ and $j$, we see that
$a=g_1\cdot a'$.  Therefore $[a]=[a']$, and uniqueness is established.
\end{proof}

This allows us to prove Lemma \ref{category object}, which says the set-monads $\{\D_0,\D_1\}$
give a category object in monads on $\Set$ that generates the remaining monads
$\D_2,\D_3,\dots$.
\begin{proof}
Since $\DD$ is an operad of categories, the operad structure maps commute with
the category structure maps, so we can just as well consider $\DD$ a category
object in operads (in $\Set$.)  Since a category object is defined by diagrams 
and a pullback condition defining the composables, any functor that preserves
pullbacks also preserves category objects.  The passage from the operad $\DD$ to
the monad $\D$ involves products, which preserve pullbacks, and orbits.  However,
we are assuming the groups $\Sigma_n$ act freely on $\DD_n$, and therefore the orbit 
functor is being taken only on free $\Sigma_n$ sets, and orbits of free actions
do preserve pullbacks, from Proposition \ref{orbitpullbacks}.  Consequently
$\{\D_0,\D_1\}$ does give a category object in monads on $\Set$.
Since the higher monads $\D_2,\D_3,\dots$ are also produced by pullbacks,
they are the further components of the nerve of this category object.
\end{proof}

We turn next to showing that all these monads are Cartesian.  Since all the operads
$\DD_n$ are $\Sigma$-free, this follows from Theorem \ref{cartesianmonad}, whose
proof is as follows.

\begin{proof}
Let's call the operad $\JJ$ and its associated monad $\J$; we must first show
that $\J$ preserves pullbacks.  Certainly the assignment 
\[
X\mapsto \JJ_n\times X^n
\]
preserves pullbacks, since it is a left adjoint.  Since $\JJ$ is $\Sigma$-free,
Proposition \ref{orbitpullbacks} now shows that 
\[
X\mapsto\JJ_n\times_{\Sigma_n}X^n
\]
also preserves pullbacks.  Finally, coproducts preserve pullbacks in $\Set$,
so 
\[
X\mapsto\coprod_{n\ge0}\JJ_n\times_{\Sigma_n}X^n
\]
also preserves pullbacks.  But this is the definition of $\J X$, so $\J$
preserves pullbacks.

Next, we observe that the naturality diagrams for the unit are all pullbacks.  For a
given $f:X\to Y$, we get the naturality diagram
\[
\xymatrix{
X\ar[r]^-{f}\ar[d]_-{\eta}&Y\ar[d]^-{\eta}
\\\coprod_{n\ge0}\JJ_n\times_{\Sigma_n}X^n\ar[r]_-{\J f}
&\coprod_{n\ge0}\JJ_n\times_{\Sigma_n}Y^n,}
\]
which restricts on the images of the unit maps to
\[
\xymatrix{
X\ar[r]^-{f}\ar[d]_-{\cong}&Y\ar[d]^-{\cong}
\\X\ar[r]_-{f}&Y,}
\]
which is trivially a pullback.

To see that the naturality square for the product $\mu$ is a pullback, we examine the square
\[
\xymatrix{
\coprod_{k\ge0}\JJ_k\times_{\Sigma_k}\left(\coprod_{n\ge0}\JJ_n\times_{\Sigma_n}X^n\right)^k
\ar[r]^-{\J^2f}\ar[d]_-{\mu}
&\coprod_{k\ge0}\JJ_k\times_{\Sigma_k}\left(\coprod_{n\ge0}\JJ_n\times_{\Sigma_n}Y^n\right)^k
\ar[d]^-{\mu}
\\\coprod_{n\ge0}\JJ_n\times_{\Sigma_n}X^n
\ar[r]_-{\J f}
&\coprod_{n\ge0}\JJ_n\times_{\Sigma_n}Y^n.}
\]
Let's write a typical element in the upper right corner as $[a;b_1,y_1,\dots,b_k,y_k]$,
and a typical element in the lower left $[c,x]$, with the understanding
that $x$ and all the $y$'s are lists of elements.  We need to show that
if these two elements map to the same one in the lower right, then
there is a unique element in the upper left mapping to them.  So we
assume that the images are the same.  But $[a;b_1,y_1,\dots,b_k,y_k]$
maps to $[\gamma(a;b_1,\dots,b_k),y]$, where $y=(y_1,\dots,y_k)$,
while $[c,x]$ maps to $[c,fx]$, with the obvious interpretation of $fx$.
Since the operad is $\Sigma$-free, there is a unique element $\sigma\in\Sigma_n$
such that
\[
c=\gamma(a;b_1,\dots,b_k)\cdot\sigma,
\]
so
\[
[c,x]=[\gamma(a;b_1\dots,b_k),\sigma\cdot x]
\]
and
\[
[c,fx]=[\gamma(a;b_1,\dots,b_k),\sigma\cdot fx]=[\gamma(a;b_1,\dots,b_k),f(\sigma\cdot x)].
\]
Now if we partition the entries in $\sigma\cdot x$ according to the dimensions
of $b_1,\dots,b_k$ into $((\sigma\cdot x)_1,\dots,(\sigma\cdot x)_k)$, we
can construct the element
\[
[a;b_1,(\sigma\cdot x)_1,\dots,b_k,(\sigma\cdot x)_k]\in\J^2X
\]
which maps to each of the elements chosen originally.  This establishes
existence for an element in the upper left corner.

For uniqueness, suppose we have given two elements,
\[
[a;b_1,x_1,\dots,b_k,x_k]\text{ and }[a';b_1',x_1',\dots,b_k',x_k'],
\]
both mapping to the same pair of elements in the upper right and
lower left corners of the diagram.  Let's call the images in the upper right
\[
[a;b_1,y_1,\dots,b_k,y_k]\text{ and }[a';b_1',y_1',\dots,b_k',y_k'].
\]
Now by freeness, there is a unique $\sigma\in\Sigma_k$ such that
$a'=a\cdot\sigma$, from which we get
\[
[a';b_1',y_1',\dots,b_k',y_k']=[a;b_{\sigma^{-1}(1)}',y_{\sigma^{-1}(1)}',\dots,b_{\sigma^{-1}(k)}',y_{\sigma^{-1}(k)}'].
\]
Next, again by freeness, there is a unique $\tau_j\in\Sigma_{n_j}$ for each
$1\le j\le k$ such that $b'_{\sigma^{-1}(i)}=b_i\cdot\tau_i$, which can also 
be written $b'_i=b_{\sigma(i)}\cdot\tau_{\sigma(i)}$.  
Now using May's notation
$\sigma\br{j_1,\dots,j_k}$ from \cite{M} 
(to be precise, May uses $\sigma(j_1,\dots,j_k)$) for the permutation that permutes blocks
of size $j_1,\dots,j_k$ in the same way $\sigma$ permutes letters, we have
\begin{gather*}
\gamma(a';b_1',\dots,b_k')=\gamma(a\cdot\sigma;b_{\sigma(1)}\cdot\tau_{\sigma(1)},\dots,b_{\sigma(k)}\cdot\tau_{\sigma(k)})
\\=\gamma(a;b_1\cdot\tau_1,\dots,b_k\cdot\tau_k)\cdot\sigma\br{j_1,\dots,j_k}
\\=\gamma(a;b_1,\dots,b_k)\cdot(\tau_1\oplus\cdots\oplus\tau_k)\cdot\sigma\br{j_1,\dots,j_k}.
\end{gather*}
Mapping both elements to the lower left corner, we now see that
\begin{gather*}
[\gamma(a;b_1,\dots,b_k),x]=[\gamma(a',b_1',\dots,b_k'),x']
\\=[\gamma(a;b_1,\dots,b_k),(\tau_1\oplus\cdots\oplus\tau_k)\cdot\sigma\br{j_1,\dots,j_k}\cdot x'].
\end{gather*}
By freeness, this means that 
\[
x=(\tau_1\oplus\cdots\oplus\tau_k)\cdot\sigma\br{j_1,\dots,j_k}\cdot x'
=(\tau_1\cdot x'_{\sigma^{-1}(1)},\dots,\tau_k\cdot x'_{\sigma^{-1}(k)})
\]
where we have partitioned $x'$ into blocks of the correct size for $\sigma\br{j_1,\dots,j_k}$ to act.
It follows that for $1\le i\le k$,
\[
x_i=\tau_i\cdot x'_{\sigma^{-1}(i)}.
\]
Now we can compute:
\begin{gather*}
[a';b_1',x_1',\dots,b_k',x_k']=[a\cdot\sigma,b_1',x_1',\dots,b_k',x_k']
\\=[a;b'_{\sigma^{-1}(1)},x'_{\sigma^{-1}(1)},\dots,b'_{\sigma^{-1}(k)},x'_{\sigma^{-1}(k)}]
\\=[a;b_1\cdot\tau_1,x'_{\sigma^{-1}(1)},\dots,b_k\cdot\tau_k,x'_{\sigma^{-1}(k)}]
\\=[a;b_1,\tau_1\cdot x'_{\sigma^{-1}(1)},\dots,b_k,\tau_k\cdot x'_{\sigma^{-1}(k)}]
\\=[a,b_1,x_1,\dots,b_k,x_k].
\end{gather*}
This establishes uniqueness, and therefore $\J$ is a Cartesian monad.
\end{proof}

\section{Presheaves and Cover Functors}\label{presheafproofs}
This section is devoted to the proofs of all the statements we will need concerning
cover functors and their relation to presheaves.  Recall that a \emph{target cover}
is a functor $F:\C\to\C'$ for which the square
\[
\xymatrix{
\C_1\ar[r]^-{T}\ar[d]_-{F_1}&\C_0\ar[d]^-{F_0}
\\\C'_1\ar[r]_-{T}&\C'_0}
\]
is a pullback, and a \emph{source cover} is a functor in which the pullback condition
applies to the analogous square in which the target maps are replaced with source maps.
We have the following consequence for all the squares in the map of nerves.
The analogous statement holds for source covers, with the same proof.

\begin{corollary}\label{coverpullbacks} 
If $F:\C\to\C'$ is a target cover, then all the squares
\[
\xymatrix{
\C_n\ar[r]^-{T}\ar[d]_-{F_n}&C_{n-1}\ar[d]^-{F_{n-1}}
\\\C_n'\ar[r]_-{T}&\C_{n-1}'}
\]
are pullbacks for $n\ge1$.
\end{corollary}

\begin{proof}
The case $n=1$ is the definition of a target cover.  Suppose for induction that 
the conclusion holds at $n-1$.  In the following diagram, the top, central, and
bottom subsquares are then pullbacks:
\[
\xymatrix{
C_n\ar[rrr]^-{T}\ar[dr]^-{S}\ar[ddd]_-{F_n}
&&&C_{n-1}\ar[dl]_-{S}\ar[ddd]^-{F_{n-1}}
\\&\C_{n-1}\ar[r]^-{T}\ar[d]_-{F_{n-1}}
&C_{n-2}\ar[d]^-{F_{n-2}}
\\&\C_{n-1}'\ar[r]_-{T}
&C_{n-2}'
\\\C_n'\ar[rrr]_-{T}\ar[ur]^-{S}
&&&C_{n-1}'\ar[ul]_-{S}}
\]
It now follows that the outer square is a pullback, establishing the claim.
\end{proof}

Theorem \ref{equivalentpresheaves} gives the key connection between
presheaves and target covers:
presheaves over the two categories connected by a target cover are
essentially equivalent; all they need is to have their structure maps factor
through the map on objects.
The proof is as follows.
\ignore

The following is a key connection between presheaves and target covers.

\begin{theorem}\label{equivalentpresheaves}
Let $F:\C\to\C'$ be a target cover.  Then a presheaf structure on $X$ over $\C$
with structure map $\e:X\to\C_0$ is equivalent to a presheaf structure on $X$
over $\C'$ with structure map $\e':X\to\C'_0$ together with an explicit factorization
$\e'=F_0\circ\e$.
\end{theorem}
\endignore

\begin{proof}[Proof of Theorem \ref{equivalentpresheaves}]
Since $F$ is a target cover, we have an explicit isomorphism
\[
\xymatrix{
\C_0\times_{\C_0'}\C'_1&\C_1\ar[l]_-{(T,F_1)}^-{\cong}}.
\]
Now an explicit factorization $\e'=F_0\circ\e$ gives us an explicit isomorphism
\[
\xymatrix{
X\times_{C_0}\C_1\ar[r]^-{1\times_{F_0}F_1}_-{\cong}& X\times_{\C_0'}\C_1',}
\]
and the correspondence between presheaf
action maps is given by the requirement that 
\[
\xymatrix{
X\times_{\C_0}\C_1\ar[rr]_-{\cong}^-{1\times_{F_0}F_1}\ar[dr]_-{\xi}
&&X\times_{\C_0'}\C_1'\ar[dl]^-{\xi'}
\\&X}
\]
commute.  Equivalence of the unital conditions for the two actions are a consequence
of the commutative diagram
\[
\xymatrix{
&X\ar[dl]^-{\cong}_-{1\times_{\e_0}\e_0}\ar[dr]_-{\cong}^-{1\times_{\e_0'}\e_0'}
\\X\times_{\C_0}\C_0\ar[rr]^-{\cong}_-{1\times_{F_0}F_0}\ar[d]_-{1\times I}
&&X\times_{\C_0'}\C_0'\ar[d]^-{1\times I}
\\X\times_{\C_0}\C_1\ar[rr]_-{\cong}^-{1\times_{F_0}F_1}\ar[dr]_-{\xi}
&&X\times_{\C_0'}\C_1'\ar[dl]^-{\xi'}
\\&X,}
\]
and equivalence of associativity is a consequence of the commutative cube
\[
\xymatrix{
X\times_{\C_0}\C_1\times_{\C_0}\C_1
\ar[rr]^-{1\times\mu}\ar[dr]_-{\cong}^-{1\times_{F_0} F_2}\ar[dd]_-{\xi\times1}
&&X\times_{\C_0}\C_1
\ar[dr]_-{\cong}^-{1\times_{F_0} F_1}\ar[dd]^(.65){\xi}
\\&X\times_{\C_0'}\C_1'\times_{C_0'}\C_1'
\ar[rr]^(.4){1\times\mu}\ar[dd]_(.35){\xi'\times1}
&&X\times_{\C_0'}\C_1'\ar[dd]^-{\xi'}
\\X\times_{\C_0}\C_1
\ar[rr]_(.75){\xi}\ar[dr]^-{\cong}_-{1\times_{F_0}F_1}
&&X\ar[dr]^-{=}
\\&X\times_{\C_0'}\C_1'\ar[rr]_-{\xi'}&&X.}
\]
\end{proof}

We begin our use of Theorem \ref{equivalentpresheaves} with the following trivial and very useful
corollary.

\begin{corollary}
Let $F:\C\to\C'$ be a target cover.  Then $\C_0$ is naturally a presheaf over $\C'$.
\end{corollary}

\begin{proof} 
First, $\C_0$ is a presheaf over $\C$ with structure map $\id_{\C_0}$.  It follows 
that it is also a presheaf over $\C'$ with structure map $F_0$.
\end{proof}

We have the following converse as well.

\begin{proposition}
Let $X$ be a presheaf over a category $\C$.  Then $X$ is canonically the
set of objects of a category $\C\int X$, together with a target cover $\C\int X\to\C$
whose map on objects is the structure map of $X$ as a presheaf.
\end{proposition}

\begin{proof} 
The category $\C\int X$ is just the Grothendieck construction on the composite
functor
\[
\xymatrix{
\C^\op\ar[r]^-{X}&\Set\ar[r]^-{\delta}&\Cat.}
\]
It has objects $X$, as required, and a morphism $\phi:a\to b$ is a morphism
$\phi\in\C(\e a,\e b)$ for which $a=b\cdot\phi$.  This can also be expressed by
saying that
\[
(\C\textstyle\int X)_1:=X\times_{\C_0}\C_1,
\]
with target given by the projection to $X$, and source given by the action map.
The projection to $\C_1$ gives the functor on morphisms, and it is easy to see
that this gives a target cover with $\e$ as the map on objects.
\end{proof}

There is a generalization that is sometimes useful, as follows.

\begin{lemma}
If $F:\C\to\C'$ is a target cover, then $\C_1$ is naturally a presheaf over $\C'$ consistent
with its left action on itself through composition.
\end{lemma}

\begin{proof}
First, $\C_1$ is a presheaf over $\C$ with structure map the source map 
$S:\C_1\to\C_0$, and action given by composition.  Therefore $\C_1$ is also
a presheaf over $\C'$ with structure map $F_0\circ S=S\circ F_1$, and 
the only issue is consistency with the left action of $\C_1$ on itself.  But we
have the action induced by the composition in $\C$, and consistency
then follows from associativity of composition via the following diagram:
\[
\xymatrix@C+15pt{
\C_1\times_{\C_0}\C_1\times_{\C_0'}\C_1'
\ar[d]_-{\mu\times1}
&\C_1\times_{\C_0}\C_1\times_{\C_0}\C_1
\ar[l]_-{1\times1\times_{F_0}F_1}^-{\cong}\ar[r]^-{1\times\mu}\ar[d]_-{\mu\times1}
&\C_1\times_{\C_0}\C_1
\ar[d]^-{\mu}
\\\C_1\times_{\C_0'}\C_1'
&\C_1\times_{\C_0}\C_1
\ar[l]^-{1\times_{F_0}F_1}_-{\cong}\ar[r]_-{\mu}
&\C_1.}
\] 
\end{proof}

The presheaf structure on $\C_0$ is just the specialization to identity arrows.

We turn now to the proof of Theorem \ref{preservescovers}, 
which says that the monads associated to $\Sigma$-free categorical operads
preserve target and source covers,
and the basic
proposition needed is as follows.

\begin{proposition}\label{orbitspreservecovers}
Let $G$ be a discrete group acting freely on a target cover $F:\C\to\C'$.  Then the map
on orbit categories $F/G:\C/G\to\C'/G$ is also a target cover.  The same is true with ``target''
replaced by ``source.''
\end{proposition}

\begin{proof}
First, we note that since $G$ acts freely on $\C$, the quotient category $\C/G$ has
$(\C/G)_0\cong\C_0/G$ and $(\C/G)_1\cong\C_1/G$.  This is because the composables
$\C_2$ are given by a pullback diagram
\[
\xymatrix{
\C_2\ar[r]^-{T}\ar[d]_-{S}&\C_1\ar[d]_-{S}
\\\C_1\ar[r]_-{T}&\C_0,}
\]
so by Proposition \ref{orbitpullbacks}, we also have a pullback diagram 
\[
\xymatrix{
\C_2/G\ar[r]^-{T/G}\ar[d]_-{S/G}&\C_1/G\ar[d]^-{S/G}
\\\C_1/G\ar[r]_-{T/G}&\C_0/G.}
\]
Consequently the composition map $\mu:\C_1\times_{\C_0}\C_1\to\C_1$
descends to an induced map $\C_1/G\times_{\C_0/G}\C_1/G\to\C_1/G$ that
gives us the composition of a category, which clearly has the universal property
needed for $\C/G$.  Now the pullback diagram that tells us that $F$ is a target cover,
namely
\[
\xymatrix{
\C_1\ar[r]^-{T}\ar[d]_-{F_1}&\C_0\ar[d]^-{F_0}
\\\C'_1\ar[r]_-{T}&\C'_0,}
\]
gives a pullback diagram on orbits
\[
\xymatrix{
\C_1/G\ar[r]^-{T}\ar[d]_-{F_1/G}&\C_0/G\ar[d]^-{F_0/G}
\\\C'_1/G\ar[r]_-{T}&\C'_0/G,}
\]
which tells us that $F/G$ is also a target cover.  Similarly, passage to orbits of free actions
preserves source covers.
\end{proof}

\begin{proof}[Proof of Theorem \ref{preservescovers}]
First,
since both target and source covers are given by a pullback condition, they are both
preserved by products, and it is obvious that the identity functor is a cover.  Therefore, 
if we are given a target (or source) cover $F:\C\to\C'$, the induced functor
\[
\DD_n\times\C^n\to\DD_n\times(\C')^n
\]
is a target (or source) cover.  Now since $\Sigma_n$ acts freely on $\DD_n$, it also
acts freely on both $\DD_n\times\C^n$ and $\DD_n\times(\C')^n$.  It now follows
from Proposition \ref{orbitspreservecovers} that the induced functor
\[
\DD_n\times_{\Sigma_n}\C^n\to\DD_n\times_{\Sigma_n}(\C')^n
\]
is a target (or source) cover.
Further, it is elementary that coproducts preserve both
sorts of cover, so the induced map
\[
\coprod_{n\ge0}\DD_n\times_{\Sigma_n}\C^n\to\coprod_{n\ge0}\DD_n\times_{\Sigma_n}(\C')^n
\]
is also a target (or source) cover.  But this is precisely the functor
$\D F:\D\C\to\D\C'$, so we have proven Theorem \ref{preservescovers}.
\end{proof}

We can now give the proof of our basic Theorem \ref{Dpresheaf}, which says
that whenever $X$ has a presheaf structure over $\C$, $\D_0X$ has a presheaf
structure over $\D\C$.

\begin{proof}
Since $\e:X\to\C_0$ is a presheaf structure map for $X$ over $\C$, we have
a target cover $\C\int X\to\C$ for which $\e$ is the map on objects.  Since $\D$ 
preserves target covers, we have the target cover
\[
\D(\C\textstyle\int X)\to\D\C
\]
with $\D\e$ as the map on objects.
Since the objects of $\D(\C\int X)$ are $\D_0X$, it follows that they are a 
presheaf over $\D\C$ with structure map $\D\e$.
\end{proof}

Corollary \ref{D_0presheaf} now follows immediately.

\section{The Underlying $\D$-multicategory: Proofs}

We gave the structural data for the underlying $\D$-multicategory $U\C$ to
a $\D$-algebra $\C$ in Section \ref{underlying}.  In this section, we show that
these data do satisfy the necessary properties to define a $\D$-multicategory.
First, the formula given for the presheaf action really is a presheaf action.

\begin{theorem}
The composite
\begin{gather*}
\xymatrix{
U\C_1\times_{\D_0\C_0}\D_1\C_0\cong\C_1\times_{C_0}\D_1\C_0
\ar[r]^-{1\times\D_1I}
&\C_1\times_{\C_0}\D_1\C_1}
\\\xymatrix{
\ar[rr]^-{1\times(\xi_1,S_\D S)}
&&\C_1\times_{\C_0}\C_1\times_{\C_0}\D_0\C_0
\ar[r]^-{\gamma_\C\times1}
&\C_1\times_{\C_0}\D_0\C_0\cong U\C_1.}
\end{gather*}
defines a presheaf action on the morphisms $U\C_1$ of the underlying $\D$-multicategory.
\end{theorem}

\begin{proof}
In order to see that this action is unital, we need the diagram
\[
\xymatrix{
U\C_1\cong U\C_1\times_{\D_0\C_0}\D_0\C_0
\ar[r]^-{1\times I_\D}\ar[dr]_-{=}
&U\C_1\times_{\D_0\C_0}\D_1\C_0
\ar[d]^-{\sigma}
\\&U\C_1}
\]
to commute, where $\sigma$ is the presheaf action map.  But by expanding
the definition of $\sigma$, this becomes
\[
\xymatrix{
\C_1\times_{\C_0}\D_0\C_0\ar[r]^-{1\times I_\D}\ar[ddrr]_-{=}
&\C_1\times_{\C_0}\D_1\C_0\ar[r]^-{1\times\D_1I}
&\C_1\times_{\C_0}\D_1\C_1\ar[d]^-{1\times(\xi_1,S_D S)}
\\&&\C_1\times_{\C_0}\C_1\times_{\C_0}\D_0\C_0\ar[d]^-{\gamma_\C\times1}
\\&&\C_1\times_{C_0}\D_0\C_0.}
\]
If we stop just before the end, at the term $\C_1\times_{\C_0}\C_1\times_{\C_0}\D_0\C_0$,
and trace the projections to each factor, we find that the first factor of $\C_1$ is
identical to the original first factor, and the $\D_0\C_0$ on the end is also identical
to the original second factor.  For the $\C_1$ in the middle, we have the
commutative square
\[
\xymatrix{
\D_0\C_0\ar[r]^-{I_\D I}\ar[d]_-{\xi_0}
&\D_1\C_1\ar[d]^-{\xi_1}
\\\C_0\ar[r]_-{I}&\C_1}
\]
which shows that all the elements there are identities.  Consequently they
have no effect on the composition that takes place after, so the net result is
the identity on $U\C_1$, as desired.  

For associativity of the action, the diagram we need is a bit too wide to fit on the page,
so we split it up into two.  The left half is as follows:
\[
\xymatrix{
&\C_1\times_{\C_0}\D_1\C_0\times_{\D_0\C_0}\D_1\C_0
\ar[d]^-{1\times\D_1I\times\D_1I}
\ar[dl]_-{1\times\D_1I\times1}
\\\C_1\times_{\C_0}\D_1\C_1\times_{\D_0\C_0}\D_1\C_0
\ar[r]^-{1^2\times\D_1I}\ar[d]_-{1\times(\xi_1,S_\D S)\times1}
&\C_1\times_{\C_0}\D_1\C_1\times_{\D_0\C_0}\D_1\C_1
\ar[d]^-{1\times\xi_1\times_{\xi_0}(\xi_1,S_\D S)}
\\\C_1\times_{\C_0}\C_1\times_{\C_0}\D_0\C_0\times_{\D_0\C_0}\D_1\C_0
\ar[d]_-{\gamma_\C\times1}
&\C_1\times_{\C_0}\C_1\times_{\C_0}\C_1\times_{\C_0}\D_0\C_0
\ar[d]^-{\gamma_\C\times1^2}
\\\C_1\times_{\C_0}\D_1\C_0
\ar[r]_-{1\times(\xi_1,S_\D S)\circ\D_1I}
&\C_1\times_{\C_0}\C_1\times_{\C_0}\D_0\C_0,
}
\]
and the right half, which is to be glued onto the previous one, is as follows:
\[
\xymatrix{
\C_1\times_{\C_0}\D_1\C_0\times_{\D_0\C_0}\D_1\C_0
\ar[r]^-{1\times\gamma_\D}
\ar[d]^-{1\times\D_1I\times\D_1I}
&\C_1\times_{\C_0}\D_1\C_0
\ar[d]^-{1\times\D_1I}
\\
\C_1\times_{\C_0}\D_1\C_1\times_{\D_0\C_0}\D_1\C_1
\ar[d]^-{1\times\xi_1\times_{\xi_0}(\xi_1,S_\D S)}
\ar[r]^-{1\times\gamma_{\D\C}}
&\C_1\times_{\C_0}\D_1\C_1
\ar[d]^-{1\times(\xi_1,S_\D S)}
\\
\C_1\times_{\C_0}\C_1\times_{\C_0}\C_1\times_{\C_0}\D_0\C_0
\ar[r]^-{1\times\gamma_\C\times1}
\ar[d]^-{\gamma_\C\times1^2}
&\C_1\times_{\C_0}\C_1\times_{\C_0}\D_0\C_0
\ar[d]^-{\gamma_\C\times1}
\\
\C_1\times_{\C_0}\C_1\times_{\C_0}\D_0\C_0
\ar[r]_-{\gamma_\C\times1}
&\C_1\times_{\C_0}\D_0\C_0.}
\]
The top triangle of the left half commutes by definition.  The bottom part 
commutes by examining the projections onto each factor of the target, as follows.  The 
projection onto the first $\C_1$ involves only the $\C_1\times_{\C_0}\D_1\C_1$
of the source, and either way around the diagram gives us
\[
\xymatrix{
\C_1\times_{\C_0}\D_1\C_1\ar[r]^-{1\times\xi_1}
&\C_1\times_{\C_0}\C_1\ar[r]^-{\gamma_\C}&\C_1.}
\]
Projection to the last two factors $\C_1\times_{\C_0}\D_0\C_0$ involves
only the last factor $\D_1\C_0$ of the target, and either way around the 
diagram gives us
\[
\xymatrix{
\D_1\C_0\ar[r]^-{\D_1I}&\D_1\C_1
\ar[r]^-{(\xi_1,S_\D S)}
&\C_1\times_{\C_0}\D_0\C_0.}
\]
The left part of the associativity diagram therefore commutes.
The right half commutes, from top to bottom, by the identity properties 
of $\gamma_\C$, because $\xi:\D\C\to\C$ is a functor, and because 
$\gamma_\C$ is associative.  We may conclude that we have defined
a presheaf structure on $U\C_1$.
\end{proof}

We must also show that the source map $S:U\C_1\to\D_0\C_0$ preserves
the presheaf structure.  The presheaf structure on $\D_0\C_0$ is simply that
of the objects of $\D(\C_0^\delta)$, given by the composite
\[
\xymatrix{
\D_0\C_0\times_{\D_0\C_0}\D_1\C_0\cong\D_1\C_0\ar[r]^-{S_\D}&\D_0\C_0.}
\]
Further, the first part of the presheaf structure map on $U\C_1$ maps 
to the first part of this composite, using
the pullback diagram defining $U\C_1$, augmented slightly as follows,
\[
\xymatrix{
&U\C_1\ar[r]^-{\kappa_1}\ar[d]_-{S}
&\C_1\ar[d]^-{S}
\\\D_1\C_0\ar[r]_-{T_\D}
&\D_0\C_0\ar[r]_-{\xi_0}
&\C_0,}
\]
so we obtain the commuting diagram
\[
\xymatrix{
U\C_1\times_{\D_0\C_0}\D_1\C_0
\ar[r]^-{\kappa_1\times_{\xi_0}1}_-{\cong}
\ar[d]_-{S\times1}
&\C_1\times_{\C_0}\D_1\C_0
\ar[d]^-{p_2}
\\\D_0\C_0\times_{\D_0\C_0}\D_1\C_0
\ar[r]_-{\cong}
&\D_1\C_0.}
\]
It now follows that preservation of the presheaf actions by the source map reduces
to checking the diagram
\[
\xymatrix@C+15pt{
\C_1\times_{\C_0}\D_1\C_0
\ar[r]^-{1\times\D_1I}
\ar[d]_-{p_2}
&\C_1\times_{\C_0}\D_1\C_1
\ar[r]^-{1\times(\xi_1,S_\D S)}
&\C_1\times_{\C_0}\C_1\times_{\C_0}\D_0\C_0
\ar[r]^-{\gamma_\C\times1}
&\C_1\times_{\C_0}\D_0\C_0
\ar[d]^-{p_2}
\\\D_1\C_0
\ar[rrr]_-{S_\D}
&&&\D_0\C_0.}
\]
But this depends only on the last factors in the fiber products, and that reduces to
the commuting diagram
\[
\xymatrix{
\D_1\C_0
\ar[r]^-{\D_1I}
\ar[dr]^-{=}
\ar[ddr]_-{S_\D}
&\D_1\C_1
\ar[d]^-{\D_1S}
\\&\D_1\C_0
\ar[d]^-{S_\D}
\\&\D_0\C_0.}
\]
It follows that the source map on $U\C$ preserves the presheaf action.

For the unital property of the structure, recall that the identity map of 
the underlying multicategory is induced from the commutative square
\[
\xymatrix{
\C_0\ar[r]^-{I_\C}\ar[d]_-{\eta}&\C_1\ar[d]^-{S}
\\\D_0\C_0\ar[r]_-{\xi_0}&\C_0,}
\]
so we get
an induced map $I_{U\C}:\C_0\to\C_1\times_{\C_0}\D_0\C_0\cong(U\C)_1$.
It is now an easy exercise to see that
\[
\xymatrix{
&\C_0\ar[dl]_-{\eta}\ar[d]^-{I_{U\C}}\ar[dr]^-{=}
\\\D_0\C_0&(U\C)_1\ar[l]^-{S}\ar[r]_-{T}&\C_0}
\]
commutes.

For the remainder of the section, we verify the formal properties of the composition
law.  We must verify that it  preserves sources and targets, is unital, and is
associative.  To see that the composition preserves targets, we reduce to the corresponding
property for $\C$ by means of
the following diagram:
\[
\xymatrix{
(U\C)_2\ar[rr]^-{T}\ar[dd]_-{\gamma_{U\C}}\ar[dr]^-{\kappa_2}&&(U\C)_1\ar[d]^-{\kappa_1}
\\&\C_2\ar[r]^-{T}\ar[d]^-{\gamma_\C}&\C_1\ar[d]^-{T}
\\(U\C)_1\ar[r]_-{\kappa_1}&\C_1\ar[r]_-{T}&\C_0.}
\]
For preservation of sources, we have defined $\gamma_{U\C}$ so that the
diagram
\[
\xymatrix{
(U\C)_2\ar[rr]^-{\gamma_{U\C}}\ar[d]_-{S}&&(U\C)_1\ar[d]^-{S}
\\\D_0(U\C)_1\ar[r]_-{\D_0S}&\D_0^2\C_0\ar[r]_-{\mu}&\D_0\C_0}
\]
commutes, so $\gamma_{U\C}$ does preserve sources.

In verifying the unital properties of the composition, we introduce the notations
$I_L:M_1\to M_2$ and $I_R:M_1\to M_2$ for the composites that appear
in the unital conditions for composition in any $\D$-multicategory.  Specifically,
we use $I_L$ for the left (target) unit map defined by the composite
\[
\xymatrix{
M_1\cong M_0\times_{M_0}M_1
\ar[r]^-{I\times_\eta\eta}
&M_1\times_{\D_0M_0}\D_0M_1\cong M_2,}
\]
and $I_R$ for the right (source) unit map defined by the composite
\[
\xymatrix{
M_1\cong M_1\times_{\D_0M_0}\D_0M_0
\ar[r]^-{1\times\D_0I}
&M_1\times_{\D_0M_0}\D_0M_1\cong M_2.}
\]
Note that both triangles
\[
\xymatrix{
M_1\ar[r]^-{I_L}\ar[dr]_-{\eta}
&M_2\ar[d]^-{S}
\\&\D_0M_1}
\quad\text{and}\quad
\xymatrix{
M_1\ar[r]^-{I_R}\ar[dr]_-{\eta}
&M_2\ar[d]^-{S}
\\&\D_0M_1}
\]
commute.
To see that $\gamma_{U\C}$ is left (target) unital, we will need to show that
\[
\xymatrix{
(U\C)_1\ar[r]^-{\kappa_1}\ar[d]_-{I_L}&\C_1\ar[d]^-{I_L}
\\(U\C)_2\ar[r]_-{\kappa_2}&\C_2}
\]
commutes, in order to reduce to the analogous property of $\C$.  Since $\C_2$
is a pullback, the square commutes if and only if it does so after composing with
the two maps $S,T:\C_2\to\C_1$.  Composing with $T$, we obtain 
\[
\xymatrix{
(U\C)_1\ar[r]^-{\kappa_1}\ar[d]_-{I_L}&\C_1\ar[dr]^-{T}\ar[d]^-{I_L}
\\(U\C)_2\ar[r]^-{\kappa_2}\ar[dr]_-{T}&\C_2\ar[dr]^-{T}&\C_0\ar[d]^-{I_\C}
\\&(U\C)_1\ar[r]_-{\kappa_1}&\C_1.}
\]
Since the two new squares commute, this reduces us to verifying the commutativity
of the perimeter of the hexagon, which follows from replacing its interior as follows:
\[
\xymatrix{
(U\C)_1\ar[r]^-{\kappa_1}\ar[d]_-{I_L}&\C_1\ar[dr]^-{T}
\\(U\C)_2\ar[dr]_-{T}&&\C_0\ar[d]^-{I_\C}\ar[dl]_-{I_{U\C}}
\\&(U\C)_1\ar[r]_-{\kappa_1}&\C_1.}
\]
Continuing the verification that $\gamma_{U\C}$ is left unital, we now
compose the desired comparison square with $S$, and obtain
\[
\xymatrix{
&(U\C)_1\ar[r]^-{\kappa_1}\ar[d]^-{I_L}\ar[ddl]_-{\eta}
&\C_1\ar[d]_-{I_L}\ar[ddr]^-{=}
\\&(U\C)_2\ar[r]_-{\kappa_2}\ar[dl]^-{S}
&\C_2\ar[dr]_-{S}
\\\D_0(U\C)_1\ar[r]_-{\D_0\kappa_1}&\D_0\C_1\ar[r]_-{I_\D}&\D_1\C_1\ar[r]_-{\xi_1}&\C_1.}
\]
Again, we have reduced to the question of whether the perimeter commutes, and we 
rearrange the innards to obtain
\[
\xymatrix{
&(U\C)_1\ar[r]^-{\kappa_1}\ar[dl]_-{\eta}&\C_1\ar[dr]^-{=}\ar[d]^-{\eta}\ar[dl]_-{\eta}
\\\D_0(U\C)_1\ar[r]_-{\D_0\kappa_1}&\D_0\C_1\ar[r]_-{I_\D}&\D_1\C_1\ar[r]_-{\xi_1}&\C_1,}
\]
which does commute.  The comparison square therefore commutes.  We now use it
to show that $\gamma_{U\C}$ is left unital, which says that the diagram
\[
\xymatrix{
(U\C)_1\ar[r]^-{I_L}\ar[dr]_-{=}&(U\C)_2\ar[d]^-{\gamma_{U\C}}
\\&(U\C)_1}
\]
commutes.  Again, since $(U\C)_1$ is defined as a pullback, this diagram
commutes if and only if it does so after composition with $\kappa_1:(U\C)_1\to\C_1$
and $S:(U\C)_1\to\D_0\C_0$.  Composing with $\kappa_1$, we obtain the diagram
\[
\xymatrix{
(U\C)_1\ar[r]^-{I_L}\ar[d]_-{\kappa_1}&(U\C)_2\ar[r]^-{\gamma_{U\C}}\ar[d]^-{\kappa_2}
&(U\C)_1\ar[d]^-{\kappa_1}
\\\C_1\ar[r]_-{I_L}&\C_2\ar[r]_-{\gamma_\C}&\C_1.}
\]
Since the bottom row composes to $\id_{\C_1}$, we just get $\kappa_1$, as desired. 
Composing with $S$, we get
\[
\xymatrix{
(U\C)_1\ar[r]^-{I_L}\ar[dr]_-{\eta}\ar[dd]_-{S}&(U\C)_2\ar[r]^-{\gamma_{U\C}}\ar[d]^-{S}
&(U\C)_1\ar[dd]^-{S}
\\
&\D_0(U\C)_1\ar[d]^-{\D_0S}
\\
\D_0\C_0\ar[r]_-{\eta}&\D_0^2\C_0\ar[r]_-{\mu}&\D_0\C_0.}
\]
Again, the bottom row composes to the identity, so we just get $S$, as desired.  It follows
that $\gamma_{U\C}\circ I_L=\id_{(U\C)_1}$, so $\gamma_{U\C}$ is left unital.

To show that $\gamma_{U\C}$ is right (source) unital, we proceed in much the same way.  First,
we show that the square
\[
\xymatrix{
(U\C)_1\ar[r]^-{\kappa_1}\ar[d]_-{I_R}&\C_1\ar[d]^-{I_R}
\\(U\C)_2\ar[r]_-{\kappa_2}&\C_2}
\]
commutes, again by showing that the composites coincide after composing with both of
$S,T:\C_2\to\C_1$.  Composing with $T$, we obtain the diagram
\[
\xymatrix{&(U\C)_1\ar[r]^-{\kappa_1}\ar[d]^-{I_R}\ar[ddl]_-{=}
&\C_1\ar[ddr]^-{=}\ar[d]_-{I_R}
\\&(U\C)_2\ar[r]_-{\kappa_2}\ar[dl]^-{T}&\C_2\ar[dr]_-{T}
\\(U\C)_1\ar[rrr]_-{\kappa_1}&&&\C_1,}
\]
whose perimeter obviously commutes.  Composing with $S$, we obtain the diagram
\[
\xymatrix{
&&\D_0\C_0\ar[dr]^-{\xi_0}
\\&(U\C)_1\ar[r]_-{\kappa_1}\ar[d]^{I_R}\ar[ur]^-{S}\ar[dl]_-{S}
&\C_1\ar[r]_-{S}\ar[d]^-{I_R}&\C_0\ar[dd]^-{I_\C}
\\\D_0\C_0\ar[d]_-{\D_0I_{U\C}}&(U\C)_2\ar[r]_-{\kappa_2}\ar[dl]_-{S}&\C_2\ar[dr]^-{S}
\\\D_0(U\C)_1\ar[r]_-{\D_0\kappa_1}&\D_0\C_1\ar[r]_-{I_\D}&\D_1\C_1\ar[r]_-{\xi_1}&\C_1,}
\]
whose perimeter commutes by rearranging the insides as follows:
\[
\xymatrix{
(U\C)_1\ar[r]^-{S}&\D_0\C_0\ar[r]^-{\xi_0}\ar[dr]^-{I_{\D\C}}\ar[d]^-{\D_0I_\C}\ar[dl]_-{\D_0I_{U\C}}
&\C_0\ar[dr]^-{I_\C}
\\\D_0(U\C)_1\ar[r]_-{\D_0\kappa_1}&\D_0\C_1\ar[r]_-{I_\D}&\D_1\C_1\ar[r]_-{\xi_1}&\C_1.}
\]
The comparison square therefore commutes.  We use it to establish the right unital
condition for $\gamma_{U\C}$, which says that
\[
\xymatrix{
(U\C)_1\ar[r]^-{I_R}\ar[dr]_-{=}&(U\C)_2\ar[d]^-{\gamma_{U\C}}
\\&(U\C)_1}
\]
commutes.  As with the left unital condition, we compose with the two maps $\kappa_1:(U\C)_1\to\C_1$
and $S:(U\C)_1\to\D_0\C_0$ and verify that the resulting diagrams commute.  Composing
first with $\kappa_1$, we get the diagram
\[
\xymatrix{
(U\C)_1\ar[r]^-{I_R}\ar[d]_-{\kappa_1}&(U\C)_2\ar[r]^-{\gamma_{U\C}}\ar[d]^-{\kappa_2}
&(U\C)_1\ar[d]^-{\kappa_1}
\\\C_1\ar[r]_-{I_R}&\C_2\ar[r]_-{\gamma_\C}&\C_1,}
\]
and since the bottom row composes to $\id_{\C_1}$, we get the desired result.
Composing with $S$, we get the diagram
\[
\xymatrix{
(U\C)_1\ar[r]^-{I_R}\ar[dr]_-{\eta}\ar[dd]_-{S}&(U\C)_2\ar[r]^-{\gamma_{U\C}}\ar[d]^-{S}
&(U\C)_1\ar[dd]^-{S}
\\
&\D_0(U\C)_1\ar[d]^-{\D_0S}
\\
\D_0\C_0\ar[r]_-{\eta}&\D_0^2\C_0\ar[r]_-{\mu}&\D_0\C_0.}
\]
in which the bottom row composes to the identity.  This completes the check
of the right unital condition for $\gamma_{U\C}$.

It remains to verify associativity for $\gamma_{U\C}$.  
We will need the comparison maps $\kappa_2$ and $\kappa_3$
to be compatible with the two composition maps $\gamma_S,\gamma_T$
 in the sense that the two squares
 \[
 \xymatrix{
 (U\C)_3\ar[r]^-{\kappa_3}\ar[d]_-{\gamma_S}&\C_3\ar[d]^-{\gamma_S}
 \\(U\C)_2\ar[r]_-{\kappa_2}&\C_2}
 \quad\text{and}\quad
 \xymatrix{
 (U\C)_3\ar[r]^-{\kappa_3}\ar[d]_-{\gamma_T}&\C_3\ar[d]^-{\gamma_T}
 \\(U\C)_2\ar[r]_-{\kappa_2}&\C_2}
 \]
 both commute.  Despite their formal similarity, the second diagram requires more
 work to verify than the first one, so we begin with the first one.  Our strategy
 in both is to compose with the two maps $S,T:\C_2\to\C_1$ and verify the two
 resulting diagrams.  
 
Composing the compatibility diagram for $\gamma_S$ with $T$, we wish the perimeter of the diagram
\[
\xymatrix{
(U\C)_3\ar[r]^-{\kappa_3}\ar[d]_-{\gamma_S}&\C_3\ar[r]^-{T}\ar[d]^-{\gamma_S}&\C_2\ar[dd]^-{T}
\\(U\C)_2\ar[d]_-{T}\ar[r]_-{\kappa_2}&\C_2\ar[dr]^-{T}
\\(U\C)_1\ar[rr]_-{\kappa_1}&&\C_1}
\]
to commute, so we rearrange the interior to get the commutative diagram
\[
\xymatrix{
(U\C)_3\ar[rr]^-{\kappa_3}\ar[dr]^-{T}\ar[dd]_-{\gamma_S}&&\C_3\ar[d]^-{T}
\\&(U\C)_2\ar[r]^-{\kappa_2}\ar[d]_-{T}&\C_2\ar[d]^-{T}
\\(U\C)_2\ar[r]_-{T}&(U\C)_1\ar[r]_-{\kappa_1}&\C_1.}
\]
Composing instead with $S$, we wish the perimeter of the diagram
\[
\xymatrix{
(U\C)_3\ar[rr]^-{\kappa_3}\ar[d]_-{\gamma_S}&&\C_3\ar[r]^-{S}\ar[d]^-{\gamma_S}&\C_2\ar[dd]^-{\gamma}
\\(U\C)_2\ar[rr]_-{\kappa_2}\ar[d]_-{S}&&\C_2\ar[dr]^-{S}
\\\D_0(U\C)_1\ar[r]_-{\D_0\kappa_1}&\D_0\C_1\ar[r]_-{I_\D}&\D_1\C_1\ar[r]_-{\xi_1}&\C_1}
\]
to commute, but here we can rearrange the interior to obtain the commutative diagram
\[
\xymatrix{
(U\C)_3\ar[rrrr]^-{\kappa_3}\ar[dr]^-{S}\ar[dd]_-{\gamma_S}&&&&\C_3\ar[d]^-{S}
\\&\D_0(U\C)_2\ar[r]^-{\D_0\kappa_2}\ar[d]_-{\D_0\gamma}&\D_0\C_2\ar[r]^-{I_\D^2}\ar[d]^-{\D_0\gamma}
&\D_2\C_2\ar[r]^-{\xi_2}\ar[d]^-{\gamma_{\D\C}}&\C_2\ar[d]^-{\gamma}
\\(U\C)_2\ar[r]_-{S}&\D_0(U\C)_1\ar[r]_-{\D_0\kappa_1}&\D_0\C_1\ar[r]_-{I_\D}&\D_1\C_1\ar[r]_-{\xi_1}&\C_1.}
\]
Next, we verify the compatibility diagram for $\gamma_T$, and we first compose with $T$ to obtain
the diagram
\[
\xymatrix{
(U\C)_3\ar[r]^-{\kappa_3}\ar[d]_-{\gamma_T}&\C_3\ar[r]^-{T}\ar[d]^-{\gamma_T}&\C_2\ar[dd]^-{\gamma}
\\(U\C)_2\ar[r]_-{\kappa_2}\ar[d]_-{T}&\C_2\ar[dr]^-{T}
\\(U\C)_1\ar[rr]_-{\kappa_1}&&\C_1,}
\]
whose perimeter we want to commute.  However, we again rearrange the insides and get
\[
\xymatrix{
(U\C)_3\ar[rr]^-{\kappa_3}\ar[dd]_-{\gamma_T}\ar[dr]^-{T}&&\C_3\ar[d]^-{T}
\\&(U\C)_2\ar[r]^-{\kappa_2}\ar[d]_-{\gamma}&\C_2\ar[d]^-{\gamma}
\\(U\C)_2\ar[r]_-{T}&(U\C)_1\ar[r]_-{\kappa_1}&\C_1.}
\]
Composing with $S$, we find we wish the perimeter of
\[
\xymatrix{
(U\C)_3\ar[rr]^-{\kappa_3}\ar[d]_-{\gamma_T}&&\C_3\ar[r]^-{S}\ar[d]^-{\gamma_T}&\C_2\ar[dd]^-{S}
\\(U\C)_2\ar[rr]_-{\kappa_2}\ar[d]_-{S}&&\C_2\ar[dr]^-{S}
\\\D_0(U\C)_1\ar[r]_-{\D_0\kappa_1}&\D_0\C_1\ar[r]_-{I_\D}&\D_1\C_1\ar[r]_-{\xi_1}&\C_1}
\]
to commute.  We can fill in the interior with the following, in which we know
all but the large, irregular sub-diagram at the lower right commutes:
\[
\xymatrix{
(U\C)_3\ar[rrrr]^-{\kappa_3}\ar[dr]^-{S}\ar[ddd]_-{\gamma_T}&&&&\C_3\ar[d]^-{S}
\\&\D_0(U\C)_2\ar[r]^-{\D_0\kappa_2}\ar[d]_-{\D_0S}&\D_0\C_2\ar[r]_-{I_\D^2}
&\D_2\C_2\ar[r]_-{\xi_2}&\C_2\ar[dd]^-{S}
\\&\D_0^2(U\C)_1\ar[r]^-{\D_0^2\kappa_1}\ar[d]_-{\mu}&\D_0^2\C_1\ar[d]^-{\mu}
\\(U\C)_2\ar[r]_-{S}&\D_0(U\C)_1\ar[r]_-{\D_0\kappa_1}&\D_0\C_1\ar[r]_-{I_\D}&\D_1\C_1\ar[r]_-{\xi_1}&\C_1.}
\]
And the irregular sub-diagram can be expanded and filled in as follows:
\[
\xymatrix{
\D_0(U\C)_2\ar[rr]^-{\D_0\kappa_2}\ar[d]_-{\D_0S}&&\D_0\C_2\ar[r]^-{I_\D^2}\ar[dd]^-{\D_0S}
&\D_2\C_2\ar[r]^-{\xi_2}\ar[ddd]^-{S}&\C_2\ar[dddd]^-{S}
\\\D_0^2(U\C)_1\ar[d]_-{\D_0^2\kappa_1}
\\\D_0^2\C_1\ar[r]^-{\D_0I_\D}\ar[dd]_-{\mu}&\D_0\D_1\C_1\ar[r]^-{\D_0\xi_1}\ar[d]_-{I_\D}
&\D_0\C_1\ar[dr]^-{I_\D}
\\&\D_1^2\C_1\ar[rr]_-{\D_1\xi_1}\ar[d]^-{\mu_{\D_1}}&&\D_1\C_1\ar[dr]^-{\xi_1}
\\\D_0\C_1\ar[r]_-{I_\D}&\D_1\C_1\ar[rrr]_-{\xi_1}&&&\C_1.}
\]
We may conclude that the comparison diagram for $\gamma_T$ commutes.

Now we can verify the actual associativity diagram, which we do by again recognizing
the target of the diagram
\[
\xymatrix{
(U\C)_3\ar[r]^-{\gamma_T}\ar[d]_-{\gamma_S}&(U\C)_2\ar[d]^-{\gamma}
\\(U\C)_2\ar[r]_-{\gamma}&(U\C)_1}
\]
as a pullback, and composing with the two maps $S:(U\C)_1\to\D_0\C_0$ and
$\kappa_1:(U\C)_1\to\C_1$.  Composing with $\kappa_1$ results in the following diagram,
whose perimeter we wish to commute:
\[
\xymatrix{
(U\C)_3\ar[r]^-{\gamma_T}\ar[d]_-{\gamma_S}&(U\C)_2\ar[r]^-{\kappa_2}\ar[d]^-{\gamma}
&\C_2\ar[dd]^-{\gamma}
\\(U\C)_2\ar[r]_-{\gamma}\ar[d]_-{\kappa_2}&(U\C)_1\ar[dr]^-{\kappa_1}
\\\C_2\ar[rr]_-{\gamma}&&\C_1.}
\]
But this follows from the next diagram, in which we use the comparison diagrams
just verified:
\[
\xymatrix{
(U\C)_3\ar[rr]^-{\gamma_T}\ar[dd]_-{\gamma_S}\ar[dr]^-{\kappa_3}&&(U\C)_2\ar[d]^-{\kappa_2}
\\&\C_3\ar[r]^-{\gamma_T}\ar[d]_-{\gamma_S}&\C_2\ar[d]^-{\gamma}
\\(U\C)_2\ar[r]_-{\kappa_2}&\C_2\ar[r]_-{\gamma}&\C_1.}
\]
Composing instead with $S$, we get the following diagram, whose perimeter
we also wish to commute:
\[
\xymatrix{
(U\C)_3\ar[r]^-{\gamma_T}\ar[d]_-{\gamma_S}&(U\C)_2\ar[r]^-{S}\ar[d]^-{\gamma}
&\D_0(U\C)_1\ar[d]^-{\D_0S}
\\(U\C)_2\ar[r]_-{\gamma}\ar[d]_-{S}&(U\C)_1\ar[dr]^-{S}&\D_0^2\C_0\ar[d]^-{\mu}
\\\D_0M_1\ar[r]_-{\D_0S}&\D_0^2\C_0\ar[r]_-{\mu}&\D_0\C_0.}
\]
However, this follows from the commutativity of 
\[
\xymatrix{
(U\C)_3\ar[rrr]^-{\gamma_T}\ar[dr]^-{S}\ar[ddd]_-{\gamma_S}&&&(U\C)_2\ar[d]^-{S}
\\&\D_0(U\C)_2\ar[r]^-{\D_0S}\ar[dd]_-{\D_0\gamma}&\D_0^2(U\C)_1\ar[r]^-{\mu}\ar[d]_-{\D_0^2S}
&\D_0(U\C)_1\ar[d]^-{\D_0S}
\\&&\D_0^3\C_0\ar[r]^-{\mu}\ar[d]_-{\D_0\mu}&\D_0^2\C_0\ar[d]^-{\mu}
\\(U\C)_2\ar[r]_-{S}&\D_0(U\C)_1\ar[r]_-{\D_0S}&\D_0^2\C_0\ar[r]_-{\mu}&\D_0\C_0.}
\]
This completes the verification of associativity, and therefore $U\C$ has all the 
necessary properties of a $\D$-multicategory. 

\section{The provisional left adjoint: category structure}

In this section we show that the definition given for the provisional left adjoint
$\Lhat M$ for a $\D$-multicategory $M$ is actually a category.  Recall first that
the morphisms $\Lhat M_1$ are given by a pullback diagram
\[
\xymatrix{
\Lhat M_1\ar[r]^-{\That}\ar[dd]_-{\Shat}&\D_0M_1\ar[d]^-{S}
\\&\D_0^2M_0\ar[d]^-{\mu}
\\\D_1M_0\ar[r]_-{T_\D}&\D_0M_0,}
\]
so we can write
\[
\Lhat M_1\cong\D_0M_1\times_{\D_0M_0}\D_1M_0.
\]
We specify an identity map for $\Lhat M$ by means of the
commutative diagram
\[
\xymatrix{\D_0M_0\ar[r]^-{\D_0I}\ar[dd]_-{I_\D}&\D_0M_1\ar[d]^-{\D_0S}
\\&\D_0^2M_0\ar[d]^-{\mu}
\\\D_1M_0\ar[r]_-{T_\D}&\D_0M_0.}
\]
We see that both composites coincide with the identity on $\D_0M_0$, 
so the diagram commutes and we get a map to $\Lhat M_1$.  Further,
both composites
\[
\xymatrix{\D_0M_0\ar[r]^-{\D_0I}&\D_0M_1\ar[r]^-{\D_0T}&\D_0M_0}
\quad\text{and}\quad
\xymatrix{\D_0M_0\ar[r]^-{I_\D}&\D_1M_0\ar[r]^-{S_\D}&\D_0M_0}
\]
coincide with the identity on $\D_0M_0$, which verifies the necessary
properties for an identity map.

The remainder of this section is devoted to the formal properties of the 
composition on $\Lhat M$.  Recall that the composition map is defined
as the composite
\begin{gather*}
\Lhat M_2
\cong\D_0M_1\times_{\D_0M_0}\D_1M_0\times_{\D_0M_0}\D_0M_1\times_{\D_0M_0}\D_1M_0
\\
\xymatrix{\relax
&&\D_0M_1\times_{\D_0M_0}\D_1M_1\times_{\D_0M_0}\D_1M_0\ar[ll]_-{1\times(\D_1T,S_\D)\times1}^-{\cong}}
\\
\xymatrix{
\ar[rr]^-{1\times(T_\D,\theta)\times1}&&\D_0M_1\times_{\D_0M_0}\D_0M_1\times_{\D_0M_0}\D_1M_0\times_{\D_0M_0}\D_1M_0}
\\
\xymatrix{
\cong\D_0M_2\times_{\D_0M_0}\D_2M_0
\ar[rr]^-{\D_0\gamma_M\times\gamma_\D}&&\D_0M_1\times_{\D_0M_0}\D_1M_0\cong\Lhat M_1.}
\end{gather*}
We will often think of the first part of this construction as being given by an
interchange map
\[
\chi: \D_1M_0\times_{\D_0M_0}\D_0M_1\to\D_0M_1\times_{\D_0M_0}\D_1M_0
\]
given by the composite
\[
\xymatrix@C+10pt{
\D_1M_0\times_{\D_0M_0}\D_0M_1
&\D_1M_1\ar[l]_-{(\D_1T,S_\D)}^-{\cong}\ar[r]^-{(T_\D,\theta)}
&D_0M_1\times_{\D_0M_0}\D_1M_0,}
\]
where the map $\theta$ is the composite
\[
\xymatrix{
\D_1M_1\ar[r]^-{\D_1S}
&\D_1\D_0M_0\ar[r]^-{\D_1I_\D}
&\D_1^2M_0\ar[r]^-{\mu}
&\D_1M_0.}
\]
We must show that the composition preserves source and target maps, is left
and right unital, and is associative.  For preservation of source and target maps, 
we see from the defining diagram for $\Lhat M_1$ that we can write the target as
\[
\xymatrix{
\Lhat M_1\cong\D_0M_1\times_{\D_0M_0}\D_1M_0\ar[r]^-{p_1}
&\D_0M_1\ar[r]^-{\D_0T}&\D_0M_0.}
\]
From the same diagram, we can write the source as
\[
\xymatrix{
\Lhat M_1\cong\D_0M_1\times_{\D_0M_0}\D_1M_0\ar[r]^-{p_2}
&\D_1M_0\ar[r]^-{S_\D}&\D_0M_0.}
\]
We now verify preservation of targets from the commutativity of the
following diagram, where the subscript $\D_0M_0$'s have been 
suppressed in the interest of space:
\[
\xymatrix{
\D_0M_1\times\D_1M_0\times\D_0M_1\times\D_1M_0
\ar[rr]^-{p_{12}}\ar[d]_-{1\times\chi\times1}
&&\D_0M_1\times\D_1M_0\ar[dd]^-{p_1}
\\\D_0M_1\times\D_0M_1\times\D_1M_0\times\D_1M_0
\ar[r]^-{p_{12}}\ar[d]_-{\cong}
&\D_0M_1\times\D_0M_1\ar[dr]^-{p_1}\ar[d]_-{\cong}
\\\D_0M_2\times\D_2M_0\ar[r]^-{p_1}\ar[d]_-{\D_0\gamma\times\gamma_\D}
&\D_0M_2\ar[r]_-{\D_0T}\ar[d]_-{\D_0\gamma}&\D_0M_1\ar[d]^-{\D_0T}
\\\D_0M_1\times\D_1M_0\ar[r]_-{p_1}&\D_0M_1\ar[r]_-{\D_0T}&\D_0M_0.}
\]
Preservation of sources is a similar diagram:
\[
\xymatrix{
\D_0M_1\times\D_1M_0\times\D_0M_1\times\D_1M_0
\ar[rr]^-{p_{34}}\ar[d]_-{1\times\chi\times1}
&&\D_0M_1\times\D_1M_0\ar[dd]^-{p_2}
\\\D_0M_1\times\D_0M_1\times\D_1M_0\times\D_1M_0
\ar[r]^-{p_{34}}\ar[d]_-{\cong}
&\D_1M_0\times\D_1M_0\ar[dr]^-{p_2}\ar[d]_-{\cong}
\\\D_0M_2\times\D_2M_0\ar[r]^-{p_2}\ar[d]_-{\D_0\gamma\times\gamma_\D}
&\D_2M_0\ar[r]_-{S_\D}\ar[d]_-{\gamma_\D}&\D_1M_0\ar[d]^-{S_\D}
\\\D_0M_1\times\D_1M_0\ar[r]_-{p_2}&\D_1M_0\ar[r]_-{S_\D}&\D_0M_0.}
\]

The left unit map $I_L:\Lhat M_1\to\Lhat M_2$ is given by the composite
\begin{gather*}
\Lhat M_1\cong\D_0M_1\times_{\D_0M_0}\D_1M_0\cong
\D_0M_0\times_{\D_0M_0}\D_0M_1\times_{\D_0M_0}\D_1M_0
\\
\xymatrix{
\relax\ar[rr]^-{(\D_0I,I_\D)\times1}
&&\D_0M_1\times_{\D_0M_0}\D_1M_0\times_{\D_0M_0}\D_0M_1\times_{\D_0M_0}\D_1M_0
\cong\Lhat M_2,}
\end{gather*}
and similarly the right unit map $I_R$ is given by the composite
\begin{gather*}
\Lhat M_1\cong\D_0M_1\times_{\D_0M_0}\D_1M_0\cong
\D_0M_1\times_{\D_0M_0}\D_1M_0\times_{\D_0M_0}\D_0M_0
\\
\xymatrix{
\relax\ar[rr]^-{1\times(\D_0I,I_\D)}
&&\D_0M_1\times_{\D_0M_0}\D_1M_0\times_{\D_0M_0}\D_0M_1\times_{\D_0M_0}\D_1M_0
\cong\Lhat M_2.}
\end{gather*}
To show that $\gamma$ is unital,  we need to verify 
commutativity for
both triangles in the diagram
\[
\xymatrix{
\Lhat M_1\ar[r]^-{I_L}\ar[dr]_-{=}&\Lhat M_2\ar[d]_-{\gamma}&\Lhat M_1\ar[l]_-{I_R}\ar[dl]^-{=}
\\&\Lhat M_1.}
\]
This requires the following property of the interchange 
$\chi:\D_1M_0\times_{\D_0M_0}\D_0M_1\to\D_0M_1\times_{\D_0M_0}\D_1M_0$:

\begin{lemma}\label{composites}
The composites
\begin{gather*}
\xymatrix{
\D_0M_1\cong\D_0M_0\times_{\D_0M_0}\D_0M_1
\ar[r]^-{I_\D\times1}&\D_1M_0\times_{\D_0M_0}\D_0M_1}
\\
\xymatrix{
\relax\ar[r]^-{\chi}&\D_0M_1\times_{\D_0M_0}\D_1M_0
\ar[r]^-{p_1}&\D_0M_1}
\end{gather*}
and
\begin{gather*}
\xymatrix{
D_1M_0\cong\D_1M_0\times_{\D_0M_0}\D_0M_0
\ar[rr]^-{1\times\D_0I_M}&&\D_1M_0\times_{\D_0M_0}\D_0M_1}
\\
\xymatrix{
\relax\ar[r]^-{\chi}&\D_0M_1\times_{\D_0M_0}\D_1M_0
\ar[r]^-{p_2}&\D_1M_0}
\end{gather*}
both coincide with identity  maps on $\D_0M_1$ and $\D_1M_0$, 
respectively.
\end{lemma}

\begin{proof}
The interchange map $\chi$ factors as
\[
\xymatrix{
\D_1M_0\times_{\D_0M_0}\D_0M_1\cong\D_1M_1
\ar[r]^-{(T_\D,\theta)}&\D_0M_1\times_{\D_0M_0}\D_1M_0,}
\]
and examining the pullback diagram giving the isomorphism
part of the composite shows easily that
\[
\xymatrix{
\D_0M_1\ar[r]^-{(I_\D,1)}&\D_1M_0\times_{\D_0M_0}\D_0M_1\cong\D_1M_1}
\]
coincides with $I_\D:\D_0M_1\to\D_1M_1$.  Now the composite
\[
\xymatrix{
\D_1M_1\ar[r]^-{(T_\D,\theta)}&\D_0M_1\times_{\D_0M_0}\D_1M_0\ar[r]^-{p_1}&\D_0M_1}
\]
coincides with $T_\D$, and since $T_\D\circ I_\D=\id_{\D_0M_1}$, we see
that the first composite is as claimed.  

For the second composite, we again examine the pullback diagram
giving $\D_1M_1\cong\D_1M_0\times_{\D_0M_0}\D_0M_1$ and conclude that
\[
\xymatrix{
\D_1M_0\ar[rr]^-{(1,\D_0I_M)}&&\D_1M_0\times_{\D_0M_0}\D_0M_1\cong\D_1M_1}
\]
coincides with $\D_1I_M$.  Now the conclusion about the second composite
follows from the commutative diagram
\[
\xymatrix{
\D_1M_1\ar[d]_-{\D_1S}
&\D_1M_0\ar[l]_-{\D_1I}\ar[dl]^-{\D_1\eta_{\D_0}}\ar[dr]_-{\D_1\eta_{\D_1}}\ar[r]^-{=}
&\D_1M_0
\\\D_1\D_0M_0\ar[rr]_-{\D_1I_\D}&&\D_1^2M_0.\ar[u]_-{\mu}}
\]
\end{proof}

Now the left unitality triangle commutes if and only if it does so after
composing with $p_1:\D_0M_1\times_{\D_0M_0}\D_1M_0\to\D_0M_1$
and $p_2:\D_0M_1\times_{\D_0M_0}\D_1M_0\to\D_1M_0$.  Composing with
$p_1$, we wish the composite
\begin{gather*}
\Lhat M_1\cong\D_0M_1\times_{\D_0M_0}\D_1M_0\cong
\D_0M_0\times_{\D_0M_0}\D_0M_1\times_{\D_0M_0}\D_1M_0
\\
\xymatrix{
\relax\ar[rr]^-{(\D_0I,I_\D)\times1}
&&\D_0M_1\times_{\D_0M_0}\D_1M_0\times_{\D_0M_0}\D_0M_1\times_{\D_0M_0}\D_1M_0}
\\
\xymatrix{
\relax\ar[rr]^-{1\times\chi\times1}
&&\D_0M_1\times_{\D_0M_0}\D_0M_1\times_{\D_0M_0}\D_1M_0\times_{\D_0M_0}\D_1M_0}
\\
\xymatrix{
\relax\ar[rr]^-{\D_0\gamma_M\times\gamma_\D}&&\D_0M_1\times_{\D_0M_0}\D_1M_0
\ar[r]^-{p_1}&\D_0M_1}
\end{gather*}
to coincide with just $p_1$.  The next diagram shows that we can project off the last factor
$\D_1M_0$ from the beginning, where we again have suppressed subscript $\D_0M_0$'s in the
interest of space:
\[
\xymatrix{
\D_0M_1\times\D_1M_0\ar[r]^-{p_1}\ar[d]_-{\cong}&\D_0M_1\ar[d]^-{\cong}
\\\D_0M_0\times\D_0M_1\times\D_1M_0\ar[r]^-{p_{12}}\ar[d]_-{(\D_0I,I_\D)\times1}
&\D_0M_0\times\D_0M_1\ar[d]^-{(\D_0I,I_\D)\times1}
\\\D_0M_1\times\D_1M_0\times\D_0M_1\times\D_1M_0\ar[r]^-{p_{123}}\ar[dd]^-{1\times\chi\times1}
&\D_0M_1\times\D_1M_0\times\D_0M_1\ar[d]^-{1\times\chi}
\\&\D_0M_1\times\D_0M_1\times\D_1M_0\ar[d]^-{p_{12}}
\\\D_0M_1\times\D_0M_1\times\D_1M_0\times\D_1M_0\ar[r]^-{p_{12}}\ar[d]_-{\D_0\gamma\times\gamma_\D}
&\D_0M_1\times\D_0M_1\ar[d]^-{\D_0\gamma}
\\\D_0M_1\times\D_1M_0\ar[r]_-{p_1}&\D_0M_1.}
\]
Now it suffices to have the right column in the above diagram compose to the identity
on $\D_0M_1$.  But the following diagram shows that the part before the final $\D_0\gamma$
coincides with $\D_0I_L$, since the right hand composite is $\id_{\D_0M_1}$ by Lemma \ref{composites}:
\[
\xymatrix{
&\D_0M_1\ar[dl]_-{\D_0T}\ar[d]^-{\cong}
\\\D_0M_0\ar[ddd]_-{\D_0I}&\D_0M_0\times\D_0M_1\ar[l]_-{p_1}\ar[dr]^-{I_\D\times1}\ar[d]_-{(\D_0I,I_\D)\times1}
\\&\D_0M_1\times\D_1M_0\times\D_0M_1\ar[r]_-{p_{23}}\ar[d]_-{1\times\chi}
&\D_1M_0\times\D_0M_1\ar[d]^-{\chi}
\\&\D_0M_1\times\D_0M_1\times\D_1M_0\ar[r]^-{p_{23}}
\ar[d]^-{p_{12}}
&\D_0M_1\times\D_1M_0\ar[d]_-{p_1}
\\\D_0M_1&\D_0M_1\times\D_0M_1\ar[l]^-{p_1}\ar[r]_-{p_2}&\D_0M_1.}
\]
Now composing, we have $\D_0\gamma\circ\D_0I_L=\id_{\D_0M_1}$ since $M$ is left unital.
This shows that $\gamma_{\Lhat M}$ is left unital after composing with $p_1$.

Showing that $\gamma_{\Lhat M}$ is left unital after composing with $p_2$
follows from the following diagram, in which we still need to verify the irregular
sub-diagram in the upper right, and as before, we suppress the subscript $\D_0M_0$'s:
\[
\xymatrix{
\D_0M_1\times\D_1M_0\ar[rr]^-{(\mu\circ\D_0S)\times1=p_2}\ar[d]_-{\cong}
&&\D_0M_0\times\D_1M_0\ar[ddd]^-{I_\D\times1}
\\\D_0M_0\times\D_0M_1\times\D_1M_0\ar[dr]^-{I_\D\times1}\ar[d]_-{(\D_0I,I_\D)\times1}
\\\D_0M_1\times\D_1M_0\times\D_0M_1\times\D_1M_0
\ar[r]_-{p_{234}}\ar[d]_-{1\times\chi\times1}
&\D_1M_0\times\D_0M_1\times\D_1M_0\ar[d]^-{\chi\times1}
\\\D_0M_1\times\D_0M_1\times\D_1M_0\times\D_1M_0
\ar[r]_-{p_{234}}\ar[d]_-{\D_0\gamma\times\gamma_\D}
&\D_0M_1\times\D_1M_0\times\D_1M_0\ar[r]_-{p_{23}}
&\D_1M_0\times\D_1M_0\ar[d]^-{\gamma_\D}
\\\D_0M_1\times\D_1M_0\ar[rr]_-{p_2}&&\D_1M_0.}
\]
The upper right sub-diagram is the product with $\D_1M_0$ of
a diagram that expresses another aspect of the interchange map $\chi$,
in which we also reflect across the main diagonal:
\[
\xymatrix{
\D_0M_1\ar[rr]^-{\cong}\ar[dr]^-{I_\D}\ar[dd]_-{\D_0S}
&&\D_0M_0\times_{\D_0M_0}\D_0M_1\ar[d]^-{I_\D\times1}
\\&\D_1M_1\ar[r]_-{\cong}\ar[d]^-{\D_1S}&\D_1M_0\times_{\D_0M_0}\D_0M_1\ar[ddd]^-{\chi}
\\\D_0^2M_0\ar[r]^-{I_\D}\ar[dr]_-{I_\D^2}\ar[dd]_-{\mu}&\D_1\D_0M_0\ar[d]^-{\D_1I_\D}
\\&\D_1^2M_0\ar[d]^-{\mu}
\\\D_0M_0\ar[r]_-{I_\D}&\D_1M_0&\D_0M_1\times_{\D_0M_0}\D_1M_0.\ar[l]^-{p_2}}
\]
We have now shown that the composition in $\Lhat M$ is left unital.

To show that the composition is right unital, we again compose with both
$p_1$ and $p_2$ and verify the resulting diagrams.  First composing with $p_2$, we
first project off the first factor by means of the following diagram:
\[
\xymatrix{
\D_0M_1\times\D_1M_0\ar[r]^-{p_2}\ar[d]_-{\cong}&\D_1M_0\ar[d]^-{\cong}
\\\D_0M_1\times\D_1M_0\times\D_0M_0\ar[r]^-{p_{23}}\ar[d]_-{1\times(\D_0I,I_\D)}
&\D_1M_0\times\D_0M_0\ar[d]^-{1\times(\D_0I,I_\D)}
\\\D_0M_1\times\D_1M_0\times\D_0M_1\times\D_1M_0\ar[r]^-{p_{234}}\ar[dd]_-{1\times\chi\times1}
&\D_1M_0\times\D_0M_1\times\D_1M_0\ar[d]^-{\chi\times1}
\\&\D_0M_1\times\D_1M_0\times\D_1M_0\ar[d]^-{p_{23}}
\\\D_0M_1\times\D_0M_1\times\D_1M_0\times\D_1M_0\ar[r]^-{p_{34}}\ar[d]_-{\D_0\gamma\times\gamma_\D}
&\D_1M_0\times\D_1M_0\ar[d]^-{\gamma_\D}
\\\D_0M_1\times\D_1M_0\ar[r]_-{p_2}&\D_1M_0.}
\]
It now suffices to show that the right column composes to the identity, but the part
before the $\gamma_\D$ coincides with the right unit map for $\D(M_0^\delta)$, because
of the following diagram, in which the right composite is the identity on $\D_1M_0$ by the
second part of Lemma \ref{composites}:
\[
\xymatrix{
&\D_1M_0\ar[dl]_-{S_\D}\ar[d]^-{\cong}
\\\D_0M_0\ar[ddd]_-{I_\D}
&\D_1M_0\times\D_0M_0
\ar[l]^-{p_2}
\ar[d]_-{1\times(\D_0I,I_\D)}
\ar[dr]^-{1\times\D_0I}
\\&\D_1M_0\times\D_0M_1\times\D_1M_0\ar[r]_-{p_{12}}\ar[d]_-{\chi\times1}
&\D_1M_0\times\D_0M_1\ar[d]^-{\chi}
\\&\D_0M_1\times\D_1M_0\times\D_1M_0\ar[r]_-{p_{12}}\ar[d]_-{p_{23}}
&\D_0M_1\times\D_1M_0\ar[d]^-{p_2}
\\\D_1M_0&\D_1M_0\times\D_1M_0\ar[l]^-{p_2}\ar[r]_-{p_1}&\D_1M_0.}
\]
The composite with $\gamma_\D$ is therefore the identity, showing that $\gamma_{\Lhat M}$
is right unital after composing with $p_2$.

Composing with $p_1$, we get a diagram similar to that for the left unital property composed
with $p_2$, namely
\[
\xymatrix{
\D_0M_1\times\D_1M_0\ar[rr]^-{1\times T_\D=p_1}\ar[d]_-{\cong}
&&\D_0M_1\times\D_0M_0\ar[ddd]^-{1\times\D_0I}
\\\D_0M_1\times\D_1M_0\times\D_0M_0\ar[d]_-{1\times(\D_0I,I_\D)}\ar[dr]^-{1\times\D_0I}
\\\D_0M_1\times\D_1M_0\times\D_0M_1\times\D_1M_0
\ar[r]_-{p_{123}}\ar[d]_-{1\times\chi\times1}
&\D_0M_1\times\D_1M_0\times\D_0M_1\ar[d]^-{1\times\chi}
\\\D_0M_1\times\D_0M_1\times\D_1M_0\times\D_1M_0\ar[r]_-{p_{123}}\ar[d]_-{\D_0\gamma\times\gamma_\D}
&\D_0M_1\times\D_0M_1\times\D_1M_0\ar[r]_-{p_{12}}
&\D_0M_1\times\D_0M_1\ar[d]^-{\D_0\gamma}
\\\D_0M_1\times\D_1M_0\ar[rr]_-{p_1}&&\D_0M_1.}
\]
Again, we need to verify the large sub-diagram in the upper right, but that follows from
the commutativity of
\[
\xymatrix{
\D_1M_0\ar[rr]^-{\cong}\ar[dd]_-{T_\D}\ar[dr]^-{\D_1I}&&\D_1M_0\times_{\D_0M_0}\D_0M_0\ar[d]^-{1\times\D_0I}
\\&\D_1M_1\ar[r]^-{\cong}\ar[d]_-{T_\D}&\D_1M_0\times_{\D_0M_0}\D_0M_1\ar[d]^-{\chi}
\\\D_0M_0\ar[r]_-{\D_0I}&\D_0M_1&\D_0M_1\times_{\D_0M_0}\D_1M_0\ar[l]^-{p_1}}
\]
after crossing on the left with $\D_0M_1$.
We conclude that $\gamma_{\Lhat M}$ is right unital.

It remains to show that $\gamma_{\Lhat M}$ is associative.  This reduces to showing 
commutativity for each of the sub-diagrams of a square diagram with four
sub-squares; however, the diagram is too large to fit onto a page, so we display the
left half and right half separately.  The left half is
\[
\xymatrix{
(\D_0M_1\times\D_1M_0)^3
\ar[r]^-{1\times\chi\times1^3}\ar[d]_-{1^3\times\chi\times1}
&(\D_0M_1)^2\times(\D_1M_0)^2\times\D_0M_1\times\D_1M_0
\ar[d]^-{1\times\chi^2\times1}
\\\D_0M_1\times\D_1M_0\times(\D_0M_1)^2\times(\D_1M_0)^2
\ar[r]^-{1\times\chi^2\times1}\ar[d]_-{1\times\D_0\gamma\times\gamma_\D}
&(\D_0M_1)^3\times(\D_1M_0)^3
\ar[d]^-{\D_0\gamma_S\times(\gamma_\D)_S}
\\\D_0M_1\times\D_1M_0\times\D_0M_1\times\D_1M_0
\ar[r]_-{1\times\chi\times1}
&(\D_0M_1)^2\times(\D_1M_0)^2,
}
\]
where as before we suppress subscript $\D_0M_0$'s, and the right half is
\[
\xymatrix@C+15pt{
(\D_0M_1)^2\times(\D_1M_0)^2\times\D_0M_1\times\D_1M_0
\ar[r]^-{\D_0\gamma\times\gamma_\D\times1}
\ar[d]_-{1\times\chi^2\times1}
&(\D_0M_1\times\D_1M_0)^2
\ar[d]^-{1\times\chi\times1}
\\
(\D_0M_1)^3\times(\D_1M_0)^3
\ar[r]^-{\D_0\gamma_T\times(\gamma_\D)_T}
\ar[d]_-{\D_0\gamma_S\times(\gamma_\D)_S}
&(\D_0M_1)^2\times(\D_1M_0)^2
\ar[d]^-{\D_0\gamma\times\gamma_\D}
\\
(\D_0M_1)^2\times(\D_1M_0)^2
\ar[r]_-{\D_0\gamma\times\gamma_\D}
&\D_0M_1\times\D_1M_0.
}
\]
The top half of the left diagram reduces to a diagram on the inner four
factors, and both ways around the square then reduce to the composite
\[
\xymatrix{
\D_1M_0\times\D_0M_1\times\D_1M_0\times\D_0M_1\ar[d]^-{\chi\times\chi}
\\\D_0M_1\times\D_1M_0\times\D_0M_1\times\D_1M_0\ar[d]^-{1\times\chi\times1}
\\\D_0M_1\times\D_0M_1\times\D_1M_0\times\D_1M_0.}
\]
The bottom half of the right diagram commutes because $M$ is 
a multicategory and $\D$ is a category object in monads on $\Set$.
This leaves us with the two other sub-squares.

We will continue to suppress subscript $\D_0M_0$'s for the remainder of this section.
For the lower left sub-square, we will need the fact,
to be verified, that the
interchange $\chi$ commutes with the composition in $M$, in the
sense that the diagram
\[
\xymatrix{
\D_1M_0\times\D_0M_1\times\D_0M_1\ar[r]^-{\chi\times1}\ar[d]_-{1\times\D_0\gamma}
&\D_0M_1\times\D_1M_0\times\D_0M_1\ar[r]^-{1\times\chi}
&\D_0M_1\times\D_0M_1\times\D_1M_0\ar[d]^-{\D_0\gamma\times1}
\\\D_1M_0\times\D_0M_1\ar[rr]_-{\chi}&&\D_0M_1\times\D_1M_0}
\]
commutes.  We need to recall that the interchange is 
given by the composite of an isomorphism to $\D_1M_1$ with the map
\[
\xymatrix{
\D_1M_1\ar[r]^-{(T_\D,\theta)}
&\D_0M_1\times_{\D_0M_0}\D_1M_0,}
\]
where the map $\theta$ is the composite
\[
\xymatrix{
\D_1M_1\ar[r]^-{\D_1S}
&\D_1\D_0M_0\ar[r]^-{\D_1I_\D}
&\D_1^2M_0\ar[r]^-{\mu}
&\D_1M_0.}
\]
We also agree to write $\theta$ for the related composite
\[
\xymatrix{
D_1M_2\ar[r]^-{\D_1S}&\D_1\D_0M_1\ar[r]^-{\D_1I_\D}&\D_1^2M_1\ar[r]^-{\mu}&\D_1M_1.}
\]
Now turning our diagram on its side, we can fill it in as follows with the sub-diagrams
still to be explained and verified:
\[
\xymatrix@C-20pt{
\D_1M_0\times\D_0M_1\times\D_0M_1
\ar[rrr]^-{1\times\D_0\gamma}\ar[dr]^-{\cong}\ar[dd]_-{\chi\times1}
&&&\D_1M_0\times\D_0M_1\ar[ddd]^-{\chi}
\\&\D_1M_2\ar[r]^-{\D_1\gamma}\ar[dr]_-{(T_\D,\theta^2)}\ar[d]_-{(T_\D\circ\D_1T,\theta)}
&\D_1M_1\ar[ur]^-{\cong}
\ar[ddr]^-{(T_\D,\theta)}
\\\D_0M_1\times\D_1M_0\times\D_0M_1
\ar[r]^-{\cong}\ar[dr]_-{1\times\chi}
&\D_0M_1\times\D_1M_1
\ar[d]^-{1\times(T_\D,\theta)}
&\D_0M_2\times\D_1M_0
\ar[dr]_-{\D_0\gamma\times1}\ar[dl]^-{\cong}
\\&\D_0M_1\times\D_0M_1\times\D_1M_0\ar[rr]_-{\D_0\gamma\times1}
&&\D_0M_1\times\D_1M_0.}
\]

The first thing we explain is the isomorphism out of
the upper left corner
\[
\D_1M_0\times\D_0M_1\times\D_0M_1\cong\D_1M_2,
\]
where we have suppressed subscript $\D_0M_0$'s.  This is a consequence
\ignore
of our known isomorphism $\D_1M_0\times_{\D_0M_0}\D_0M_1\cong\D_1M_1$, together
with the following diagram in which we claim all the squares are pullbacks:
\[
\xymatrix{
\D_1M_2\ar[r]^-{\D_1T}\ar[d]_-{\D_1S}&\D_1M_1\ar[d]^-{\D_1S}
\\\D_1\D_0M_1\ar[r]^-{\D_1\D_0T}\ar[d]_-{\D_1I_\D}&\D_1\D_0M_0\ar[d]^-{\D_1I_\D}
\\\D_1^2M_1\ar[r]^-{\D_1^2T}\ar[d]_-{\mu}&\D_1^2M_0\ar[d]^-{\mu}
\\\D_1M_1\ar[r]^-{\D_1T}\ar[d]_-{S_\D}&\D_1M_0\ar[d]^-{S_\D}
\\\D_0M_1\ar[r]_-{\D_0T}&\D_0M_0.}
\]
The top square is a pullback from the definition of $M_2$ and the fact
that $\D_1$ preserves pullbacks, the bottom square is a pullback since
$T:M_1^\delta\to M_0^\delta$ is a cover and $\D$ preserves covers,
and the square with the $\mu$'s is a pullback since $\D_1$ is Cartesian.
The remaining square is the application of $\D_1$ to a square that is
a pullback by the following lemma.

\begin{lemma}
For any map of sets $f:X\to Y$, the naturality square
\[
\xymatrix{
\D_0X\ar[r]^-{\D_0f}\ar[d]_-{I_\D}&\D_0Y\ar[d]^-{I_\D}
\\\D_1X\ar[r]_-{\D_1f}&\D_1Y}
\]
is a pullback.
\end{lemma}

\begin{proof}
Suppose given elements
\[
[\phi,x]\in\D_1X=\coprod_{n\ge0}(\DD_n)_1\times_{\Sigma_n}X^n
\quad\text{and}\quad
[a,y]\in\D_0Y=\coprod_{n\ge0}(\DD_n)_0\times_{\Sigma_n}Y^n,
\]
with the understanding that $x$ and $y$ are lists of elements, such that
$\D_1f[\phi,x]=I_\D[a,y]$, that is, 
\[
[\phi,fx]=[\id_a,y].
\]
Then since $\DD$ is $\Sigma$-free, there is a unique $\sigma\in\Sigma_n$
such that $\phi=\id_a\cdot\sigma=\id_{a\cdot\sigma}$, where the
last equality follows from $\DD$ being a $\Cat$-operad.  Now we have
\[
[\phi,fx]=[\id_{a\cdot\sigma},fx]=[\id_a,\sigma\cdot fx]
=[\id_a,f(\sigma\cdot x)].
\]
Since this coincides with $[\id_a,y]$, it follows from freeness that
$y=f(\sigma\cdot x)$.  Now we see that the element
$
[a,\sigma\cdot x]$ will map to both elements we started with,
showing existence.  Uniqueness follows from the fact that $I_\D$
is injective, having both $T_\D$ and $S_\D$ as left inverses.
\end{proof}
\endignore
of the following diagram, in which all the squares are pullbacks:
\[
\xymatrix{
\D_1M_2\ar[r]^-{\D_1T}\ar[d]_-{S_\D}
&\D_1M_1\ar[r]^-{\D_1T}\ar[d]^-{S_\D}
&\D_1M_0\ar[d]^-{S_\D}
\\\D_0M_2\ar[r]^-{\D_0T}\ar[d]_-{\D_0S}
&\D_0M_1\ar[r]_-{\D_0T}\ar[d]^-{\D_0S}
&\D_0M_0
\\\D_0^2M_1\ar[r]^-{\D_0^2T}\ar[d]_-{\mu}
&\D_0^2M_0\ar[d]^-{\mu}
\\\D_0M_1\ar[r]_-{\D_0T}&\D_0M_0.}
\]
Now since the upper right isomorphism $\D_1M_1\cong\D_1M_0\times\D_0M_1$
is given by $(\D_1T,S_\D)$, the top part of the diagram is a consequence of
the commutative squares
\[
\xymatrix{
\D_1M_2\ar[r]^-{\D_1\gamma}\ar[d]_-{\D_1T}&\D_1M_1\ar[d]^-{\D_1T}
\\\D_1M_1\ar[r]_-{\D_1T}&\D_1M_0}
\quad\text{and}\quad
\xymatrix{
\D_1M_2\ar[r]^-{\D_1\gamma}\ar[d]_-{S_\D}&\D_1M_1\ar[d]^-{S_\D}
\\\D_0M_2\ar[r]_-{\D_0\gamma}&\D_0M_1.}
\]
For the left part of the diagram, the same display of pullback squares shows
us that the inverse of the upper left isomorphism can be expressed as
\[
(\D_1T,\mu\circ\D_0S\circ S_\D):\D_1M_2\to\D_1M_1\times_{\D_0M_0}\D_0M_1,
\]
while the inverse of the isomorphism below it is induced by
\[
(\D_1T,S_\D):\D_1M_1\to\D_1M_0\times_{\D_0M_0}\D_0M_1.
\]
Meanwhile, we recall that the interchange $\chi$ is given by
\[
(T_\D,\theta):\D_1M_1\to\D_0M_1\times_{\D_0M_0}\D_1M_0.
\]
Now proceeding from $\D_1M_2$ both ways to $\D_0M_1\times\D_1M_0\times\D_0M_1$ and
projecting to the left $\D_0M_1$, we see that both ways coincide with $T_\D\circ\D_1T$.
Projecting to $\D_1M_0$, we find we require $\theta\circ T_\D=T_\D\circ\theta$,
but that is expressed by the commutative diagram
\[
\xymatrix{
\D_1M_2\ar[r]^-{D_1T}\ar[d]_-{\D_1S}&\D_1M_1\ar[d]^-{\D_1S}
\\\D_1\D_0M_1\ar[r]^-{\D_1\D_0T}\ar[d]_-{\D_1I_\D}&\D_1\D_0M_0\ar[d]^-{\D_1I_\D}
\\\D_1^2M_1\ar[r]^-{\D_1^2T}\ar[d]_-{\mu}&\D_1^2M_0\ar[d]^-{\mu}
\\\D_1M_1\ar[r]_-{\D_1T}&\D_1M_0.}
\]
And projecting to the right hand $\D_0M_1$ requires 
\[
\mu\circ\D_0S\circ S_\D=\S_\D\circ\theta,
\]
but that follows from the following commutative diagram:
\[
\xymatrix{
\D_1M_2\ar[r]^-{S_\D}\ar[d]_-{\D_1S}&\D_0M_2\ar[dd]^-{\D_0S}
\\\D_1\D_0M_1\ar[dr]^-{S_\D}\ar[d]_-{\D_0I_\D}
\\\D_1^2M_1\ar[r]_-{S_\D^2}\ar[d]_-{\mu}&\D_0^2M_1\ar[d]^-{\mu}
\\\D_1M_1\ar[r]_-{S_\D}&\D_0M_1.}
\]
The left quadrilateral of the diagram therefore commutes.

The bottom left triangle commutes by the definition of $\chi$.

We come next to the (somewhat distorted) square 
\[
\xymatrix@C+20pt{
\D_1M_2\ar[r]^-{(T_\D,\theta^2)}\ar[d]_-{(T_\D\circ\D_1T,\theta)}&\D_0M_2\times\D_1M_0\ar[d]^-{\cong}
\\\D_0M_1\times\D_1M_1\ar[r]_-{1\times(T_\D,\theta)}&\D_0M_1\times\D_0M_1\times\D_1M_0.}
\]
The right vertical isomorphism is a consequence of the diagram
of pullback squares
\[
\xymatrix{
\D_0M_2\ar[r]^-{\D_0T}\ar[d]_-{\D_0S}&\D_0M_1\ar[d]^-{\D_0S}
\\\D_0^2M_1\ar[r]^-{\D_0^2T}\ar[d]_-{\mu}&\D_0^2M_0\ar[d]^-{\mu}
\\\D_0M_1\ar[r]_-{\D_0T}&\D_0M_0,}
\]
so can be expressed as
\[
(D_0T,\mu\circ\D_0S):\D_0M_2\to\D_0M_1\times_{\D_0M_0}\D_0M_1.
\]
We see immediately that the desired square commutes after projecting to
the first $\D_0M_1$ or the $\D_1M_0$.  Projecting to the middle $\D_0M_1$, we 
require
\[
\mu\circ\D_0S\circ T_\D=T_\D\circ\theta,
\]
but that follows from the commutativity of
\[
\xymatrix{
\D_1M_2\ar[r]^-{T_\D}\ar[d]_-{\D_1S}&\D_0M_2\ar[dd]^-{\D_0S}
\\\D_1\D_0M_1\ar[dr]^-{T_\D}\ar[d]_-{\D_1I_\D}
\\\D_1^2M_1\ar[r]_-{T_\D^2}\ar[d]_-{\mu}&\D_0^2M_1\ar[d]^-{\mu}
\\\D_1M_1\ar[r]_-{T_\D}&\D_0M_1.}
\]

The triangle at the bottom of our desired diagram simply expresses the
definition of $\D_0\gamma$.

Now we have another somewhat distorted square to verify, namely
\[
\xymatrix@C+15pt{
\D_1M_2\ar[r]^-{\D_1\gamma}\ar[d]_-{(T_\D,\theta^2)}&\D_1M_1\ar[d]^-{(T_\D,\theta)}
\\\D_0M_2\times\D_1M_0\ar[r]_-{\D_0\gamma\times1}&\D_0M_1\times\D_1M_0,}
\]
as usual with subscript $\D_0M_0$'s suppressed.  Projecting to the factor
of $\D_0M_1$, this is just the naturality of $T_\D$ with respect to $\gamma:M_2\to M_1$.
Projecting to the factor of $\D_1M_0$, we need to verify that the square
\[
\xymatrix{
\D_1M_2\ar[r]^-{\D_1\gamma}\ar[d]_-{\theta}&\D_1M_1\ar[d]^-{\theta}
\\\D_1M_1\ar[r]_-{\theta}&\D_1M_0}
\]
commutes.  This follows from its expansion as follows, using the definition of $\theta$:
\[
\xymatrix{
\D_1M_2\ar[rrr]^-{\D_1\gamma}\ar[d]_-{\D_1S}
&&&\D_1M_1\ar[d]^-{\D_1S}
\\\D_1\D_0M_1\ar[r]^-{\D_1\D_0S}\ar[d]_-{\D_1I_\D}
&\D_1\D_0^2M_0\ar[rr]^-{\D_1\mu}\ar[d]^-{\D_1I_\D}
&&\D_1\D_0M_0\ar[d]^-{\D_1I_\D}
\\\D_1^2M_1\ar[r]^-{\D_1^2S}\ar[d]_-{\mu}
&\D_1^2M_0\ar[r]^-{\D_1^2I_\D}\ar[d]^-{\mu}
&\D_1^3M_0\ar[r]^-{\D_1\mu}\ar[d]_-{\mu}
&\D_1^2M_0\ar[d]^-{\mu}
\\\D_1M_1\ar[r]_-{\D_1S}
&\D_1\D_0M_0\ar[r]_-{\D_1I_\D}
&\D_1^2M_0\ar[r]_-{\mu}
&\D_1M_0.}
\]
The final right hand triangle simply exhibits the definition of $\chi$.  This completes
the verification that $\chi$ respects the multiplication $\D_0\gamma$.

We return now to the lower part of the left half of the associativity diagram,
\[
\xymatrix{
\D_0M_1\times\D_1M_0\times(\D_0M_1)^2\times(\D_1M_0)^2
\ar[r]^-{1\times\chi^2\times1}\ar[d]_-{1\times\D_0\gamma\times\gamma_\D}
&(\D_0M_1)^3\times(\D_1M_0)^3
\ar[d]^-{\D_0\gamma_S\times(\gamma_\D)_S}
\\\D_0M_1\times\D_1M_0\times\D_0M_1\times\D_1M_0
\ar[r]_-{1\times\chi\times1}
&(\D_0M_1)^2\times(\D_1M_0)^2,
}
\]
and note that $\D_0\gamma_S$ is really just
\[
1\times\D_0\gamma:\D_0M_1\times(\D_0M_1)^2\to\D_0M_1\times\D_0M_1,
\]
where we have again suppressed subscript $\D_0M_0$'s.  Similarly, $(\gamma_\D)_S$
is also really just $1\times\gamma_\D$, so this diagram is just the identity on the left
$\D_0M_1$, $\gamma_\D$ on the right two $\D_1M_0$'s, and the rest is the 
diagram we just verified.  This concludes the verification that the lower left square of
the associativity diagram commutes.

For the upper right square, 
\[
\xymatrix@C+15pt{
(\D_0M_1)^2\times(\D_1M_0)^2\times\D_0M_1\times\D_1M_0
\ar[r]^-{\D_0\gamma\times\gamma_\D\times1}
\ar[d]_-{1\times\chi^2\times1}
&(\D_0M_1\times\D_1M_0)^2
\ar[d]^-{1\times\chi\times1}
\\
(\D_0M_1)^3\times(\D_1M_0)^3
\ar[r]^-{\D_0\gamma_T\times(\gamma_\D)_T}
&(\D_0M_1)^2\times(\D_1M_0)^2,
}
\]
we note that, similarly to the above, 
\[
\D_0\gamma_T=\D_0\gamma\times1
\quad\text{and}\quad
(\gamma_\D)_T=\gamma_\D\times1.
\]
Consequently the diagram decomposes to just $\D_0\gamma$ on the 
first two factors, and $\id_{\D_1M_0}$ on the last one, leaving us with the
diagram
\[
\xymatrix{
\D_1M_0\times\D_1M_0\times\D_0M_1\ar[r]^-{\gamma_\D\times1}\ar[d]_-{1\times\chi}
&\D_1M_0\times\D_0M_1\ar[dd]^-{\chi}
\\\D_1M_0\times\D_0M_1\times\D_1M_0\ar[d]_-{\chi\times1}
\\\D_0M_1\times\D_1M_0\times\D_1M_0\ar[r]_-{1\times\gamma_\D}
&\D_0M_1\times\D_1M_0}
\]
to verify.  We fill it in with an interior analogous to the one we used
for the lower left square:
\[
\xymatrix@C-20pt{
\D_1M_0\times\D_1M_0\times\D_0M_1\ar[rrr]^-{\gamma_\D\times1}
\ar[dd]_-{1\times\chi}
&&&\D_1M_0\times\D_0M_1\ar[ddd]^-{\chi}
\\&\D_2M_1
\ar[ul]_-{\cong}\ar[r]^-{\gamma_\D}
\ar[dr]^-{(T_\D^2,\hat{\theta})}
\ar[d]_-{(T_\D,\theta\circ S_\D)}
&\D_1M_1\ar[ur]^-{\cong}\ar[ddr]^-{(T_\D,\theta)}
\\\D_1M_0\times\D_0M_1\times\D_1M_0\ar[dr]_-{\chi\times1}
&\D_1M_1\times\D_1M_0\ar[l]_-{\cong}\ar[d]^-{(T_\D,\theta)\times1}
&\D_0M_1\times\D_2M_0\ar[dl]_-{\cong}\ar[dr]_-{1\times\gamma_\D}
\\&\D_0M_1\times\D_1M_0\times\D_1M_0
\ar[rr]_-{1\times\gamma_\D}
&&\D_0M_1\times\D_1M_0.}
\]
Here we use the notation $\hat\theta$ for the composite
\[
\xymatrix{
\D_2M_1\ar[r]^-{\D_2S}&\D_2\D_0M_0\ar[r]^-{\D_2I_\D^2}&\D_2^2M_0\ar[r]^-{\mu}&\D_2M_0,}
\]
and the map $I_\D^2:\D_0M_0\to\D_2M_0$ means either of the coincident composites
\[
\xymatrix{
\D_0M_0\ar[r]^-{I_\D}&\D_1M_0\ar[r]^-{(I_\D)_L}&\D_2M_0}
\quad\text{or}\quad
\xymatrix{
\D_0M_0\ar[r]^-{I_\D}&\D_1M_0\ar[r]^-{(I_\D)_R}&\D_2M_0.}
\]
(The composites coincide since $\D$ is a category object.)  

We proceed to 
verify the sub-diagrams.  The upper left isomorphism is a consequence of
the diagram of pullback squares
\[
\xymatrix{
\D_2M_1\ar[r]^-{\D_2T}\ar[d]_-{S_\D}
&\D_2M_0\ar[r]^-{T_\D}\ar[d]^-{S_\D}
&\D_1M_0\ar[d]^-{S_\D}
\\\D_1M_1\ar[r]^-{\D_1T}\ar[d]_-{S_\D}
&\D_1M_0\ar[r]_-{T_\D}\ar[d]^-{S_\D}
&\D_0M_0
\\\D_0M_1\ar[r]_-{D_0T}&\D_0M_0,}
\]
where the right hand square is a pullback since $\D$ is a category object,
the lower left square is a pullback since $T:M_1^\delta\to M_0^\delta$ is a cover
and $\D$ preserves covers, and the upper left square is a pullback from an application
of the source cover version of Corollary \ref{coverpullbacks}  to the cover $\D T:\D(M_1^\delta)\to\D(M_0^\delta)$.

Now the top part of our desired diagram is a consequence of the two commutative
squares
\[
\xymatrix{
\D_2M_1\ar[r]^-{\D_2T}\ar[d]_-{\gamma_\D}&\D_2M_0\ar[d]^-{\gamma_\D}
\\\D_1M_1\ar[r]_-{\D_1T}&\D_1M_0}
\quad\text{and}\quad
\xymatrix{
\D_2M_1\ar[r]^-{S_\D}\ar[d]_-{\gamma_\D}&\D_1M_1\ar[d]^-{S_\D}
\\\D_1M_1\ar[r]_-{S_\D}&\D_0M_1.}
\]

For the left part of the diagram, we first use the display of pullbacks giving
the upper left isomorphism to rewrite it as
\[
(T_\D\circ\D_2T,S_\D):\D_2M_1\to\D_1M_0\times_{\D_0M_0}\D_1M_1.
\]
Also recalling again that $\chi$ is induced by the map
\[
(T_\D,\theta):\D_1M_1\to\D_0M_1\times_{\D_0M_0}\D_1M_0,
\]
we verify the left part after projecting to each of the three factors
in the lower left corner.  Projecting to the right hand $\D_1M_0$ gives
in each case $\theta\circ S_\D$, so that checks.  Projecting to
the left hand $\D_1M_0$ requires 
\[
T_\D\circ\D_2T=\D_1T\circ S_\D,
\]
but that follows from the naturality of $S_\D$.  And projecting to $\D_0M_1$,
we require $T_\D\circ S_\D=S_\D\circ T_\D$, but that is a consequence
of the defining diagram for $\D_2M_1$.  The left part of the diagram
therefore commutes.

For the square
\[
\xymatrix@C+20pt{
\D_2M_1\ar[r]^-{(T_\D^2,\hat\theta)}\ar[d]_-{(T_\D,\theta\circ S_\D)}
&\D_0M_1\times\D_2M_0\ar[d]^-{\cong}
\\\D_1M_1\times\D_1M_0\ar[r]_-{(T_\D,\theta)\times1}
&\D_0M_1\times\D_1M_0\times\D_1M_0,}
\]
the right hand isomorphism is induced by the isomorphism 
\[
(T_\D,S_\D):\D_2M_0\to\D_1M_0\times_{\D_0M_0}\D_1M_0,
\]
so we can again verify commutativity by checking after projection to
each of the three factors on the lower right.  Projecting to the $\D_0M_1$, 
both composites are $T_\D^2$.  Projecting to the right hand $\D_1M_0$, 
we require $\theta\circ S_\D=S_\D\circ\hat\theta$.  This follows from the 
following diagram, in which we are careful to use $(I_\D)_L$ rather
than $(I_\D)_R$ so that the inner triangle commutes:
\[
\xymatrix{
\D_2M_1\ar[r]^-{S_\D}\ar[d]_-{\D_2S}
&\D_1M_1\ar[d]^-{\D_1S}
\\\D_2\D_0M_0\ar[r]^-{S_\D}\ar[d]_-{\D_2I_\D}
&\D_1\D_0M_0\ar[dd]^-{\D_1I_\D}
\\\D_2\D_1M_0\ar[dr]^-{S_\D}\ar[d]_-{\D_2(I_\D)_L}
\\\D_2^2M_0\ar[r]_-{S_\D^2}\ar[d]_-{\mu}
&\D_1^2M_0\ar[d]^-{\mu}
\\\D_2M_0\ar[r]_-{S_\D}&\D_1M_0.}
\]
And projecting to the middle $\D_1M_0$ requires us to verify
$\theta\circ T_\D=T_\D\circ\hat\theta$.  This follows from the following diagram,
in which we now use $(I_\D)_R$ so the triangle commutes:
\[
\xymatrix{
\D_2M_1\ar[r]^-{T_\D}\ar[d]_-{\D_2S}
&\D_1M_1\ar[d]^-{\D_1S}
\\\D_2\D_0M_0\ar[r]^-{T_\D}\ar[d]_-{\D_2I_\D}
&\D_1\D_0M_0\ar[dd]^-{\D_1I_\D}
\\\D_2\D_1M_0\ar[dr]^-{T_\D}\ar[d]_-{\D_2(I_\D)_R}
\\\D_2^2M_0\ar[r]_-{T_\D^2}\ar[d]_-{\mu}
&\D_1^2M_0\ar[d]^-{\mu}
\\\D_2M_0\ar[r]_-{T_\D}&\D_1M_0.}
\]

We next have to verify the square
\[
\xymatrix{
\D_2M_1\ar[r]^-{\gamma_\D}\ar[d]_-{(T_\D^2,\hat\theta)}
&\D_1M_1\ar[d]^-{(T_\D,\theta)}
\\\D_0M_1\times\D_2M_0\ar[r]_-{1\times\gamma_\D}
&\D_0M_1\times\D_1M_0.}
\]
Projecting to $\D_0M_1$ reduces to the commutative square
\[
\xymatrix{
\D_2M_1\ar[r]^-{\gamma_\D}\ar[d]_-{T_\D}
&\D_1M_1\ar[d]^-{T_\D}
\\\D_1M_1\ar[r]_-{T_\D}&\D_0M_1,}
\]
while projecting to $\D_1M_0$ requires $\theta\circ\gamma_\D
=\gamma_\D\circ\hat\theta$.  This follows from the following diagram,
in which the triangle commutes since $\gamma_\D$ is unital (we could
use either $(I_\D)_L$ or $(I_\D)_R$), and the bottom square is
a result of $\D$ being a category object in monads on $\Set$:
\[
\xymatrix{
\D_2M_1\ar[r]^-{\gamma_\D}\ar[d]_-{\D_2S}
&\D_1M_1\ar[d]^-{\D_1S}
\\\D_2\D_0M_0\ar[r]^-{\gamma_\D}\ar[d]_-{\D_2I_\D}
&\D_1\D_0M_0\ar[dd]^-{\D_1I_\D}
\\\D_2\D_1M_0\ar[dr]^-{\gamma_\D}\ar[d]_-{\D_2(I_\D)_R}
\\\D_2^2M_0\ar[r]_-{\gamma_\D^2}\ar[d]_-{\mu}
&\D_1^2M_0\ar[d]^-{\mu}
\\\D_2M_0\ar[r]_-{\gamma_\D}&\D_1M_0.}
\]
The left triangle simply records the definition of $\chi$.  We have
completed the verification of the upper right square of the associativity
diagram, and therefore completed the verification that the composition
on $\Lhat M$ is associative.

\section{The provisional left adjoint: $\D$-algebra structure}

In this section we specify a $\D$-algebra structure on $\Lhat M$
by specifying a $\D_0$-action on $\Lhat M_0=\D_0M_0$ and
a $\D_1$-action on $\Lhat M_1$.  We then verify that the
structure maps for a category are preserved, so we actually get
an action $\D\Lhat M\to\Lhat M$ that is a functor.  

First, since $\Lhat M_0=\D_0M_0$ is the free $\D_0$-algebra
on $M_0$, we use that as its algebra structure over $\D_0$.  
Explicitly, we have an action map given by
\[
\xymatrix{
\D_0(\Lhat M_0)=\D_0^2M_0\ar[r]^-{\mu}&\D_0M_0.}
\]

We specify a $\D_1$-action on $\Lhat M_1=\D_0M_1\times_{\D_0M_0}\D_1M_0$
by giving its projections to the two factors $\D_0M_1$ and $\D_1M_0$, and
then verifying that we do get a map to the pullback.  Note that since
$\D_1$ preserves pullbacks, we have
\[
\D_1(\Lhat M_1)\cong\D_1(\D_0M_1\times_{\D_0M_0}\D_1M_0)
\cong\D_1\D_0M_1\times_{\D_1\D_0M_0}\D_1^2M_0.
\]
We specify a map $\xi_1:\D_1\Lhat M_1\to\Lhat M_1$ by specifying
$p_1\circ\xi_1:\D_1\Lhat M_1\to\D_0M_1$ and 
$p_2\circ\xi_1:\D_1\Lhat M_1\to\D_1M_0$.  For the first, we
specify $p_1\circ\xi_1$ to be given by the composite
\[
\xymatrix{
\D_1\Lhat M_1\cong\D_1\D_0M_1\times_{\D_1\D_0M_0}\D_1^2M_0
\ar[r]^-{p_1}&\D_1\D_0M_1\ar[r]^-{T_\D}&\D_0^2M_1\ar[r]^-{\mu}&\D_0M_1,}
\]
and for the second we specify $p_2\circ\xi_1$ to be given by the composite
\[
\xymatrix{
\D_1\Lhat M_1\cong\D_1\D_0M_1\times_{\D_1\D_0M_0}\D_1^2M_0
\ar[r]^-{p_2}&\D_1^2M_0\ar[r]^-{\mu}&\D_1M_0.}
\]
This does give us a map to $\Lhat M_1=\D_0M_1\times_{\D_0M_0}\D_1M_0$
because of the following commuting diagram; note that the left rectangle is 
simply $\D_1$ applied to the pullback defining $\Lhat M_1$:
\[
\xymatrix{
\D_1\Lhat M_1\ar[r]^-{p_1}\ar[dd]_-{p_2}
&\D_1\D_0M_1\ar[r]^-{T_\D}\ar[d]_-{\D_1\D_0S}
&\D_0^2M_1\ar[r]^-{\mu}\ar[d]_-{\D_0^2S}
&\D_0M_1\ar[d]^-{\D_0S}
\\&\D_1\D_0^2M_0\ar[r]^-{T_\D}\ar[d]_-{\D_1\mu}
&\D_0^3M_0\ar[r]^-{\mu}\ar[d]_-{\D_0\mu}
&\D_0^2M_0\ar[d]^-{\mu}
\\\D_1^2M_0\ar[r]^-{\D_1T_\D}\ar[drr]_-{\mu}
&\D_1\D_0M_0\ar[r]^-{T_\D}
&\D_0^2M_0\ar[r]^-{\mu}
&\D_0M_0.
\\&&\D_1M_0\ar[ur]_-{T_\D}}
\]

We wish to show that this really is an action, that is, that $\xi_1$ is unital
and associative.  To be unital, the diagram
\[
\xymatrix{
\Lhat M_1\ar[r]^-{\eta}\ar[dr]_-{=}&\D_1\Lhat M_1\ar[d]^-{\xi_1}
\\&\Lhat M_1}
\]
must commute, but that will follow if the diagram commutes after 
composing with both projections $p_1$ and $p_2$ with targets
$\D_0M_1$ and $\D_1M_0$ respectively.  For composition with $p_1$, we
have the diagram
\[
\xymatrix{
\Lhat M_1\ar[r]^-{\eta}\ar[dd]_-{p_1}&\D_1\Lhat M_1\ar[d]^-{\D_1p_1}
\\&\D_1\D_0M_1\ar[d]^-{T_\D}
\\\D_0M_1\ar[ur]^-{\eta_{\D_1}}\ar[r]_-{\eta_{\D_0}}\ar[dr]_-{=}
&\D_0^2M_1\ar[d]^-{\mu}
\\&\D_0M_1,}
\]
where the central triangle commutes since $T_\D:\D_1\to\D_0$ is a map
of monads.  For composition with $p_2$, we have the commuting diagram
\[
\xymatrix{
\Lhat M_1\ar[r]^-{\eta}\ar[d]_-{p_2}&\D_1\Lhat M_1\ar[d]^-{\D_1p_2}
\\\D_1M_0\ar[r]^-{\eta}\ar[dr]_-{=}&\D_1^2M_0\ar[d]^-{\mu}
\\&\D_1M_0,}
\]
and it follows that $\xi_1$ is unital.  

To show that the action is associative, we need to show that the diagram
\[
\xymatrix{
\D_1^2\Lhat M_1\ar[r]^-{\D_1\xi_1}\ar[d]_-{\mu}&\D_1\Lhat M_1\ar[d]^-{\xi_1}
\\\D_1\Lhat M_1\ar[r]_-{\xi_1}&\Lhat M_1}
\]
commutes, which again happens if and only if it does after composing
with $p_1$ and $p_2$.  Composing with $p_1$, we find the diagram
commutes as a result of the following:
\[
\xymatrix{
&&\D_1\Lhat M_1\ar[dr]^-{\D_1p_1}
\\\D_1^2\Lhat M_1\ar[urr]^-{\D_1\xi_1}\ar[r]_-{\D_1^2p_1}\ar[dd]_-{\mu}
&\D_1^2\D_0M_1\ar[r]_-{\D_1T_\D}\ar[dd]_-{\mu}
&\D_1\D_0^2M_1\ar[r]_-{\D_1\mu}\ar[d]_-{T_\D}
&\D_1\D_0M_1\ar[d]^-{T_\D}
\\&&\D_0^3M_1\ar[r]^-{\D_0\mu}\ar[d]_-{\mu}
&\D_0^2M_1\ar[d]^-{\mu}
\\\D_1\Lhat M_1\ar[r]_-{\D_1p_1}
&\D_1\D_0M_1\ar[r]_-{T_\D}
&\D_0^2M_1\ar[r]_-{\mu}
&\D_0M_1.}
\]
And composing with $p_2$, we have the following commutative diagram:
\[
\xymatrix{
\D_1^2\Lhat M_1\ar[rr]^-{\D_1\xi_1}\ar[dr]^-{\D_1^2p_2}\ar[dd]_-{\mu}
&&\D_1\Lhat M_1\ar[d]^-{\D_1p_2}
\\&\D_1^3M_0\ar[r]^-{\D_1\mu}\ar[d]_-{\mu}
&\D_1^2M_0\ar[d]^-{\mu}
\\\D_1\Lhat M_1\ar[r]_-{\D_1p_2}
&\D_1^2M_0\ar[r]_-{\mu}
&\D_1M_0.}
\]
It now follows that $\xi_1$ is an action of the monad $\D_1$ on $\Lhat M_1$.

We need to show that the actions $\xi_0$ and $\xi_1$ preserve 
the identity, source, target, and composition maps, so that we
actually get a functor $\D\Lhat M\to\Lhat M$.  For preservation of $I$, 
we need the diagram
\[
\xymatrix{
\D_0\Lhat M_0\ar[r]^-{\xi_0}\ar[d]_-{I_\D I_{\Lhat M}}&\Lhat M_0\ar[d]^-{I_{\Lhat M}}
\\\D_1\Lhat M_1\ar[r]_-{\xi_1}&\Lhat M_1}
\]
to commute.  Since the target of the diagram is the pullback 
$\Lhat M_1=\D_0M_1\times_{\D_0M_0}\D_1M_0$, we project onto each
factor and verify the resulting diagrams.  Composing with $p_1$ and using
the defining property that $p_1\circ\xi_1=\mu\circ T_\D\circ p_1$, as well as
$\D_0\Lhat M_0=\D_0^2M_0$ and $p_1\circ I_{\Lhat M}=\D_0I$, the result
for $p_1$ follows from the commutative diagram
\[
\xymatrix{
\D_0^2M_0\ar[d]_-{I_\D}\ar[drr]^-{=}
\\\D_1\D_0M_0\ar[rr]^-{T_\D}\ar[dr]^-{\D_1\D_0I}\ar[d]_-{\D_1I_{\Lhat M}}
&&\D_0^2M_0\ar[d]_-{\D_0^2I}\ar[r]^-{\mu}
&\D_0M_0\ar[d]^-{\D_0I}
\\\D_1\Lhat M_1\ar[r]_-{\D_1p_1}
&\D_1\D_0M_1\ar[r]_-{T_\D}
&\D_0^2M_1\ar[r]_-{\mu}
&\D_0M_1.}
\]
Composing with $p_2$ and recalling that $p_2\circ\xi_1=\mu\circ p_2$,
we wish the perimeter of
\[
\xymatrix{
\D_0^2M_0\ar[rrr]^-{\mu}\ar[dd]_-{I_\D}\ar[dr]^-{I_\D I_{\Lhat M}}
&&&\D_0M_0\ar[dl]_-{I_{\Lhat M}}\ar[dd]^-{I_\D}
\\&\D_1\Lhat M_1\ar[r]^-{\xi_1}\ar[d]^-{p_2}
&\Lhat M_1\ar[dr]^-{p_2}
\\\D_1\D_0M_0\ar[r]_-{\D_1I_\D}&\D_1^2M_0\ar[rr]_-{\mu}&&\D_1M_0}
\]
to commute, but that is simply a property of a a category object in monads
on $\Cat$.  It follows that our action maps preserve the identity structure maps.

To show that our action maps preserve the source structure maps, we
recall that the source map for $\Lhat M$ is given by
\[
\xymatrix{
\Lhat M_1\cong\D_0M_1\times_{\D_0M_0}\D_1M_0\ar[r]^-{p_2}
&\D_1M_0\ar[r]^-{S_\D}&\D_0M_0,}
\]
and we wish
\[
\xymatrix{
\D_1\Lhat M_1\ar[r]^-{\xi_1}\ar[d]_-{S_\D S_{\Lhat M}}
&\Lhat M_1\ar[d]^-{S_{\Lhat M}}
\\\D_0^2M_0\ar[r]_-{\mu}&\D_0M_0}
\]
to commute.  But this expands to
\[
\xymatrix{
\D_1\Lhat M_1\ar[rr]^-{\xi_1}\ar[d]_-{\D_1p_2}
&&\Lhat M_1\ar[d]^-{p_2}
\\\D_1^2M_0\ar[rr]_-{\mu}\ar[d]_-{\D_1S_\D}
&&\D_1M_0\ar[d]^-{S_\D}
\\\D_1\D_0M_0\ar[r]_-{S_\D}
&\D_0^2M_0\ar[r]_-{\mu}&\D_0M_0,}
\]
in which the top square commutes by the definition of $\xi_1$ and
the bottom rectangle because we have a category object in monads
on $\Cat$.

For preservation of the target structure maps, recall that the target
on $\Lhat M$ is given by the composite
\[
\xymatrix{
\Lhat M_1\cong\D_0M_1\times_{\D_0M_0}\D_0M_1\ar[r]^-{p_1}
&\D_0M_1\ar[r]^-{\D_0T}&\D_0M_0}
\]
Then the desired square
\[
\xymatrix{
\D_1\Lhat M_1\ar[r]^-{\xi_1}\ar[d]_-{T_\D T_{\Lhat M}}
&\Lhat M_1\ar[d]^-{T_{\Lhat M}}
\\\D_0^2M_0\ar[r]_-{\mu}&\D_0M_0}
\]
expands to
\[
\xymatrix{
\D_1\Lhat M_1\ar[rr]^-{\xi_1}\ar[d]_-{\D_1p_1}
&&\Lhat M_1\ar[d]^-{p_1}
\\\D_1\D_0M_1\ar[r]^-{T_\D}\ar[d]_-{\D_1\D_0T}
&\D_0^2M_1\ar[r]^-{\mu}\ar[d]_-{\D_0^2T}
&\D_0M_1\ar[d]^-{\D_0T}
\\\D_1\D_0M_0\ar[r]_-{T_\D}
&\D_0^2M_0\ar[r]_-{\mu}
&\D_0M_0,}
\]
in which the top part follows from our definition of $\xi_1$, and
the bottom from having a category object in monads in $\Cat$.

It remains to show that the actions commute with the composition
map $\gamma_{\Lhat M}:\Lhat M_2\to\Lhat M_1$.  In order for 
this to make sense, we need an action of $\D_2$ on $\Lhat M_2$; 
this is given by a map $\xi_2:\D_2\Lhat M_2\to\Lhat M_2$ such that the two diagrams
\[
\xymatrix@C+25pt{
\D_2\Lhat M_2\ar[r]^-{\xi_2}\ar@<-.5ex>[d]_-{S_\D S_{\Lhat M}}\ar@<.5ex>[d]^-{T_\D T_{\Lhat M}}
&\Lhat M_2\ar@<-.5ex>[d]_-{S_{\Lhat M}}\ar@<.5ex>[d]^-{T_{\Lhat M}}
\\\D_1\Lhat M_1\ar[r]_-{\xi_1}&\Lhat M_1}
\]
commute.  We then want the diagram
\[
\xymatrix{
\D_2\Lhat M_2\ar[r]^-{\xi_2}\ar[d]_-{\gamma_\D\gamma_{\Lhat M}}
&\Lhat M_2\ar[d]^-{\gamma_{\Lhat M}}
\\\D_1\Lhat M_1\ar[r]_-{\xi_1}&\Lhat M_1}
\]
to commute.  For this purpose, we want an explicit expression for $\xi_2$, which
is given by the following lemma.

\begin{lemma}
The action map $\xi_2:\D_2\Lhat M_2\to\Lhat M_2$ can be expressed as the composite
\[
\xymatrix{
\D_2\D_0M_1\times_{\D_2\D_0M_0}\D_2\D_1M_1\times_{\D_2\D_0M_0}\D_2\D_1M_0
\ar[d]^-{T_\D^2\times_{T_\D^2}T_\D\times_{T_\D S_\D}S_\D}
\\\D_0^2M_1\times_{\D_0^2M_0}\D_1^2M_1\times_{\D_0^2M_0}\D_1^2M_0
\ar[d]^-{\mu\times_\mu\mu\times_\mu\mu}
\\\D_0M_1\times_{\D_0M_0}\D_1M_1\times_{\D_0M_0}\D_1M_0.}
\]
\end{lemma}

\begin{proof}
This is true if and only if the two diagrams defining $\xi_2$ commute with this expression
in place of $\xi_2$, and in turn, each of those diagrams commute if and only if they commute
when composed with each of $p_1:\Lhat M_1\to\D_0M_1$ and $p_2:\Lhat M_1\to\D_1M_0$.
We can expand $S, T:\Lhat M_2\to\Lhat M_1$ as in the following two diagrams:
\[
\xymatrix{
\D_0M_1\times_{\D_0M_0}\D_1M_1\times_{\D_0M_0}\D_1M_0
\ar[r]^-{=}\ar[d]_-{p_{23}}
&\Lhat M_2\ar[dd]^-{S}
\\\D_1M_1\times_{\D_0M_0}\D_1M_0
\ar[d]_-{S_\D\times1}
\\\D_0M_1\times_{\D_0M_0}\D_1M_0\ar[r]_-{=}&\Lhat M_1}
\]
and
\[
\xymatrix{
\D_0M_1\times_{\D_0M_0}\D_1M_1\times_{\D_0M_0}\D_1M_0
\ar[r]^-{=}\ar[d]_-{p_{12}}
&\Lhat M_2\ar[dd]^-{T}
\\\D_0M_1\times_{\D_0M_0}\D_1M_1
\ar[d]_-{1\times\D_1T}
\\\D_0M_1\times_{\D_0M_0}\D_1M_0\ar[r]_-{=}&\Lhat M_1.}
\]
Consequently, we can express four of our composites of interest as follows:
\begin{gather*}
\xymatrix{
p_1\circ T:\Lhat M_2\ar[r]^-{p_1}&\D_0M_1,}
\\
\xymatrix{
p_2\circ T:\Lhat M_2\ar[r]^-{p_2}&\D_1M_1\ar[r]^-{\D_1T}&\D_1M_0,}
\\
\xymatrix{
p_1\circ S:\Lhat M_2\ar[r]^-{p_2}&\D_1M_1\ar[r]^-{S_\D}&\D_0M_1,}
\\
\xymatrix{
p_2\circ S:\Lhat M_2\ar[r]^-{p_2}&\D_1M_0.}
\end{gather*}
Further, we already have expressions for $p_1\circ\xi_1$ and $p_2\circ\xi_1$, namely
\begin{gather*}
\xymatrix{
p_1\circ\xi_1:\D_1\Lhat M_1\ar[r]^-{\D_1p_1}&\D_1\D_0M_1\ar[r]^-{T_\D}
&\D_0^2M_1\ar[r]^-{\mu}&\D_0M_1,}
\\
\xymatrix{
p_2\circ\xi_1:\D_1\Lhat M_1\ar[r]^-{\D_1p_2}&\D_1^2M_0\ar[r]^-{\mu}&\D_1M_0.}
\end{gather*}
These expressions allow us to verify the four diagrams we want, which we take in
order.  First, the one for $p_2$ and $S$ becomes
\[
\xymatrix{
\D_2\Lhat M_2\ar[dr]^-{p_3}\ar[d]_-{p_{23}}
\\\D_2\D_1M_1\times_{\D_2\D_0M_0}\D_2\D_1M_0
\ar[r]_-{p_2}\ar[d]_-{S_\D\times_{S_\D}S_\D}
&\D_2\D_1M_0\ar[dd]^-{S_\D}
\\\D_1^2M_2\times_{\D_1\D_0M_0}\D_1^2M_0\ar[d]_-{\D_1S_\D\times1}
\\\D_1\D_0M_1\times_{\D_1\D_0M_0}\D_1^2M_0\ar[r]_-{p_2}
&\D_1^2M_0\ar[r]_-{\mu}
&\D_1M_0.}
\]
The one for $p_1$ and $S$ becomes
\[
\xymatrix{
\D_2\Lhat M_2\ar[r]^-{p_2}\ar[d]_-{p_{23}}
&\D_2\D_1M_1\ar[r]^-{T_\D}\ar[ddd]^-{S_\D^2}
&\D_1^2M_1\ar[r]^-{\mu}\ar[ddd]^-{S_\D^2}
&\D_1M_1\ar[ddd]^-{S_\D}
\\\D_2\D_1M_1\times_{\D_2\D_0M_0}\D_2\D_1M_0
\ar[ur]^-{p_2}\ar[d]_-{S_\D\times_{S_\D}S_\D}
\\\D_1^2M_1\times_{\D_1\D_0M_0}\D_1^2M_0
\ar[d]_-{\D_1S_\D\times1}
\\\D_1\D_0M_1\times_{\D_1\D_0M_0}\D_1^2M_0
\ar[r]_-{\D_1p_1}
&\D_1\D_0M_1\ar[r]_-{T_\D}
&\D_0^2M_1\ar[r]_-{\mu}
&\D_0M_1.}
\]
The one for $p_2$ and $T$ becomes
\[
\xymatrix{
\D_2\Lhat M_2\ar[dr]^-{p_2}\ar[d]_-{p_{12}}
\\\D_2\D_0M_1\times_{\D_2\D_0M_0}\D_2\D_1M_1
\ar[r]_-{p_2}\ar[d]_-{T_\D\times_{T_\D}T_\D}
&\D_2\D_1M_1\ar[d]^-{T_\D}
\\\D_1\D_0M_1\times_{\D_1\D_0M_0}\D_1^2M_1
\ar[r]^-{p_2}\ar[d]_-{1\times\D_1^2T}
&\D_1^2M_1\ar[r]^-{\mu}\ar[d]^-{\D_1^2T}
&\D_1M_1\ar[d]^-{\D_1T}
\\\D_1\D_0M_1\times_{\D_1\D_0M_0}\D_1^2M_0
\ar[r]_-{\D_1p_2}
&\D_1^2M_0\ar[r]_-{\mu}
&\D_1M_0.}
\]
And the diagram for $p_1$ and $T$ becomes
\[
\xymatrix{
\D_2\Lhat M_2\ar[dr]^-{p_1}\ar[d]_-{p_{12}}
\\\D_2\D_0M_1\times_{\D_2\D_0M_0}\D_2\D_1M_1
\ar[r]_-{p_1}\ar[d]_-{T_\D\times_{T_\D}T_\D}
&\D_2\D_0M_1
\ar[dd]^-{T_\D}\ar[ddr]^-{T_\D^2}
\\\D_1\D_0M_1\times_{\D_1\D_0M_0}\D_1^2M_1
\ar[dr]^-{p_1}\ar[d]_-{1\times\D_1^2T}
\\\D_1\D_0M_1\times_{\D_1\D_0M_0}\D_1^2M_0
\ar[r]_-{\D_1p_1}
&\D_1\D_0M_1\ar[r]_-{T_\D}
&\D_0^2M_1\ar[r]_-{\mu}
&\D_0M_1.}
\]
Since all four diagrams commute, the expression for $\xi_2$ is correct.
\end{proof}
It will also be convenient for us to realize that the expression we already
have in place for $\xi_1:\D_1\Lhat M_1\to\Lhat M_1$ can be written as
the composite
\[
\xymatrix@C+10pt{
\D_1\D_0M_1\times_{\D_1\D_0M_0}\D_1^2M_0
\ar[r]^-{T_\D\times_{T_\D}1}
&\D_0^2M_1\times_{\D_0^2M_0}\D_1^2M_0
\ar[r]^-{\mu\times_\mu\mu}
&\D_0M_1\times_{\D_0M_0}\D_1M_0.}
\]
We now verify our desired square
\[
\xymatrix{
\D_2\Lhat M_2\ar[r]^-{\xi_2}\ar[d]_-{\gamma_\D\gamma_{\Lhat M}}
&\Lhat M_2\ar[d]^-{\gamma_{\Lhat M}}
\\\D_1\Lhat M_1\ar[r]_-{\xi_1}&\Lhat M_1}
\]
by composing with $p_1:\Lhat M_1\to\D_0M_1$ and $p_2:
\Lhat M_1\to\D_1M_0$ and checking commutativity of the resulting
diagrams.  Composing with $p_1$, we can rewrite $p_1\circ\gamma_{\Lhat M}$
as the composite
\begin{gather*}
\xymatrix{
\D_0M_1\times_{\D_0M_0}\D_1M_1\times_{\D_0M_0}\D_1M_0
\ar[r]^-{p_{12}}&\D_0M_1\times_{\D_0M_0}\D_1M_1}
\\
\xymatrix{
\relax\ar[r]^-{1\times T_\D}
&\D_0M_1\times_{\D_0M_0}\D_0M_1\cong\D_0M_2\ar[r]^-{\D_0\gamma_M}
&\D_0M_1.}
\end{gather*}
Consequently, we can express $p_1\circ\gamma_{\Lhat M}\circ\xi_2$ as
\begin{gather*}
\xymatrix{
\D_2\D_0M_1\times_{\D_2\D_0M_0}\D_2\D_1M_1\times_{\D_2\D_0M_0}\D_2\D_1M_0
\ar[r]^-{p_{12}}&\D_2\D_0M_1\times_{\D_2\D_0M_0}\D_2\D_1M_1}
\\
\xymatrix{
\relax\ar[r]^-{T_\D^2\times_{T_\D^2}T_\D}
&\D_0^2M_1\times_{\D_0^2M_0}\D_1^2M_1
\ar[r]^-{\mu\times_\mu\mu}&\D_0M_1\times_{\D_0M_0}\D_1M_1}
\\
\xymatrix{
\relax\ar[r]^-{1\times T}&\D_0M_1\times_{\D_0M_0}\D_0M_1
\cong\D_0M_2\ar[r]^-{\D_0\gamma_M}&\D_0M_1.}
\end{gather*}
In the other direction, we can write $p_1\circ\xi_1$ as the composite
\begin{gather*}
\xymatrix{
\D_1\Lhat M_1\cong\D_1\D_0M_1\times_{\D_1\D_0M_0}\D_1^2M_0
\ar[r]^-{p_1}&\D_1\D_0M_1}
\\
\xymatrix{
\relax\ar[r]^-{T_\D}&\D_0^2M_1\ar[r]^-{\mu}&\D_0M_1,}
\end{gather*}
so we can express $p_1\circ\xi_1\circ\gamma_\D\gamma_{\Lhat M}$ as 
\begin{gather*}
\xymatrix{
\D_2\D_0M_1\times_{\D_2\D_0M_0}\D_2\D_1M_1\times_{\D_2\D_0M_0}\D_2\D_1M_0
\ar[r]^-{p_{12}}&\D_2\D_0M_1\times_{\D_2\D_0M_0}\D_2\D_1M_1}
\\
\xymatrix{
\relax\ar[r]^-{\gamma_\D\times_{\gamma_\D}\gamma_\D}
&\D_1\D_0M_1\times_{\D_1\D_0M_0}\D_1^2M_1\ar[r]^-{1\times\D_1 T}
&\D_1(\D_0M_1\times_{\D_0M_0}\D_0M_1)\cong\D_1\D_0M_2}
\\
\xymatrix{
\relax\ar[r]^-{\D_1\D_0\gamma_M}&\D_1\D_0M_1
\ar[r]^-{T_\D}&\D_0^2M_1\ar[r]^-{\mu}&\D_0M_1.}
\end{gather*}
Since both expressions we wish to coincide begin with $p_{12}$, we can
simply ask that the rest of the expressions coincide.  This follows from 
the following commutative diagram:
\[
\xymatrix{
\D_2\D_0M_1\times_{\D_2\D_0M_0}\D_2\D_1M_1
\ar[r]^-{T_\D^2\times_{T_\D^2}T_\D}\ar[d]_-{\gamma_\D\times_{\gamma_\D}\gamma_\D}
&\D_0^2M_1\times_{\D_0^2M_0}\D_1^2M_1
\ar[r]^-{\mu\times_\mu\mu}\ar[d]^-{1\times T_\D}
&\D_0M_1\times_{\D_0M_0}\D_1M_1\ar[dd]^-{1\times T_\D}
\\\D_1\D_0M_1\times_{\D_1\D_0M_0}\D_1^2M_1
\ar[r]^-{T_\D\times_{T_\D}T_\D}\ar[d]_-{1\times\D_1T_\D}
&\D_0^2M_1\times_{\D_0^2M_0}\D_0\D_1M_1
\ar[d]^-{1\times\D_0T_\D}
\\\D_1\D_0M_1\times_{\D_1\D_0M_0}\D_1\D_0M_1
\ar[r]^-{T_\D\times_{T_\D}T_\D}\ar[d]_-{\cong}
&\D_0^2M_1\times_{\D_0^2M_0}\D_0^2M_1
\ar[r]^-{\mu\times_\mu\mu}\ar[d]^-{\cong}
&\D_0M_1\times_{\D_0M_0}\D_0M_1\ar[d]^-{\cong}
\\\D_1\D_0M_2\ar[r]^-{T_\D}\ar[d]_-{\D_1\D_0\gamma_M}
&\D_0^2M_2\ar[r]^-{\mu}\ar[d]^-{\D_0^2\gamma_M}
&\D_0M_2\ar[d]^-{\D_0\gamma_M}
\\\D_1\D_0M_1\ar[r]_-{T_\D}
&\D_0^2M_1\ar[r]_-{\mu}
&\D_0M_1.}
\]
For the composition with $p_2$, we can first rewrite $p_2\circ\gamma_{\Lhat M}$
as the composite
\begin{gather*}
\xymatrix{
\D_0M_1\times_{\D_0M_0}\D_1M_1\times_{\D_0M_0}\D_1M_0
\ar[r]^-{p_{23}}&\D_1M_1\times_{D_0M_0}\D_1M_0}
\\
\xymatrix{
\relax\ar[r]^-{\D_1S\times1}&\D_1\D_0M_0\times_{\D_0M_0}\D_1M_0
\ar[r]^-{\D_1I_\D\times1}&\D_1^2M_0\times_{\D_0M_0}\D_1M_0}
\\
\xymatrix{
\relax\ar[r]^-{\mu\times1}&\D_1M_0\times_{\D_0M_0}\D_1M_0\cong\D_2M_0
\ar[r]^-{\gamma_\D}&\D_1M_0.}
\end{gather*}
Consequently, we can write $p_2\circ\gamma_{\Lhat M}\circ\xi_2$ as the composite
\begin{gather*}
\xymatrix{
\D_2\D_0M_1\times_{\D_2\D_0M_0}\D_2\D_1M_1\times_{\D_2\D_0M_0}\D_2\D_1M_0
\ar[r]^-{p_{23}}&\D_2\D_1M_1\times_{\D_2\D_0M_0}\D_2\D_1M_0}
\\
\xymatrix{
\relax\ar[rr]^-{T_\D\times_{T_\D S_\D}S_\D}&&\D_1^2M_1\times_{\D_0^2M_0}\D_1^2M_0
\ar[r]^-{\mu\times_\mu\mu}&\D_1M_1\times_{\D_0M_0}\D_1M_0}
\\
\xymatrix{
\relax\ar[r]^-{\D_1S\times1}&\D_1\D_0M_0\times_{\D_0M_0}\D_1M_0
\ar[r]^-{\D_1I_\D\times1}&\D_1^2M_0\times_{\D_0M_0}\D_1M_0}
\\
\xymatrix{
\relax\ar[r]^-{\mu\times1}&\D_1M_0\times_{\D_0M_0}\D_1M_0\cong\D_2M_0
\ar[r]^-{\gamma_\D}&\D_1M_0.}
\end{gather*}
In the other direction, we can write $p_2\circ\xi_1$ as the composite
\[
\xymatrix{
\D_1\Lhat M_1\cong\D_1\D_0M_1\times_{\D_1\D_0M_0}\D_1^2M_0
\ar[r]^-{p_2}&\D_1^2M_0\ar[r]^-{\mu}&\D_1M_0,}
\]
and consequently we can write $p_2\circ\xi_1\circ\gamma_\D\gamma_{\Lhat M}$
as the composite
\begin{gather*}
\xymatrix{
\D_2\D_0M_1\times_{\D_2\D_0M_0}\D_2\D_1M_1\times_{\D_2\D_0M_0}\D_2\D_1M_0
\ar[r]^-{p_{23}}&\D_2\D_1M_1\times_{\D_2\D_0M_0}\D_2\D_1M_0}
\\
\xymatrix{
\relax\ar[rr]^-{\gamma_\D\times_{\gamma_\D}\gamma_\D}&&\D_1^2M_1\times_{\D_1\D_0M_0}\D_1^2M_0
\ar[r]^-{\D_1^2S\times1}&\D_1^2\D_0M_0\times_{\D_1\D_0M_0}\D_1^2M_0}
\\
\xymatrix{
\relax\ar[r]^-{\D_1^2I_\D}&\D_1^3M_0\times_{\D_1\D_0M_0}\D_1^2M_0
\ar[r]^-{\D_1\mu\times1}&\D_1^2M_0\times_{\D_1\D_0M_0}\D_1^2M_0}
\\
\xymatrix{
\relax\ar[r]^-{\cong}&\D_1\D_2M_0\ar[r]^-{\D_1\gamma_\D}&\D_1^2M_0
\ar[r]^-{\mu}&\D_1M_0.}
\end{gather*}
Adopting our previous notation $\theta$ for the composite of source type
\[
\xymatrix{
\D_1M_1\ar[r]^-{\D_1S}&\D_1\D_0M_0\ar[r]^-{\D_1I_\D}&\D_1^2M_0\ar[r]^-{\mu}&\D_1M_0,}
\]
we can see that these coincide by means of the following commutative diagram:
\[
\xymatrix@C+15pt{
\D_2\D_1M_1\times_{\D_2\D_0M_0}\D_2\D_1M_0
\ar[r]^-{T_\D\times_{T_\D S_\D}S_\D}\ar[d]_-{\D_2\theta\times1}
&\D_1^2M_1\times_{\D_0^2M_0}\D_1^2M_0
\ar[r]^-{\mu\times_\mu\mu}\ar[d]^-{\D_1\theta\times1}
&\D_1M_1\times_{\D_0M_0}\D_1M_0\ar[d]^-{\theta\times1}
\\\D_2\D_1M_0\times_{\D_2\D_0M_0}\D_2\D_1M_0
\ar[r]_-{T_\D\times_{T_\D S_\D}S_\D}\ar[d]_-{\gamma_\D\times_{\gamma_\D}\gamma_\D}
&\D_1^2M_0\times_{\D_0^2M_0}\D_1^2M_0
\ar[r]_-{\mu\times_\mu\mu}\ar[d]^-{\cong}
&\D_1M_0\times_{\D_0M_0}\D_1M_0\ar[d]^-{\cong}
\\\D_1^2M_0\times_{\D_1\D_0M_0}\D_1^2M_0\ar[d]_-{\cong}
&\D_2^2M_0\ar[r]^-{\mu}\ar[dl]_-{\gamma_\D}\ar[d]^-{\gamma_\D^2}
&\D_2M_0\ar[d]^-{\gamma_\D}
\\\D_1\D_2M_0\ar[r]_-{\D_1\gamma_\D}
&\D_1^2M_0\ar[r]_-{\mu}
&\D_1M_0.}
\]
Consequently, our action maps preserve the composition, and we have defined
a $\D$-algebra structure on $\Lhat M$.

\section{The adjunction structure I: the unit map}

This section is devoted to constructing a unit map $M\to U\Lhat M$ for a
$\D$-multicategory $M$  and verifying its properties.
Together with the counit described in the next section, this map almost
(but not quite) determines an adjunction between $U$ and $\Lhat$.
They will do most of the heavy lifting
for the actual adjunction once we construct the actual left adjoint $L$.  The 
reason we don't get an adjunction immediately is that the unit map 
of this section
isn't quite a map of $\D$-multicategories: it doesn't preserve the presheaf
structure on the sets of morphisms.

The unit map consists of maps $\eta_0:M_0\to U\Lhat M_0$
and $\eta_1:M_1\to U\Lhat M_1$.  Since $(U\C)_0=\C_0$ for a $\D$-algebra $\C$
and $\Lhat M_0=\D_0M_0$, we just use the unit of the monad $\D_0$
\[
\eta_0=\eta:M_0\to\D_0M_0
\]
as the map on objects of our unit.

To define the unit map on morphisms, we recall that $\Lhat M_1$ is given
by the pullback diagram
\[
\xymatrix{
\Lhat M_1\ar[r]^-{p_1}\ar[dd]_-{p_2}&\D_0M_1\ar[d]^-{\D_0S}
\\&\D_0^2M_0\ar[d]^-{\mu}
\\\D_1M_0\ar[r]_-{T_\D}&\D_0M_0,}
\]
that is, $\Lhat M_1\cong\D_0M_1\times_{\D_0M_0}\D_1M_0$.  Further,
the underlying multicategory $U\C$ of a $\D$-algebra $\C$ (such as
$\Lhat M$) is given by the pullback
\[
\xymatrix{
(U\C)_1\ar[r]^-{p_1}\ar[d]_-{S}&\C_1\ar[d]^-{S}
\\\D_0\C_0\ar[r]_-{\xi_0}&\C_0,}
\]
where $\xi_0:\D_0\C_0\to\C_0$ is the restriction to objects of
the action of $\D$ on its algebra $\C$, 
and $p_1$ coincides with the comparison map $\kappa_1$.  Combining the two diagrams,
we see that $(U\Lhat M)_1$ is given by the pullback in the diagram
\[
\xymatrix{
U\Lhat M_1\ar[r]^-{p_1}\ar[ddd]_-{S}
&\Lhat M_1\ar[r]^-{p_1}\ar[dd]_-{p_2}
&\D_0M_1\ar[d]^-{\D_0S}
\\&&\D_0^2M_0\ar[d]^-{\mu}
\\&\D_1M_0\ar[r]_-{T_\D}\ar[d]^-{S}&\D_0M_0
\\\D_0^2M_0\ar[r]_-{\mu}&\D_0M_0,}
\]
so we can write
\[
U\Lhat M_1\cong\D_0M_1\times_{\D_0M_0}\D_1M_0\times_{\D_0M_0}\D_0^2M_0.
\]
Consequently, any map
to $U\Lhat M_1$ is given by maps to $\D_0^2M_0$, to $\D_1M_0$,
and to $\D_0M_1$, subject to compatibility relations.  We agree to define
$\eta:M_1\to U\Lhat M_1$ by the following maps:
\begin{enumerate}
\item
To $\D_0^2M_0$: $\xymatrix{M_1\ar[r]^-{S}&\D_0M_0\ar[r]^-{\D_0\eta}&\D_0^2M_0}$.
\item
To $\D_1M_0$: $\xymatrix{M_1\ar[r]^-{S}&\D_0M_0\ar[r]^-{I_\D}&\D_1M_0}$.
\item
To $\D_0M_1$: $\xymatrix{M_1\ar[r]^-{\eta}&\D_0M_1}$.
\end{enumerate}
We will show this is almost a map of $\D$-multicategories: 
the only defect is that the map on morphisms doesn't preserve the presheaf structure.

  First, we observe that we have
a well-defined map.  For this, we have the two diagrams
\[
\xymatrix{
M_1\ar[d]_-{S}
\\\D_0M_0\ar[r]^-{I_\D}\ar[dr]^-{=}\ar[d]_-{\D_0\eta}&\D_1M_0\ar[d]^-{S}
\\\D_0^2M_0\ar[r]_-{\mu}&\D_0M_0}
\text{ and }
\xymatrix{
M_1
\ar[r]^-{\eta}
\ar[d]_-{S}
&\D_0M_1
\ar[d]^-{\D_0S}
\\\D_0M_0
\ar[r]^-{\eta}
\ar[dr]^-{=}\ar[d]_-{I_\D}
&\D_0^2M_0
\ar[d]^-{\mu}
\\\D_1M_0\ar[r]_-{T}&\D_0M_0,
}
\]
which verify that we have actually defined a map to $U\Lhat M_1$.

We will show that this map satisfies all the rest of the properties of a map
of $\D$-multicategories, namely, that it preserves $I$, $S$, $T$, and $\gamma$.  For preservation
of $I$, we need the diagram
\[
\xymatrix{
M_0\ar[r]^-{\eta}\ar[d]_-{I}&\D_0M_0\ar[d]^-{I}
\\M_1\ar[r]_-{\eta}&U\Lhat M_1}
\]
to commute, which we show by composing to the three components $\D_0^2M_0$,
$\D_1M_0$, and $\D_0M_1$.  First, we recall that $I:\D_0M_0\to U\Lhat M_1$ is
a special case of the more general $I:\C_0\to U\C_1$ for a $\D$-algebra $\C$, which
is given by the commutative diagram
\[
\xymatrix{
\C_0\ar[r]^-{I}\ar[dr]^-{=}\ar[d]_-{\eta}&\C_1\ar[d]^-{S}
\\\D_0\C_0\ar[r]_-{\xi_0}&\C_0,}
\]
which defines a map to the pullback $U\C_1$ of the square.  When $\C=\Lhat M$,
we have the identity map of $\Lhat M$ induced by the diagram
\[
\xymatrix{
\D_0M_0\ar[r]^-{\D_0I}\ar[dr]^-{\D_0\eta}\ar[ddr]_-{=}\ar[dd]_-{I_\D}
&\D_0M_1\ar[d]^-{\D_0S}
\\&\D_0^2M_0\ar[d]^-{\mu}
\\\D_1M_0\ar[r]_-{T}&\D_0M_0,}
\]
and then $\eta:\D_0M_0\to\D_0^2M_0$ satisfies commutativity in
\[
\xymatrix{
\D_0M_0\ar[r]^-{I_\D}\ar[dr]^-{=}\ar[d]_-{\eta}&\D_1M_0\ar[d]^-{S}
\\\D_0^2M_0\ar[r]_-{\mu}&\D_0M_0,}
\]
and so induces the identity map $I:\D_0M_0\to U\Lhat M_1$.  We
can now verify preservation of $I$ by checking the composites to
$\D_0^2M_0$, $\D_0M_0$, and $\D_0M_1$.  To $\D_0^2M_0$,
we check that
\[
\xymatrix{
M_0\ar[rr]^-{\eta}\ar[dr]^-{\eta}\ar[d]_-{I}&&\D_0M_0\ar[d]^-{\eta}
\\M_1\ar[r]_-{S}&\D_0M_0\ar[r]_-{\D_0\eta}&\D_0^2M_0}
\]
commutes by naturality of $\eta$.  To $\D_1M_0$, we have the
trivial diagram
\[
\xymatrix{
M_0\ar[rr]^-{\eta}\ar[dr]^-{\eta}\ar[d]_-{I}&&\D_0M_0\ar[d]^-{I}
\\M_1\ar[r]_-{S}&\D_0M_0\ar[r]_-{I}&\D_2M_0,}
\]
and to $\D_0M_1$, we have
\[
\xymatrix{
M_0\ar[r]^-{\eta}\ar[d]_-{I}&\D_0M_0\ar[d]^-{\D_0I}
\\M_1\ar[r]_-{\eta}&\D_0M_1,}
\]
which commutes by naturality of $\eta$.  Consequently, $I$ is preserved.

For preservation of $S$, we have
\[
\xymatrix{
M_1\ar[r]^-{\eta}\ar[d]_-{S}&U\Lhat M_1\ar[d]^-{S}
\\\D_0M_0\ar[r]_-{\D_0\eta}&\D_0^2M_0,}
\]
which is why $S\circ\eta$ is defined the way it is, and for
preservation of $T$, we have
\[
\xymatrix{
&M_1\ar[rr]^-{T}\ar[dr]^-{\eta}\ar[dl]_-{\eta}&&M_0\ar[d]^-{\eta}
\\U\Lhat M_1\ar[r]_-{p_1}&\Lhat M_1\ar[r]_-{p_1}&\D_0M_1\ar[r]_-{\D_0T}&\D_0M_0,}
\]
which is another application of naturality of $\eta$.  This leaves
preservation of $\gamma$ to be verified.

For preservation of $\gamma$, what we need is commutativity of
the diagram
\[
\xymatrix{
M_1\times_{\D_0M_0}\D_0M_1\ar[r]^-{\gamma_M}\ar[d]_-{\eta_1\times_{\D_0\eta_0}\D_0\eta_1}
&M_1\ar[d]^-{\eta_1}
\\U\Lhat M_1\times_{\D_0U\Lhat M_0}\D_0U\Lhat M_1\ar[r]_-{\gamma_{U\Lhat M}}
&U\Lhat M_1.}
\]
We already have the expression $\D_0M_1\times_{\D_0M_0}\D_1M_0\times_{\D_0M_0}\D_0^2M_0$
for $U\Lhat M_1$, so expanding out the lower left corner, which is really $U\Lhat M_2$, we get
the following isomorphism by canceling the $\D_0^2M_0$'s in the middle:
\begin{gather*}
\D_0M_1\times_{\D_0M_0}\D_1M_0\times_{\D_0M_0}\D_0^2M_0
\times_{\D_0^2M_0}\D_0^2M_1\times_{D_0^2M_0}\D_0\D_1M_0\times_{\D_0^2M_0}\D_0^3M_0
\\
\cong\D_0M_1\times_{\D_0M_0}\D_1M_0\times_{\D_0M_0}
\D_0^2M_1\times_{D_0^2M_0}\D_0\D_1M_0\times_{\D_0^2M_0}\D_0^3M_0,
\end{gather*}
Now the definition
of composition for $U\C$ is in terms of the comparison map $\kappa_2:U\C_2\to\C_2$, which
can be expressed as follows.  First, we have 
\begin{gather*}
U\C_2:=U\C_1\times_{\D_0U\C_0}\D_0U\C_1
\\\cong\C_1\times_{\C_0}\D_0\C_0\times_{\D_0\C_0}\D_0\C_1\times_{\D_0\C_0}\D_0^2\C_0
\\\cong\C_1\times_{\C_0}\D_0\C_1\times_{\D_0\C_0}\D_0^2\C_0,
\end{gather*}
and then the comparison map $\kappa_2$ can be expressed as the composite
\[
\xymatrix{
\C_1\times_{\C_0}\D_0\C_1\times_{\D_0\C_0}\D_0^2\C_0
\ar[r]^-{p_{12}}
&\C_1\times_{\C_0}\D_0\C_1\ar[r]^-{1\times I_\D}
&\C_1\times_{\C_0}\D_1\C_1\ar[r]^-{1\times\xi_1}
&\C_1\times_{\C_0}\C_1=\C_2.}
\]
Now we can write the composition in $U\C$ in terms of its projections
to the factors of $U\C_1=\C_1\times_{\C_0}\D_0\C_0$.  The projection to 
$\C_1$ is just $\kappa_2$ as expressed above, composed with $\gamma_\C:\C_2\to\C_1$.
The projection to $\D_0\C_0$ consists of projection to the last factor $\D_0^2\C_0$ of
$U\C_2$, composed with $\mu:\D_0^2\C_0\to\D_0\C_0$.  This now has to be
interpreted in the case where $\C=\Lhat M$, in which case we have the actual
composition taking place in the composite
\begin{align*}
&\xymatrix{
&&\D_0M_1\times_{\D_0M_0}\D_1M_0\times_{\D_0M_0}
\D_0^2M_1\times_{D_0^2M_0}\D_0\D_1M_0\times_{\D_0^2M_0}\D_0^3M_0}
\\
&\xymatrix{
\relax\ar[rr]^-{p_{1234}}
&&\D_0M_1\times_{\D_0M_0}\D_1M_0\times_{\D_0M_0}
\D_0^2M_1\times_{D_0^2M_0}\D_0\D_1M_0}
\\
&\xymatrix{
\relax\ar[rr]^-{1\times1\times I_\D\times I_\D}
&&\D_0M_1\times_{\D_0M_0}\D_1M_0\times_{\D_0M_0}
\D_1\D_0M_1\times_{D_1\D_0M_0}\D_1^2M_0}
\\
&\xymatrix{
\relax\ar[rr]^-{1\times1\times T_\D\times_{T_\D}1}
&&\D_0M_1\times_{\D_0M_0}\D_1M_0\times_{\D_0M_0}
\D_0^2M_1\times_{D_0^2M_0}\D_1^2M_0}
\\
&\xymatrix{
\relax\ar[rr]^-{1\times1\times\mu\times_\mu\mu}
&&\D_0M_1\times_{\D_0M_0}\D_1M_0\times_{\D_0M_0}
\D_0M_1\times_{\D_0M_0}\D_1M_0}
\\
&\xymatrix{
\relax
&&\D_0M_1\times_{\D_0M_0}\D_1M_1\times_{\D_0M_0}\D_1M_0
\ar[ll]_-{1\times(\D_1T,S_\D)\times1}^-{\cong}}
\\
&\xymatrix{
\relax\ar[rr]^-{1\times(T_\D,\theta)\times1}
&&\D_0M_1\times_{\D_0M_0}\D_0M_1\times_{\D_0M_0}
\D_1M_0\times_{\D_0M_0}\D_1M_0}
\\
&\xymatrix{
\relax\ar[rr]^-{\D_0\gamma_M\times\gamma_\D}
&&\D_0M_1\times_{\D_0M_0}\D_1M_0.}
\end{align*}
The other factor is much more straightforward: it is just
\[
\xymatrix{
U\Lhat M_1\ar[r]^-{p_5}&\D_0^3M_0\ar[r]^-{\mu}&\D_0^2M_0.}
\]
We verify preservation of $\gamma$ by projecting
to each of the three factors of $U\Lhat M_1=
\D_0M_1\times_{\D_0M_0}\D_1M_0\times_{\D_0M_0}\D_0^2M_0$,
and the last factor of $\D_0^2M_0$ doesn't present much difficulty:
it's a consequence of the diagram
\[
\xymatrix{
M_1\times_{\D_0M_0}\D_0M_1\ar[rr]^-{\gamma_M}\ar[d]_-{p_2=S}
&&M_1\ar[d]^-{S}
\\\D_0M_1\ar[r]^-{\D_0S}
&\D_0^2M_0\ar[r]^-{\mu}\ar[d]_-{\D_0^2\eta}
&\D_0M_0\ar[d]^-{\D_0\eta}
\\&\D_0^3M_0\ar[r]_-{\mu}
&\D_0^2M_0.}
\]

We next consider the projection to the first factor $\D_0M_1$, and observe
first that the long composite above, projected to $\D_0M_1$, depends
only on the first three factors of $U\Lhat M_2$, and after projecting
onto those, we can express the map as follows:
\begin{align*}
&\xymatrix{
&&\D_0M_1\times_{\D_0M_0}\D_1M_0\times_{\D_0M_0}\D_0^2M_1}
\\
&\xymatrix{
\relax\ar[rr]^-{1\times1\times\mu_\D}
&&\D_0M_1\times_{\D_0M_0}\D_1M_0\times_{\D_0M_0}\D_0M_1}
\\
&\xymatrix{
\relax\ar[rr]^-{1\times\chi}
&&\D_0M_1\times_{\D_0M_0}\D_0M_1\times_{\D_0M_0}\D_1M_0}
\\
&\xymatrix{
\relax\ar[rr]^-{\D_0\gamma_M\circ p_{12}}&&\D_0M_1.}
\end{align*}
Next, the portion of the unit
map $\eta_2=\eta_1\times_{\D_0\eta_0}\D_0\eta_1:M_1\times_{\D_0M_0}\D_0M_1
\to U\Lhat M_2$ that actually maps to the first three factors consists of
\[
(\eta,I_\D\circ S,\D_0\eta\circ S):
M_1\to\D_0M_1\times_{\D_0M_0}\D_1M_0\times_{\D_0M_0}\D_0^2M_0
\]
and
\[
\D_0\eta:\D_0M_1\to\D_0^2M_1.
\]
However, since we're taking a fiber product over $\D_0^2M_0$ of these two
maps, the result is just $(\eta,I_\D\circ S)\times\D_0\eta$.  Further, the diagram
\[
\xymatrix{
\D_0M_1\ar[r]^-{\D_0\eta}\ar[dr]_-{=}&\D_0^2M_1\ar[d]^-{\mu}
\\&\D_0M_1}
\]
allows us to simplify and write the material part of the composite of $\eta$ and $\gamma_{\Lhat M}$
with $p_1$ as follows:
\begin{align*}
&\xymatrix{
&&M_1\times_{\D_0M_0}\D_0M_1}
\\
&\xymatrix{
\ar[rr]^-{(\eta,I_\D\circ S)\times1}
&&\D_0M_1\times_{\D_0M_0}\D_1M_0\times_{\D_0M_0}\D_0^2M_1}
\\
&\xymatrix{
\relax\ar[rr]^-{1\times1\times\mu_\D}
&&\D_0M_1\times_{\D_0M_0}\D_1M_0\times_{\D_0M_0}\D_0M_1}
\\
&\xymatrix{
\relax\ar[rr]^-{1\times\chi}
&&\D_0M_1\times_{\D_0M_0}\D_0M_1\times_{\D_0M_0}\D_1M_0}
\\
&\xymatrix{
\relax\ar[rr]^-{\D_0\gamma_M\circ p_{12}}&&\D_0M_1.}
\end{align*}
The diagram we wish to commute then becomes the perimeter of the following one,
in which as usual we suppress subscript $\D_0M_0$'s:
\[
\xymatrix{
M\times\D_0M_1\ar[rrr]^-{\gamma_M}\ar[drr]_-{\cong}\ar[d]_-{(\eta,I_\D\circ S)\times1}
&&&M_1\ar[ddd]^-{\eta}
\\\D_0M_1\times\D_1M_0\times\D_0M_1\ar[dd]^-{1\times\chi}&&M_2\ar[ur]_-{\gamma_M}\ar[d]^-{\eta}
\\&\D_0M_1\times\D_1M_1
\ar[ul]_-{1\times(\D_1T,S_\D)}^-{\cong}
\ar[dl]^-{1\times(T_\D,\theta)}\ar[dr]_-{1\times T_\D}
&\D_0M_2\ar[dr]^-{\D_0\gamma_M}\ar[d]^-{\cong}
\\\D_0M_1\times\D_0M_1\times\D_1M_0\ar[rr]_-{p_{12}}
&&\D_0M_1\times\D_0M_1\ar[r]_-{\D_0\gamma_M}
&\D_0M_1.}
\]
The lower left triangle records the definition of $\chi$, and all the rest of 
the sub-diagrams are clear with the exception of the central slanted
rectangle.  We claim that the following diagram commutes, which will
allow us to rewrite the upper right composite in the remaining sub-diagram:
\[
\xymatrix{
M_1\times_{\D_0M_0}\D_0M_1\ar[r]^-{\cong}\ar[d]_-{\eta\times1}
&M_2\ar[d]^-{\eta}
\\\D_0M_1\times_{\D_0M_0}\D_0M_1\ar[r]_-{\cong}&\D_0M_2.}
\]
The main reason this is true is that the unit map $\eta:M_1\to\D_0M_1$
preserves the structure map of source type, as seen in the following
diagram:
\[
\xymatrix{
M_1\ar[r]^-{\eta}\ar[d]_-{S}&\D_0M_1\ar[d]^-{\D_0S}
\\\D_0M_0\ar[r]^-{\eta}\ar[dr]_-{=}&\D_0^2M_0\ar[d]^-{\mu}
\\&\D_0M_0.}
\]
Consequently, the following cube commutes, with the front and back faces
being pullbacks:
\[
\xymatrix{
M_2
\ar[rr]^-{T}\ar[dr]^-{\eta}\ar[dd]_-{S}&&M_1\ar[dr]^-{\eta}\ar[dd]^(0.7){S}
\\&\D_0M_2\ar[rr]^(0.3){\D_0T}\ar[dd]^(0.7){\mu\circ\D_0S}&&\D_0M_1\ar[dd]^-{\mu\circ\D_0S}
\\\D_0M_1\ar[rr]^(0.3){\D_0T}\ar[dr]_-{=}&&\D_0M_0\ar[dr]^-{=}
\\&\D_0M_1\ar[rr]_-{\D_0T}&&\D_0M_0.}
\]
This implies that the claimed square commutes.  We can now replace
the desired sub-diagram with the following one:
\[
\xymatrix@C+15pt{
M\times\D_0M_1\ar[ddrr]^-{\eta\times1}\ar[dd]_-{(\eta,I_\D\circ S)\times1}
\\\relax
\\\D_0M_1\times\D_1M_0\times\D_0M_1
&\D_0M_1\times\D_1M_1\ar[l]^-{1\times(\D_1T,S_\D)}_-{\cong}\ar[r]_-{1\times T_\D}
&\D_0M_1\times\D_0M_1.}
\]
After composition with projection to the first $\D_0M_1$, both composites coincide with 
$\eta$, so we can concentrate on the composition with the second factor of $\D_0M_1$.
Here we find the diagram commutes as a consequence of the following filling:
\[
\xymatrix{
M_1\times\D_0M_1\ar[dr]^-{p_2}\ar[d]_-{S\times1}
\\\D_0M_0\times\D_0M_1\ar[d]_-{I_\D\times1}
&\D_0M_1\ar[l]^-{(\D_0T,\id)}_-{\cong}\ar[dr]^-{=}\ar[d]_-{I_\D}
\\\D_1M_0\times\D_0M_1
&\D_1M_1\ar[l]^-{(D_1T,S_\D)}_-{\cong}\ar[r]_-{T_\D}&\D_0M_1.}
\]
We may conclude that the unit maps preserve $\gamma$ after projection onto
the first factor.

It remains to verify preservation of $\gamma$ after projection onto the second factor,
$\D_1M_0$.  Looking again at the long composite giving the composition in $\Lhat M$,
we see that the projection onto $\D_1M_0$ depends only on the middle three factors.
Further, the portion of the unit mapping to these three factors consists of the following, 
in which we again suppress the subscripts for the fiber products:
\[
\xymatrix{
M_1\times\D_0M_1\ar[rrr]^-{(I_\D\circ S)\times(\D_0\eta,\D_0(I_\D\circ S))}
&&&\D_1M_0\times\D_0^2M_1\times\D_0\D_1M_0.}
\]
We now simplify as follows:

\begin{lemma}
The composite of the previous part of the unit with the portion of the composition
in $\Lhat M$ ending at $\D_1M_0\times\D_0M_1\times\D_1M_0$
coincides with 
\[
(I_\D\circ S)\times(1,I_D\circ\mu\circ\D_0S).
\]
\end{lemma}

\begin{proof}
Tracing the projections to each factor of the target, the first one is clear: both
coincide with $I_\D\circ S$.  For the middle factor, we do just get the identity on $\D_0M_1$
since 
\[
\xymatrix{
\D_0M_1\ar[r]^-{\D_0\eta}\ar[dr]_-{=}&\D_0^2M_1\ar[d]^-{\mu}
\\&\D_0M_1}
\]
commutes.  The expressions for the last $\D_1M_0$ coincide by the commutativity of
\[
\xymatrix{
\D_0M_1\ar[r]^-{\D_0S}&\D_0^2M_0\ar[r]^-{\D_0I_\D}\ar[d]_-{\mu}
&\D_0\D_1M_0\ar[r]^-{I_\D}&\D_1^2M_0\ar[d]^-{\mu}
\\&\D_0M_0\ar[rr]_-{I_\D}&&\D_1M_0.}
\]
\end{proof}

Now we simplify the composite of $p_2\circ\gamma_{\Lhat M}$ with the unit map even
further by means of the following commutative diagram, in which it is extremely important
to realize that the subscript $\D_0M_0$'s have been suppressed:
\[
\xymatrix@C+25pt{
M_1\times\D_0M_1\ar[dr]^-{p_2}\ar[d]_-{S\times1}
\\\D_0M_0\times\D_0M_1\ar[d]_-{I_\D\times1}
&\D_0M_1\ar[l]^-{(\D_0T,1)}_-{\cong}\ar[d]^-{I_\D}
\\\D_1M_0\times\D_0M_1\ar[d]_-{1\times(1,\D_0S)}
&\D_1M_1\ar[dd]^-{=}\ar[l]_-{(\D_1T,S_\D)}^-{\cong}
\\\D_1M_0\times\D_0M_1\times\D_0^2M_0\ar[d]_-{1\times1\times\mu}
\\\D_1M_0\times\D_0M_1\times\D_0M_0\ar[d]_-{1\times1\times I_\D}
&\D_1M_1\ar[l]_-{(D_1T,S_\D,\mu\circ\D_0S\circ S_\D)}^-{\cong}
\ar[d]^-{(1,I_\D\circ\mu\circ\D_0S\circ S_\D)}
\\\D_1M_0\times\D_0M_1\times\D_1M_0\ar[d]_-{\chi\times1}
&\D_1M_1\times\D_1M_0\ar[l]_-{(\D_1T,S_\D)\times1}^-{\cong}
\ar[dl]^-{(S_\D,\theta)\times1}\ar[d]^-{\theta\times1}
\\\D_0M_1\times\D_1M_0\times\D_1M_0\ar[r]_-{p_{23}}
&\D_1M_0\times\D_1M_0\ar[r]_-{\gamma_\D}&\D_1M_0.}
\]
Next, we claim the following diagram commutes:
\[
\xymatrix{
\D_1M_1\ar[dr]^-{\,\,\,\,\,\,\,(\theta,\mu\circ\D_0S\circ S_\D)}\ar[d]_-{\theta}
\\\D_1M_0\ar[r]^-{(1,S_\D)}_-{\cong}\ar[dr]_-{=}
&\D_1M_0\times\D_0M_0\ar[r]^-{(1,I_\D)}
&\D_1M_0\times\D_1M_0\ar[dl]^-{\gamma_\D}
\\&\D_1M_0.}
\]
The lower triangle commutes because $\gamma_\D$ is right unital,
and the upper one commutes because of the commutative diagram
\[
\xymatrix{
\D_1M_1\ar[r]^-{\D_1S}\ar[d]_-{S_\D}
&\D_1\D_0M_0\ar@<.5ex>[r]^-{\D_1I_\D}\ar[d]_-{S_\D}
&\D_1^2M_0\ar[r]^-{\mu}
\ar@{.>}@<.5ex>[l]^-{\D_1 S_\D}
&\D_1M_0\ar[d]^-{S_\D}
\\\D_0M_1\ar[r]_-{\D_0S}
&\D_0^2M_0\ar[rr]_-{\mu}
&&\D_0M_0.}
\]
The right square commutes because $\D_1S_\D$ splits $\D_1I_\D$
and the square commutes when starting at $\D_1^2M_0$.
We can now replace the composite in the previous large diagram
starting at $\D_1M_1$ with just $\theta:\D_1M_1\to\D_1M_0$, which
expands to $\mu_\D\circ\D_1I_\D\circ\D_1S$.  The whole diagram
we wish to commute then reduces to the perimeter of the following
diagram, which does commute:
\[
\xymatrix{
M_1\times_{\D_0M_0}\D_0M_1\ar[r]^-{\cong}\ar[d]_-{p_2}
&M_2\ar[dl]^-{S}\ar[r]^-{\gamma_M}
&M_1\ar[d]^-{S}
\\\D_0M_1\ar[r]_-{\D_0S}\ar[d]_-{I_\D}
&\D_0^2M_0\ar[r]^-{\mu}\ar[d]_-{I_\D}
&\D_0M_0\ar[dd]^-{I_\D}
\\\D_1M_1\ar[r]_-{\D_1S}
&\D_1\D_0M_0\ar[d]_-{\D_1I_\D}
\\&\D_1^2M_0\ar[r]_-{\mu}
&\D_1M_0.}
\]
This completes the verification that the unit map preserves the compositions.
The only thing missing in showing that it is a map of $\D$-multicategories
is to show that $\eta_1:M_1\to\Lhat M_1$ is a map of $\D(*)$-presheaves.
However, this is in general false, and is the reason we need to descend to 
a quotient, which will give us the actual left adjoint.

\section{The adjunction structure II: the counit map}

Next we define the counit map $\e:\Lhat U\C\to\C$, and
verify its properties.  On objects, we have 
\[
\Lhat U\C_0=\D_0\C_0,
\]
so for the counit $\e_0$ we use the action map $\xi_0:\D_0\C_0\to\C_0$ given by the
algebra structure on $\C$ over $\D$.  The significant part is the definition
of the map on morphisms, $\e_1:(\Lhat U\C)_1\to\C_1$.  First, we can apply $\D_0$ to 
the pullback diagram defining $(U\C)_1$ to get another pullback diagram
\[
\xymatrix{
\D_0(U\C)_1\ar[r]^-{\D_0\kappa_1}\ar[d]_-{\D_0S}&\D_0C_1\ar[d]^-{\D_0S}
\\\D_0^2\C_0\ar[r]_-{\D_0\xi_0}&\D_0\C_0.}
\]
Next, the definition of morphisms in $\Lhat M$ applied to $M=U\C$ gives us
a pullback diagram
\[
\xymatrix{
(\Lhat U\C)_1\ar[r]\ar[dd]&\D_0U\C_1\ar[d]^-{\D_0S}
\\&\D_0^2\C_0\ar[d]^-{\mu}
\\\D_1\C_0\ar[r]_-{T}&\D_0\C_0.}
\]
These paste together to form the core of the following diagram, which
allows us to form a composable pair in $\C_1$ from an element of $\Lhat U\C_1$;
we define the counit on morphisms by the resulting composite:
\[
\xymatrix{
(\Lhat U\C)_1\ar[r]\ar[dd]&\D_0U\C_1\ar[r]^-{\D_0\kappa_1}\ar[d]_-{\D_0S}
&\D_0\C_1\ar[dr]^-{I_\D}\ar[d]_-{\D_0S}
\\&\D_0^2\C_0\ar[r]^-{\D_0\xi_0}\ar[d]_-{\mu}
&\D_0\C_0\ar[d]_-{\xi_0}&\D_1\C_1\ar[l]_-{S}\ar[d]^-{\xi_1}
\\\D_1\C_0\ar[r]^-{T}\ar[dr]_-{\D_1I}&\D_0\C_0\ar[r]^-{\xi_0}
&\C_0&\C_1\ar[l]_-{S}
\\&\D_1\C_1\ar[r]_-{\xi_1}\ar[u]_-{T}&\C_1\ar[u]_-{T}&\C_2\ar[l]^-{S}\ar[u]_-{T}\ar[dr]^-{\gamma}
\\&&&&\C_1.}
\]
We can further augment this diagram to verify that the resulting
composite has the correct source and target:
\[
\xymatrix{
(\Lhat U\C)_1\ar[r]\ar[dd]&\D_0U\C_1\ar[r]^-{\D_0\kappa_1}\ar[d]_-{\D_0S}
&\D_0\C_1\ar[dr]_-{I_\D}\ar[d]_-{\D_0S}\ar[drr]^-{\D_0T}
\\&\D_0^2\C_0\ar[r]^-{\D_0\xi_0}\ar[d]_-{\mu}
&\D_0\C_0\ar[d]_-{\xi_0}&\D_1\C_1\ar[l]_-{S}\ar[d]^-{\xi_1}\ar[r]_-{T}&\D_0\C_0\ar[d]^-{\xi_0}
\\\D_1\C_0\ar[r]^-{T}\ar[dr]^-{\D_1I}\ar[ddr]_-{S}&\D_0\C_0\ar[r]^-{\xi_0}
&\C_0&\C_1\ar[l]_-{S}\ar[r]_-{T}&\C_0
\\&\D_1\C_1\ar[r]_-{\xi_1}\ar[u]_-{T}\ar[d]^-{S}
&\C_1\ar[u]_-{T}\ar[d]_-{S}&\C_2\ar[l]^-{S}\ar[u]_-{T}\ar[dr]^-{\gamma}
\\&\D_0\C_0\ar[r]_-{\xi_0}&\C_0&&\C_1.\ar[ll]^-{S}\ar[uu]_-{T}}
\]
To complete showing that $\e$ is a functor, we need to show that it preserves
identity and compositions.  We then wish to show that it is a map of $\D$-algebras,
for which we need to show that $\e$ preserves $\xi_0$ and $\xi_1$.

For preservation of identities, we wish 
\[
\xymatrix{
\Lhat \C_0\ar[r]^-{I_{\Lhat U\C}}\ar[d]_-{\e_0}&\Lhat U\C_1\ar[d]^-{\e_1}
\\\C_0\ar[r]_-{I_\C}&\C_1}
\]
to commute.  The identity map for $\Lhat M$ is given by 
\[
\xymatrix{
\Lhat M_0=\D_0M_0\ar[rr]^-{(\D_0I_M,I_\D)}&&\D_0M_1\times_{\D_0M_0}\D_1M_0=\Lhat M_1,}
\]
so the identity map for $\Lhat U\C$ is given by
\[
\xymatrix{
\Lhat U\C_0=\D_0\C_0\ar[rr]^-{(\D_0I_{U\C},I_\D)}&&\D_0U\C_1\times_{\D_0\C_0}\D_1\C_0.}
\]
Now the counit map $\e_1:\Lhat U\C_1\to\C_1$ is a composite of a map
to $\C_1\times_{\C_0}\C_1$ with the composition $\gamma_\C$.  Deferring
the composition for the time being, the projection to the first factor of $\C_1$ can
be expressed as the composite
\[
\xymatrix{
\Lhat U\C_1\cong\D_0U\C_1\times_{\D_0\C_0}\D_1\C_0
\ar[r]^-{p_1}&\D_0U\C_1
\ar[r]^-{\D_0\kappa_1}&\D_0\C_1
\ar[r]^-{I_\D}&\D_1\C_1
\ar[r]^-{\xi_1}&\C_1.}
\]
Since this depends only on the first factor $\D_0U\C_1$ of $\Lhat U\C_1$, we can
compose just with the first term for the identity map, and we get the following
commutative diagram:
\[
\xymatrix{
\D_0\C_0\ar[d]_-{\D_0I_{U\C}}\ar[rrr]^-{\xi_0}\ar[dr]^-{\D_0I_\C}
&&&\C_0\ar[d]^-{I_\C}
\\\D_0U\C_1\ar[r]_-{\D_0\kappa_1}
&\D_0\C_1\ar[r]_-{I_\D}
&\D_1\C_1\ar[r]_-{\xi_1}
&\C_1.}
\]
The projection to the second factor of $\C_1$ is given by the composite
\[
\xymatrix{
\Lhat U\C_1\cong\D_0U\C_1\times_{\D_0\C_0}\D_1\C_0
\ar[r]^-{p_2}&\D_1\C_0
\ar[r]^-{\D_1I_\C}&\D_1\C_1
\ar[r]^-{\xi_1}&\C_1.}
\]
This only depends on the second factor $\D_1\C_0$ of $\Lhat U\C_1$, so composing
with just the second term of the identity map for $\Lhat U\C$, we get the
commutative diagram
\[
\xymatrix{
\D_0\C_0\ar[rr]^-{\xi_0}\ar[d]_-{I_\D}
&&\C_0\ar[d]^-{I_\C}
\\\D_1\C_0\ar[r]_-{\D_1I_\C}
&\D_1\C_1\ar[r]_-{\xi_1}
&\C_1.}
\]
Consequently, both composites to $\C_1$ consist of $I_\C\circ\xi_0$.  
Now the diagram we wish to commute collapses simply to
\[
\xymatrix{
\D_0\C_0\ar[r]^-{\xi_0}
&\C_0\ar[r]^-{(I_\C,I_\C)}\ar[dr]_-{I_\C}
&\C_1\times_{\C_0}\C_1\ar[d]^-{\gamma_\C}
\\&&\C_1.}
\]
The triangle commutes because $\gamma_\C$ is unital, so the counit does
preserve the identity maps.

To show that the counit preserves compositions, we need to show that the diagram
\[
\xymatrix@C+10pt{
\Lhat U\C_1\times_{\Lhat U\C_0}\Lhat U\C_1
\ar[r]^-{\e_1\times_{\e_0}\e_1}\ar[d]_-{\gamma_{\Lhat U\C}}
&\C_1\times_{\C_0}\C_1\ar[d]^-{\gamma_\C}
\\\Lhat U\C_1\ar[r]_-{\e_1}&\C_1}
\]
commutes.  It will be convenient for us to have a formula for the counit $\e_1$ on morphisms,
and examination of the defining diagram shows that it can be expressed as follows:
\begin{align*}
&\xymatrix@C+17pt{\Lhat U\C_1\cong\D_0U\C_1\times_{\D_0\C_0}\D_1\C_0
\ar[r]^-{\D_0\kappa_1\times_{\xi_0}1}&\D_0\C_1\times_{\C_0}\D_1\C_0}
\\&\xymatrix@C+10pt{
\relax\ar[r]^-{I_\D\times\D_1I}&\D_1\C_1\times_{\C_0}\D_1\C_1
\ar[r]^-{\xi_1\times\xi_1}&\C_1\times_{\C_0}\C_1\ar[r]^-{\gamma_\C}&\C_1.}
\end{align*}
Meanwhile, the definition  gives the following expression
for the composition in $\Lhat U\C$:
\begin{align*}
&
\D_0U\C_1\times_{\D_0\C_0}\D_1\C_0\times_{\D_0\C_0}\D_0U\C_1\times_{\D_0\C_0}\D_1\C_0
\\&\xymatrix{
\relax&&\D_0U\C_1\times_{\D_0\C_0}\D_1U\C_1\times_{\D_0\C_0}\D_1\C_0
\ar[ll]_-{1\times(\D_1T,S_\D)\times1}^-{\cong}}
\\&\xymatrix{
\relax\ar[rr]^-{1\times(T_\D,\theta)\times1}
&&\D_0U\C_1\times_{\D_0\C_0}\D_0U\C_1\times_{\D_0\C_0}\D_1\C_0\times_{\D_0\C_0}\D_1\C_0}
\\&\xymatrix@C+20pt{
\relax\ar[r]^-{\D_0\gamma_{U\C}\times\gamma_\D}
&\D_0U\C_1\times_{\D_0\C_0}\D_1\C_0.}
\end{align*}

The total diagram we need in order to show that the counit preserves composition
is too large to display as one piece, so we break it up into three pieces.  The first one is the left part,
and looks like this, where the unadorned products have a suppressed subscript $\D_0\C_0$:
\[
\xymatrix@C+63pt{
\D_0U\C_1\times\D_1\C_0\times\D_0U\C_1\times\D_1\C_0
\ar[r]^-{\D_0\kappa_1\times_{\xi_0}1\times\D_0\kappa_1\times_{\xi_0}1}
&\D_0\C_1\times_{\C_0}\D_1\C_0\times\D_0\C_1\times_{\C_0}\D_1\C_0
\\\D_0U\C_1\times\D_1U\C_1\times\D_1\C_0
\ar[u]^-{1\times(\D_1T,S_\D)\times1}_-{\cong}
\ar[r]_-{\D_0\kappa_1\times_{\xi_0}\D_1\kappa_1\times_{\xi_0}1}
\ar[d]_-{1\times(T_\D,\theta)\times1}
&\D_0\C_1\times_{\C_0}\D_1\C_1\times_{\C_0}\D_1\C_0
\ar[u]^-{\cong}_-{1\times(\D_1T,S_\D)\times1}
\\\D_0U\C_1\times\D_0U\C_1\times\D_1\C_0\times\D_1\C_0
\ar[r]^-{\D_0\kappa_1\times_{\xi_0}\D_0\kappa_1\times_{\xi_0}1\times_{\xi_0}1}
&\D_0\C_1\times_{\C_0}\D_0\C_1\times_{\C_0}\D_1\C_0\times_{\C_0}\D_1\C_0
\\\D_0U\C_2\times\D_2\C_0
\ar[u]_-{(\D_0T,\mu\circ\D_0S)\times(T_\D,S_\D)}^-{\cong}
\ar[r]_-{\D_0\kappa_2\times_{\xi_0}1}
\ar[d]_-{\D_0\gamma_{U\C}\times\gamma_\D}
&\D_0\C_2\times_{\C_0}\D_2\C_0
\ar[d]^-{\D_0\gamma_\C\times\gamma_\D}
\\\D_0U\C_1\times\D_1\C_0
\ar[r]_-{\D_0\kappa_1\times_{\xi_0}1}
&\D_0\C_1\times_{\C_0}\D_1\C_0.}
\]
The central piece, which is to be glued to the right of the previous one, is as follows:
\[
\xymatrix@C+15pt{
\D_0\C_1\times_{\C_0}\D_1\C_0\times\D_0\C_1\times_{\C_0}\D_1\C_0
\ar[r]^-{(I_\D\times\D_1I)^2}
&(\D_1\C_1)^4
\ar[r]^-{\xi_1^2\times_{\xi_0}\xi_1^2}
&\C_1^4
\ar[d]^-{1\times\gamma_\C\times1}
\\\D_0\C_1\times_{\C_0}\D_1\C_1\times_{\C_0}\D_1\C_0
\ar[u]^-{1\times(\D_1T,S_\D)\times1}_-{\cong}
\ar[r]_-{I_\D\times1\times\D_1I}
&(\D_1\C_1)^3
\ar[r]_-{\xi_1^3}
&\C_1^3
\\\D_0\C_1\times_{\C_0}\D_0\C_1\times_{\C_0}\D_1\C_0\times_{\C_0}\D_1\C_0
\ar[r]^-{I_\D^2\times(\D_1I)^2}
&(\D_1\C_1)^4
\ar[r]^-{\xi_1^4}
&\C_1^4
\ar[u]_-{1\times\gamma_\C\times1}
\\\D_0\C_2\times_{\C_0}\D_2\C_0
\ar[r]^-{(\D_0T,\D_0S)\times(T_\D,S_\D)}
\ar[d]^-{\D_0\gamma_\C\times\gamma_\D}
&(\D_0\C_1)^2\times_{\C_0}(\D_1\C_0)^2
\ar[r]_-{I_\D^2\times(\D_1I)^2}
&(\D_1\C_1)^4
\ar[u]_-{\xi_1^4}
\\\D_0\C_1\times_{\C_0}\D_1\C_0
\ar[r]_-{I_\D\times\D_1I}
&\D_1\C_1\times_{\C_0}\D_1\C_1
\ar[r]_-{\xi_1\times\xi_1}
&\C_1\times_{\C_0}\C_1.}
\]
And the right hand piece, to be glued to the right of the previous one, is
\[
\xymatrix{
\C_1^4
\ar[rr]^-{\gamma_\C\times\gamma_\C}
\ar[d]_-{1\times\gamma_\C\times1}
&&\C_1^2
\ar[dddd]^-{\gamma_\C}
\\\C_1^3
\ar[dr]^-{\gamma_\C\times1}
\\\C_1^4
\ar[u]^-{1\times\gamma_\C\times1}
\ar[ddr]^-{\gamma_\C\times\gamma_\C}
&\C_1^2
\ar[ddr]^-{\gamma_\C}
\\(\D_1\C_1)^4
\ar[u]^-{\xi_1^4}
\\\C_1\times_{\C_0}\C_1
\ar[r]_-{=}
&\C_1^2
\ar[r]_-{\gamma_\C}
&\C_1.}
\]

We proceed to verify that all the subdiagrams commute, therefore verifying that the
counit preserves compositions.  We begin in the upper left corner,
so at the top of the left portion displayed, in which 
the only part that needs explanation is the part originating
with the $\D_1U\C_1$ in the lower left and ending in the upper right,. But 
both components agree because first $T_{U\C}:=T_\C\circ\kappa_1$ by definition, and
second $S_\D$ is natural.  

Proceeding to the rectangle next to the right, so at the top of
the middle portion displayed, we find that commutation 
only needs to be explained starting with the
middle $\D_1\C_1$ in the lower left corner.  This is equivalent to checking
commutativity of the diagram
\[
\xymatrix@C+10pt{
&\D_1\C_1\ar[dl]_-{(\D_1T,S_\D)}^-{\cong}\ar[dr]^-{=}
\\\D_1\C_0\times_{\D_0\C_0}\D_0\C_1
\ar[r]_-{\D_1I\times I_\D}
&\D_1\C_1\times_{\D_0\C_0}\D_1\C_1
\ar[d]_-{\xi_1\times_{\xi_0}\xi_1}
\ar[r]_-{\gamma_{\D\C}}
&\D_1\C_1\ar[d]^-{\xi_1}
\\&\C_1\times_{\C_0}\C_1\ar[r]_-{\gamma_\C}&\C_1,}
\]
in which the bottom square commutes because $\xi:\D\C\to\C$ is a functor. 
We still need to check the top triangle.  This amounts to checking the following 
diagram:
\[
\xymatrix@C+45pt{
\D_1\C_0\times_{\D_0\C_0}\D_0\C_1
\ar[d]_-{\D_1I\times1}
&\D_1\C_1\ar[l]_-{(\D_1T,S_\D)}^-{\cong}
\ar[d]_-{\D_1I_L}
\ar[dr]^-{=}
\\\D_1\C_1\times_{\D_0\C_0}\D_0\C_1
\ar[d]_-{1\times I_\D}
&\D_1\C_2
\ar[l]_-{(\D_1T,S_\D\circ\D_1S)}^-{\cong}
\ar[r]^-{\D_1\gamma_\C}
\ar[dr]^-{=}
\ar[d]_-{(I_R)_\D}
&\D_1\C_1
\\\D_1\C_1\times_{\D_0\C_0}\D_1\C_1
&\D_2\C_2
\ar[l]_-{(T_\D\circ\D_2T,S_\D\circ\D_2S)}^-{\cong}
\ar[r]_-{\gamma_\D}
&\D_1\C_2.\ar[u]_-{\D_1\gamma_\C}}
\]
The right side of the diagram commutes by the unital properties of composition
in $\C$ and the category object $\D$.  The top left square expands to the
commutative diagram
\[
\xymatrix{
\D_1\C_1\ar[rr]^-{(\D_1T,S_\D)}_-{\cong}
\ar[dd]_-{\D_1I_L}
\ar[dr]^-{(\D_1(I\circ T),1)}
&&\D_1\C_0\times_{\D_0\C_0}\D_0\C_1
\ar[dd]^-{\D_1I\times1}
\\&\D_1\C_1\times_{\D_1\C_0}\D_1\C_1
\ar[dr]^-{1\times_{S_\D}S_\D}_-{\cong}
\\\D_1\C_2
\ar[ur]^-{\cong}_-{(\D_1T,\D_1S)}
\ar[rr]^-{\cong}_-{(\D_1T,S_\D\circ\D_1S)}
&&\D_1\C_1\times_{\D_0\C_0}\D_0\C_1.}
\]
The left triangle commutes by the definition of $I_L$, and the rest of the
diagram is straightforward.  The lower left square 
commutes as a result of its expansion as follows:
\[
\xymatrix{
\D_1\C_2\ar[rr]^-{(\D_1T,S_\D\circ\D_1S)}_-{\cong}
\ar[dd]_-{(I_R)_\D}
\ar[dr]^-{(1,I_\D\circ S_\D)}
&&\D_1\C_1\times_{\D_0\C_0}\D_0\C_1
\ar[dd]^-{1\times I_\D}
\\&\D_1\C_2\times_{\D_0\C_2}\D_1\C_2
\ar[dr]^-{\D_1T\times_{\D_0(T\circ S)}\D_1S}_-{\cong}
\\\D_2\C_2
\ar[ur]^-{\cong}_-{(T_\D,S_\D)}
\ar[rr]^-{\cong}_-{(T_\D\circ\D_1T,S_\D\circ\D_1S)}
&&\D_1\C_1\times_{\D_0\C_0}\D_1\C_1,}
\]
where now the left triangle commutes as a consequence of the
definition of $I_R$ for the category object $\D$.

Returning to the large multi-part diagram, we proceed next to
the second rectangle down on the left, which extends through 
both the left and middle portions.  When restricted to the outer factors,
both ways of traversing the rectangle coincide directly, leaving us with
the part originating with $\D_1U\C_1$.   To verify this part, we
begin with the following diagram:
\[
\xymatrix@C+7pt{
\D_1U\C_1\ar[r]^-{\D_1\kappa_1}\ar[dd]_-{(T_\D,\theta)}
&\D_1\C_1\ar[d]^-{(T_\D,\D_1S)}
\\&\D_0\C_1\times_{\D_0\C_0}\D_1\C_0
\ar[r]^-{I_\D\times\D_1I}
\ar@{.>}[d]^-{1\times_{\xi_0}1}
&\D_1\C_1\times_{\D_0\C_0}\D_1\C_1
\ar[dd]^-{\xi_1\times_{\xi_0}\xi_1}
\\\D_0U\C_1\times_{\D_0\C_0}\D_1\C_0
\ar[r]_-{\D_0\kappa_1\times_{\xi_0}1}
&\D_0\C_1\times_{\C_0}\D_1\C_0\ar[d]^-{I_\D\times\D_1I}
\\&\D_1\C_1\times_{\C_0}\D_1\C_1\ar[r]_-{\xi_1\times\xi_1}
&\C_1\times_{\C_0}\C_1.
}
\]
The reader is warned that with the central dotted arrow included, the
left rectangle does \emph{not} commute, however tempting it may be to
claim it does.  The problem is that when projected to the right factor
$\D_1\C_0$ of the target of the dotted arrow,
there is an ambiguity in the structure map for $\D_1\D_0\C_0$.  The use
of the map $\theta$ dictates the use of $\mu\circ\D_1I_\D$,
while the use of $\kappa_1$ dictates the use of $\D_1\xi_0$.   They aren't the same.  However,
the entire large diagram \emph{does} commute.  The projection onto
the left factor causes no problems; we can even use the dotted arrow,
although it is easy enough to see that the composites coincide.  Projection 
onto the right factor is a consequence of the following diagram:
\[
\xymatrix{
\D_1U\C_1\ar[rr]^-{\D_1\kappa_1}\ar[d]_-{\D_1S}
&&\D_1\C_1\ar[d]^-{\D_1S}
\\\D_1\D_0C_0\ar[rr]^-{\D_1\xi_0}\ar[d]_-{\D_1I_\D}
&&\D_1\C_0\ar[d]^-{\D_1I}
\\\D_1^2\C_0\ar[r]^-{\D_1^2I}\ar[d]_-{\mu}
&\D_1^2\C_1\ar[r]^-{\D_1\xi_1}\ar[d]_-{\mu}
&\D_1\C_1\ar[d]^-{\xi_1}
\\\D_1\C_0\ar[r]_-{\D_1I}
&\D_1\C_1\ar[r]_-{\xi_1}
&\C_1}
\]
The left vertical composite expresses $\theta$, so the diagram says that
although $\theta\ne\D_1S\circ\D_1\kappa_1$, they do
coincide after composing with $\xi_1\circ\D_1I$.  This is the principle reason 
the left part of the previous diagram doesn't commute with the dotted
arrow included. 

Now we paste onto the right of the double rectangle diagram the following one,
which completes the rectangle we wish to verify:
\[
\xymatrix@C+7pt{
\D_1\C_1\ar[d]_-{(T_\D,\D_1S)}\ar[drr]^-{=}
\\\D_0\C_1\times_{\D_0\C_0}\D_1\C_0
\ar[r]_-{I_\D\times\D_1I}
&\D_1\C_1\times_{\D_0\C_0}\D_1\C_1
\ar[r]_-{\gamma_{\D\C}}\ar[d]_-{\xi_1\times_{\xi_0}\xi_1}
&\D_1\C_1
\ar[d]^-{\xi_1}
\\&\C_1\times_{\C_0}\C_1\ar[r]_-{\gamma_\C}&\C_1.}
\]
The square is one we've checked before, and the triangle is analogous to the one
we checked above.  If the reader wants the details, here they are: the triangle
amounts to checking commutativity of
\[
\xymatrix@C+45pt{
 \D_0\C_1\times_{\D_0\C_0}\D_1\C_0
 \ar[d]_-{1\times\D_1I}
 &\D_1\C_1
 \ar[l]_-{(T_\D,\D_1S)}^-{\cong}\ar[d]_-{\D_1I_R}\ar[dr]^-{=}
 \\\D_0\C_1\times_{\D_0\C_0}\D_1\C_1
 \ar[d]_-{I_\D\times1}
 &\D_1\C_2
 \ar[r]^-{\D_1\gamma_\C}\ar[dr]^-{=}\ar[d]_-{(I_L)_\D}\ar[l]_-{(T_\D\circ\D_1T,\D_1S)}^-{\cong}
 &\D_1\C_1
 \\\D_1\C_1\times_{\D_0\C_0}\D_1\C_1
 &\D_1\C_2
 \ar[r]_-{\gamma_\D}\ar[l]_-{(T_\D\circ\D_2T,S_\D\circ\D_2S)}^-{\cong}
 &\D_1\C_2.\ar[u]_-{\D_1\gamma_\C}}
 \]
The top left square follows from being filled in as follows:
\[
\xymatrix{
\D_1\C_1\ar[rr]^-{(T_\D,\D_1S)}\ar[dd]_-{\D_1I_R}
\ar[dr]^-{(1,\D_1(I\circ S))}
&&\D_0\C_1\times_{\D_0\C_0}\D_1\C_0\ar[dd]^-{1\times\D_1I}
\\&\D_1\C_1\times_{\D_1\C_0}\D_1\C_1\ar[dr]^-{T_\D\times_{T_\D}1}_-{\cong}
\\\D_1\C_2\ar[ur]^-{\cong}_-{(\D_1T,\D_1S)}
\ar[rr]_-{(T_\D\circ\D_1T,\D_1S)}^-{\cong}
&&\D_0\C_1\times_{\D_0\C_0}\D_1\C_1.} 
 \]
 The bottom left square fills as follows,
 \[
\xymatrix{
\D_1\C_2\ar[rr]^-{(T_\D\circ\D_1T,\D_1S)}\ar[dd]_-{(I_L)_\D}
\ar[dr]^-{(I_\D\circ T_\D,1)}
&&\D_0\C_1\times_{\D_0\C_0}\D_1\C_1\ar[dd]^-{I_\D\times1}
\\&\D_1\C_2\times_{\D_0\C_2}\D_1\C_2\ar[dr]^-{\D_1T\times_{\D_0(T\circ S)}\D_1S}_-{\cong}
\\\D_2\C_2\ar[ur]^-{\cong}_-{(T_\D,S_\D)}
\ar[rr]_-{(T_\D\circ\D_1T,S_\D\circ\D_1S)}^-{\cong}
&&\D_1\C_1\times_{\D_0\C_0}\D_1\C_1,} 
 \]
 and the rest follows from unital properties of $\gamma_\C$ and $\gamma_{\D\C}$.
 
 We return again to the large multi-part diagram, and proceed to the next 
 rectangle down on the left: the second one that spans both the left and middle
 displays.  The portion starting with $\D_2\C_0$ is unproblematic: both composites
 coincide explicitly.  For the portion starting with $\D_0U\C_2$,
 the part proceeding first via $\D_0T$
 commutes because of the square
 \[
\xymatrix{
\D_0U\C_2\ar[r]^-{\D_0T}\ar[d]_-{\D_0\kappa_2}
&\D_0U\C_1\ar[d]^-{\D_0\kappa_1}
\\\D_0\C_2\ar[r]_-{\D_0T}
&\D_0\C_1,}
\]
to which $\xi_1\circ I_\D$ is then appended.  The part proceeding 
from $\D_0U\C_2$ via $\mu\circ\D_0S$ is more of a challenge,
but it too commutes as a result of the following diagram:
 \[
\xymatrix{
\D_0U\C_2
\ar[rrr]^-{\D_0\kappa_2}
\ar[d]_-{\D_0S}
&&&\D_0\C_2
\ar[d]^-{\D_0S}
\\\D_0^2U\C_1
\ar[r]^-{\D_0^2\kappa_1}
\ar[dd]_-{\mu}
&\D_0^2\C_1
\ar[r]^-{\D_0I_\D}
\ar[dd]_-{\mu}
&\D_0\D_1\C_1
\ar[r]^-{\D_0\xi_1}
\ar[d]_-{I_\D}
&\D_0\C_1
\ar[d]^-{I_\D}
\\&&\D_1^2\C_1
\ar[r]^-{\D_1\xi_1}
\ar[d]_-{\mu}
&\D_1\C_1
\ar[d]^-{\xi_1}
\\\D_0U\C_1
\ar[r]_-{\D_0\kappa_1}
&\D_0\C_1
\ar[r]_-{I_\D}
&\D_1\C_1
\ar[r]_-{\xi_1}
&\C_1.}
\]
In particular, this diagram exhibits the fact that although 
$\D_0\kappa_1\circ\mu\circ\D_0S\ne\D_0S\circ\D_0\kappa_2$, 
they do coincide when composed with $\xi_1\circ I_\D$.  This is 
the reason the rectangle has no shorter fill.

Returning again to the large, multi-part diagram, we have arrived at 
the bottom of the left hand portion, which commutes because of the
definition of $\gamma_{U\C}$.  The subdiagram to its right,
which includes portions in both the middle and right displays, is the 
product of two diagrams, one with source $\D_0\C_2$ and the other
with source $\D_2\C_0$.  The one with source $\D_0\C_2$ commutes
as a result of the following diagram:
\[
\xymatrix@C+7pt{
\D_0\C_1\times_{\D_0\C_0}\D_0\C_1
\ar[r]^-{I_\D\times I_\D}
&\D_1\C_1\times_{\D_0\C_0}\D_1\C_1
\ar[r]^-{\xi_1\times_{\xi_0}\xi_1}
&\C_1\times_{\C_0}\C_1
\\\D_0\C_2
\ar[u]^-{(\D_0T,\D_0S)}_-{\cong}
\ar[r]^-{I_\D^2}
\ar[d]_-{\D_0\gamma_\C}
&\D_2\C_2
\ar[u]^-{(T_\D T,S_\D S)}_-{\cong}
\ar[r]^-{\xi_2}
\ar[d]_-{\gamma_{\D\C}}
&\C_2
\ar[u]_-{(T,S)}^-{\cong}
\ar[d]^-{\gamma_\C}
\\\D_0\C_1\ar[r]_-{I_\D}
&\D_1\C_1\ar[r]_-{\xi_1}
&\C_1.}
\]
And the one starting with $\D_2\C_0$ commutes as a result
of the following analogous diagram:
\[
\xymatrix@C+7pt{
\D_1\C_0\times_{\D_0\C_0}\D_1\C_0
\ar[r]^-{\D_1I\times\D_1I}
&\D_1\C_1\times_{\D_0\C_0}\D_1\C_1
\ar[r]^-{\xi_1\times_{\xi_0}\xi_1}
&\C_1\times_{\C_0}\C_1
\\\D_2\C_0
\ar[u]^-{(T_\D,S_\D)}_-{\cong}
\ar[r]^-{\D_2I^2}
\ar[d]_-{\gamma_\D}
&\D_2\C_2
\ar[u]^-{(T_\D T,S_\D S)}_-{\cong}
\ar[r]^-{\xi_2}
\ar[d]_-{\gamma_{\D\C}}
&\C_2
\ar[u]_-{(T,S)}^-{\cong}
\ar[d]^-{\gamma_\C}
\\\D_1\C_0\ar[r]_-{\D_1I}
&\D_1\C_1\ar[r]_-{\xi_1}
&\C_1.}
\]

The remaining parts of the large, multi-part diagram are both in the right
hand display, and are a consequence of associativity for $\gamma_\C$.  
We may conclude that the counit preserves composition, and is therefore
a functor.

We still must show that the counit respects the $\D$-algebra structure 
on $\Lhat M$, that is, that the diagram
\[
\xymatrix{
\D_1\Lhat U\C_1
\ar[r]^-{\xi_1}
\ar[d]_-{\D_1\e}
&\Lhat U\C_1
\ar[d]^-{\e}
\\\D_1\C_1
\ar[r]_-{\xi_1}
&\C_1}
\]
commutes.  We expand using the definitions of $\e$ and $\xi_1$ and
fill as follows:
\[
\xymatrix@C+7pt{
\D_1\D_0U\C_1\times_{\D_1\D_0\C_0}\D_1^2\C_0
\ar[r]^-{T_\D\times_{T_\D}1}
\ar[d]_-{\D_1\D_0\kappa_1\times_{\D_0\xi_0}1}
&\D_0^2U\C_1\times_{\D_0^2\C_0}\D_1^2\C_0
\ar[r]^-{\mu\times_\mu\mu}
\ar[d]_-{\D_0^2\kappa_1\times_{\D_0\xi_0}1}
&\D_0U\C_1\times_{\D_0\C_0}\D_1\C_0
\ar[d]^-{\D_0\kappa\times_{\xi_0}1}
\\\D_1\D_0\C_1\times_{\D_1\C_0}\D_1^2\C_0
\ar[r]^-{T_\D\times_{T_\D}1}
\ar[d]_-{\D_1I_\D\times\D_1^2I}
&\D_0^2\C_1\times_{\D_0\C_0}\D_1^2\C_0
\ar[r]^-{\mu\times_{\xi_0}\mu}
\ar[d]_-{I_\D^2\times\D_1^2I}
&\D_0\C_1\times_{\C_0}\D_1\C_0
\ar[d]^-{I_\D\times\D_1I}
\\\D_1^2\C_1\times_{\D_1\C_0}\D_1^2\C_1
\ar[r]^-{(I_\D\circ T_\D)\times_{T_\D}1}
\ar[d]_-{\D_1\xi_1\times\D_1\xi_1}
&\D_1^2\C_1\times_{\D_0\C_0}\D_1^2\C_1
\ar[r]^-{\mu\times_{\xi_0}\mu}
\ar[d]_-{\D_1\xi_1\times\D_1\xi_1}
&\D_1\C_1\times_{\C_0}\D_1\C_1
\ar[d]^-{\xi_1\times\xi_1}
\\\D_1\C_1\times_{\D_1\C_0}\D_1\C_1
\ar[r]^-{(I_\D\circ T_\D)\times_{T_\D}1}
\ar[d]_-{\D_1\gamma_\C}
&\D_1\C_1\times_{\D_0\C_0}\D_1\C_1
\ar[dl]^-{\gamma_{\D\C}}
\ar[r]^-{\xi_1\times_{\xi_0}\xi_1}
&\C_1\times_{\C_0}\C_1
\ar[d]^-{\gamma_\C}
\\\D_1\C_1
\ar[rr]_-{\xi_1}
&&\C_1.}
\]
The one part of this fill that may be less than clear is the triangle
in the lower left corner, but this is a consequence of the left
unital property of the category structure on $\D$.  In particular, we
can expand the triangle as follows:
\[
\xymatrix{
\D_1\C_1\times_{\D_1\C_0}\D_1\C_1
\ar[rr]^-{(I_\D\circ T_\D)\times_{T_\D}1}
&&\D_1\C_1\times_{\D_0\C_0}\D_1\C_1
\\\D_1\C_2
\ar[u]^-{(\D_1T,\D_1S)}_-{\cong}
\ar[r]^-{(I_\D\circ T_\D,1)}
\ar[dr]_-{=}
&\D_1\C_2\times_{\D_0\C_2}\D_1\C_2
\ar[ur]^-{(\D_1T,\D_1S)}_-{\cong}
&\D_2\C_2
\ar[u]_-{(T_\D T,S_\D S)}^-{\cong}
\ar[l]_-{(T_\D,S_\D)}^-{\cong}
\ar[dl]_-{\gamma_\D}
\ar[d]^-{\gamma_{\D\C}}
\\&\D_1\C_2
\ar[r]_-{\D_1\gamma_\C}
&\D_1\C_1.}
\]
Here the lower part of the diagram captures the left unital property
for the category structure on $\D$, and therefore the diagram commutes.
It follows that the counit preserves the $\D$-algebra structure on $\Lhat M$.

\section{The adjunction structure III: commuting triangles}

We wish to verify the adjunction triangles
\[
\xymatrix{
U\C\ar[r]^-{\eta}\ar[dr]_-{=}&U\Lhat U\C\ar[d]^-{U\e}
\\&U\C}
\quad\text{and}\quad
\xymatrix{
\Lhat M\ar[r]^-{\Lhat\eta}\ar[dr]_-{=}&\Lhat U\Lhat M\ar[d]^-{\e}
\\&\Lhat M.}
\]
Both of these are straightforward when restricted to objects: we have
the actual adjunction between sets and $\D_0$-algebras.  We therefore
concentrate on the maps on morphisms, and start with the first triangle.

In general we have $U\C_1\cong\C_1\times_{\C_0}\D_0\C_0$, so replacing
$\C$ with $\Lhat U\C$, we have
\[
U\Lhat U\C_1\cong\D_0U\C_1\times_{\D_0\C_0}\D_1\C_0\times_{\D_0\C_0}\D_0^2\C_0,
\]
where the structure map on the $\D_0^2\C_0$ on the end is given by $\mu$.  
Next, in general we have 
\[
U\Lhat M_1\cong \D_0M_1\times_{\D_0M_0}\D_1M_0\times_{\D_0M_0}\D_0^2M_0,
\]
and the unit $\eta_1:M_1\to U\Lhat M_1$ is given by $(\eta_\D,I_\D\circ S,\D_0\eta\circ S)$.
Replacing $M_1$ with $U\C_1$ and rewriting $\D_0U\C_1$ as $\D_0\C_1\times_{\D_0\C_0}\D_0^2\C_0$,
we see that the unit $\eta_1:U\C_1\to U\Lhat U\C_1$ is a map
\[
\C_1\times_{\C_0}\D_0\C_0\to
\D_0\C_1\times_{\D_0\C_0}\D_0^2\C_0\times_{\D_0\C_0}\D_1\C_0\times_{\D_0\C_0}\D_0^2\C_0
\]
given by the four components $(\eta_\D\circ p_1,\eta_\D\circ p_2,I_\D\circ p_2,\D_0\eta\circ p_2)$.
The structure maps on the two copies of $\D_0^2\C_0$ need particular care: the first one has 
structure map of target type $\D_0\xi_0$ and source type $\mu$, while the second one has structure
map of target type $\mu$.

Next, the counit $\e_1:\Lhat U\C_1\to\C_1$ consists of the following composite:
\begin{align*}
\Lhat U\C_1
&\xymatrix@C+15pt{
\relax\ar[r]^-{\cong}&\D_0U\C_1\times_{\D_0\C_0}\D_1\C_0
\ar[r]^-{\D_0\kappa_1\times_{\xi_0}1}&\D_0\C_1\times_{\C_0}\D_1\C_0}
\\&\xymatrix@C+15pt{
\relax\ar[r]^-{I_\D\times\D_1I}&\D_1\C_1\times_{\C_0}\D_1\C_1
\ar[r]^-{\xi_1\times\xi_1}&\C_1\times_{\C_0}\C_1
\ar[r]^-{\gamma_\C}&\C_1.}
\end{align*}
Since $U\Lhat U\C_1\cong\Lhat U\C_1\times_{\Lhat U\C_0}\D_0\Lhat U\C_0
=\Lhat U\C_1\times_{\D_0\C_0}\D_0^2\C_0$, what $U\e_1$ does is $\e_1$ on 
the first factor, and $\D_0\e_0=\D_0\xi_0$ on the second factor.  Writing $\xi_0^2$ for
either way of composing in the commutative square
\[
\xymatrix{
\D_0^2\C_0\ar[r]^-{\D_0\xi_0}\ar[d]_-{\mu}
&\D_0\C_0\ar[d]^-{\xi_0}
\\\D_0\C_0\ar[r]_-{\xi_0}&\C_0,}
\]
we can express the induced map $U\e_1$ as the following composite,
in which the $\xi_0^2$ simply has the effect of deleting the $\D_0^2\C_0$ term:
\begin{align*}
&\D_0\C_1\times_{\D_0\C_0}\D_0^2\C_0\times_{\D_0\C_0}\D_1\C_0\times_{\D_0\C_0}\D_0^2\C_0
\\&\xymatrix{
\relax\ar[rr]^-{1\times_{\xi_0}\xi_0^2\times_{\xi_0}1\times1}
&&\D_0\C_1\times_{\C_0}\D_1\C_0\times_{\D_0\C_0}\D_0^2\C_0}
\\&\xymatrix{
\relax\ar[rr]^-{I_\D\times\D_1I\times1}
&&\D_1\C_1\times_{\C_0}\D_1\C_1\times_{\D_0\C_0}\D_0^2\C_0}
\\&\xymatrix{
\relax\ar[rr]^-{\xi_1\times\xi_1\times_{\xi_0}\D_0\xi_0}
&&\C_1\times_{\C_0}\C_1\times_{\C_0}\D_0\C_0}
\\&\xymatrix{
\relax\ar[rr]^-{\gamma_\C\times1}
&&\C_1\times_{\C_0}\D_0\C_0=U\C_1.}
\end{align*}
We proceed by identifying the composite up to, but not including, the very last step,
which is the use of $\gamma_\C\times1$, so the target is $\C_1\times_{\C_0}\C_1\times_{\C_0}\D_0\C_0$.
if we project to the first factor of $\C_1$, tracing through shows that the composite 
depends only on the factor of $\C_1$ in the initial source, and consists of either way
around the triangle
\[
\xymatrix@C+20pt{
\C_1\ar[r]^-{\eta_{\D_0}}\ar[dr]^-{\eta_{\D_1}}\ar[ddr]_-{=}
&\D_0\C_1\ar[d]^-{I_\D}
\\&\D_1\C_1\ar[d]^-{\xi_1}
\\&\C_1.}
\]
The top part commutes because the unit $\eta:1\to\D$ is natural in the simplicial structure of $\D$, 
and the bottom because $\C_1$ is a $\D_1$-algebra.

If we project to the second factor of $\C_1$, we see that the map depends only
on the factor of $\D_0\C_0$ in the initial source, and consists of either way around
the commutative rectangle
\[
\xymatrix{
\D_0\C_0\ar[r]^-{I_\D}\ar[d]_-{\xi_0}
&\D_1\C_0\ar[r]^-{\D_1I}
&\D_1\C_1\ar[d]^-{\xi_1}
\\\C_0\ar[rr]_-{I}&&\C_1.}
\]
Consequently, the image in this factor of $\C_1$ consists entirely of identity maps,
so won't change the first factor upon composing.  We see therefore that the composite
including $\gamma_\C\times1$ is the identity on the factor of $\C_1$.

Projecting to the factor of $\D_0\C_0$, the composite consists of either way
around the commuting triangle
\[
\xymatrix{
\D_0\C_0\ar[r]^-{\D_0\eta}\ar[dr]_-{=}&\D_0^2\C_0\ar[d]^-{\D_0\xi_0}
\\&\D_0\C_0.}
\]
The entire composite is the identity, and we see that the first adjunction
triangle commutes.

For the second adjunction triangle, we need explicit expressions for both
$\Lhat\eta_1:\Lhat M_1\to\Lhat U\Lhat M_1$ and $\e_1:\Lhat U\Lhat M_1\to\Lhat M_1$.
We start by recalling $\Lhat M_1\cong\D_0M_1\times_{\D_0M_0}\D_1M_0$, 
and also $\Lhat M_0:=\D_0M_0$.  Then since in general $U\C_1=\C_1\times_{\C_0}\D_0\C_0$,
we have
\[
U\Lhat M_1\cong\Lhat M_1\times_{\D_0M_0}\D_0^2M_0
\cong\D_0M_1\times_{\D_0M_0}\D_1M_0\times_{\D_0M_0}\D_0^2M_0,
\]
where the structure map for $\D_0^2M_0$ is $\mu$.  We also have $U\Lhat M_0=\Lhat M_0=\D_0M_0$.
Applying $\Lhat $ to all this, we get
\begin{gather*}
\Lhat U\Lhat M_1=\D_0U\Lhat M_1\times_{\D_0U\Lhat M_0}\D_1U\Lhat M_0
\\\cong\D_0^2M_1\times_{\D_0^2M_0}\D_0\D_1M_0\times_{\D_0^2M_0}\D_0^3M_0
\times_{\D_0^2M_0}\D_1\D_0M_0.
\end{gather*}
(We must take care to note that the structure maps on $\D_0^3M_0$ are $\mu$ for
target type, but $\D_0\mu$ for source type.)
Now $\eta:1\to U\Lhat$ has components $\eta_0:M_0\to U\Lhat M_0=\D_0M_0$ given by
$\eta_{\D_0}$, the unit for the monad $\D_0$, and $\eta_1:M_1\to U\Lhat M_1$ given by
\[
M_1\to\D_0M_1\times_{\D_0M_0}\D_1M_0\times_{\D_0M_0}\D_0^2M_0
\]
given by component maps $(\eta_{\D_0},I_\D\circ S,\D_0\eta\circ S)$.  Applying $\Lhat$ 
to these data, we find the induced map $\Lhat\eta_1:\Lhat M_1\to\Lhat U\Lhat M_1$ is
given by the map
\[
\D_0M_1\times_{\D_0M_0}\D_1M_0
\to\D_0^2M_1\times_{\D_0^2M_0}\D_0\D_1M_0\times_{\D_0^2M_0}\D_0^3M_0
\times_{\D_0^2M_0}\D_1\D_0M_0
\]
with components $(\D_0\eta\circ p_1,\D_0(I_\D\circ S)\circ p_1,
\D_0(\D_0\eta\circ S)\circ p_1,\D_1\eta_{\D_0}\circ p_2)$.  In particular, the projection
to the first three terms depend only on the first term of the source, and the projection to 
the fourth depends only on the last term in the source.

For the counit map, we have in general the composite given above, namely
\begin{align*}
\Lhat U\C_1
&\xymatrix@C+15pt{
\relax\ar[r]^-{\cong}&\D_0U\C_1\times_{\D_0\C_0}\D_1\C_0
\ar[r]^-{\D_0\kappa_1\times_{\xi_0}1}&\D_0\C_1\times_{\C_0}\D_1\C_0}
\\&\xymatrix@C+15pt{
\relax\ar[r]^-{I_\D\times\D_1I}&\D_1\C_1\times_{\C_0}\D_1\C_1
\ar[r]^-{\xi_1\times\xi_1}&\C_1\times_{\C_0}\C_1
\ar[r]^-{\gamma_\C}&\C_1.}
\end{align*}
When we have $\C=\Lhat M$, this becomes in particular the following composite,
in which $\mu^2:\D_0^3M_0\to\D_0M_0$ is either way around the evident square,
and simply has the effect of deleting the $\D_0^3M_0$ term:
\begin{align*}
&\D_0^2M_1\times_{\D_0^2M_0}\D_0\D_1M_0\times_{\D_0^2M_0}\D_0^3M_0
\times_{\D_0^2M_0}\D_1\D_0M_0
\\&\xymatrix{
\relax\ar[rr]^-{1\times1\times_\mu\mu^2\times_\mu1}
&&\D_0^2M_1\times_{\D_0^2M_0}\D_0\D_1M_0\times_{\D_0M_0}\D_1\D_0M_0}
\\&\xymatrix@C+15pt{
\relax\ar[rr]^-{I_\D\times_{I_\D}I_\D\times(\D_1\D_0I,I_\D)}
&&\D_1\D_0M_1\times_{\D_1\D_0M_0}\D_1^2M_0\times_{\D_0M_0}\D_1\D_0M_1
\times_{\D_1\D_0M_0}\D_1^2M_0}
\\&\xymatrix{
\relax\ar[rr]^-{[(\mu\circ T_\D\times_{\mu\circ T_\D}\mu]^2}
&&\D_0M_1\times_{\D_0M_0}\D_1M_0\times_{\D_0M_0}\D_0M_1\times_{\D_0M_0}\D_1M_0}
\\&\xymatrix{
\relax&&\D_0M_1\times_{\D_0M_0}\D_1M_1\times_{\D_0M_0}\D_1M_0
\ar[ll]_-{1\times(\D_1T,S_\D)\times1}^-{\cong}}
\\&\xymatrix{
\relax\ar[rr]^-{1\times(T_\D,\theta)\times1}
&&\D_0M_1\times_{\D_0M_0}\D_0M_1\times_{\D_0M_0}\D_1M_0\times_{\D_0M_0}\D_1M_0}
\\&\xymatrix{
\relax\ar[rr]^-{\D_0\gamma_M\times\gamma_\D}
&&\D_0M_1\times_{\D_0M_0}\D_1M_0=\Lhat M_1.}
\end{align*}
We analyze the composite of this with $\Lhat \eta$ by first stopping halfway through
the counit at the term that gives $\Lhat M_2$, namely
\[
\D_0M_1\times_{\D_0M_0}\D_1M_0\times_{\D_0M_0}\D_0M_1\times_{\D_0M_0}\D_1M_0.
\]
Projecting to the first term, and restricting to the first term in the source of $\Lhat \eta$, we 
get the following commutative diagram:
\[
\xymatrix{
\D_0M_1\ar[r]^-{\D_0\eta}\ar[ddrr]_-{=}
&\D_0^2M_1\ar[r]^-{I_\D}\ar[dr]_-{=}
&\D_1\D_0M_1\ar[d]^-{T_\D}
\\&&\D_0^2M_1\ar[d]^-{\mu}
\\&&\D_0M_1.}
\]
Projecting to the second term, and again restricting to $\D_0M_1$ in the source, we
have the diagram
\[
\xymatrix{
\D_0M_1\ar[r]^-{\D_0S}
&\D_0^2M_0\ar[r]^-{\D_0I_\D}\ar[d]_-{\mu}
&\D_0\D_1M_0\ar[r]^-{I_\D}
&\D_1^2M_0\ar[d]^-{\mu}
\\&\D_0M_0\ar[rr]_-{I_\D}
&&\D_1M_0.}
\]
Notice that $\mu\circ\D_0S$ is the structure map for $\D_0M_1$ of source type.  Projecting
to the third term, and now restricting to $\D_1M_0$, we get the commuting diagram
\[
\xymatrix{
\D_1M_0\ar[r]^-{\D_1I_\D}\ar[d]_-{T_\D}
&\D_1\D_0M_0\ar[r]^-{\D_1\D_0I}\ar[d]^-{T_\D}
&\D_1\D_0M_1\ar[d]^-{T_\D}
\\\D_0M_0\ar[r]^-{\D_0\eta_\D}\ar[dr]_-{=}
&\D_0^2M_0\ar[r]^-{\D_0^2I}\ar[d]^-{\mu}
&\D_0^2M_1\ar[d]^-{\mu}
\\&\D_0M_0\ar[r]_-{\D_0I}
&\D_0M_1.}
\]
Notice that $T_\D$ is the structure map for $\D_1M_0$ of target type.  Projecting
to the fourth term, and still restricting to $\D_1M_0$, we get
\[
\xymatrix@C+25pt{
\D_1M_0\ar[r]^-{\D_1\eta_{\D_0}}\ar[dr]^-{\D_1\eta_{\D_1}}\ar[ddr]_-{=}
&\D_1\D_0M_0\ar[d]^-{\D_1I_\D}
\\&\D_1^2M_0\ar[d]^-{\mu}
\\&\D_1M_0.}
\]
The net result of these calculations is that the composite map to $\Lhat M_2$ can 
be written as the map $1\times(\mu\circ\D_0S,T_\D)\times1$, as follows:
\begin{gather*}
\D_0M_1\times_{\D_0M_0}\D_1M_0
\\\xymatrix@C+10pt{
\relax
\ar[rr]^-{1\times(\mu\circ\D_0S,T_\D)\times1}
&&\D_0M_1\times_{\D_0M_0}\D_1M_0\times_{\D_0M_0}\D_0M_1\times_{\D_0M_0}\D_1M_0,}
\end{gather*}
where the source of the middle terms is the $\D_0M_0$ over which the pullback
defining the overall source is defined.  We examine what happens to these middle terms
as we go on in the counit: the next term is the backwards arrow (which is however
an isomorphism) and the forwards one, both from $\D_1M_1$.  We obtain the following
diagram, which we claim commutes:
\[
\xymatrix@C+15pt{
\D_0M_0\ar[r]^-{(I_\D,\D_0I)}\ar[dr]^-{I_\D I}\ar[ddr]_-{(\D_0I,I_\D)}
&\D_1M_0\times_{\D_0M_0}\D_0M_1
\\&\D_1M_1\ar[u]_-{(\D_1T,S_\D)}^-{\cong}\ar[d]^-{(T_\D,\theta)}
\\&\D_0M_1\times_{\D_0M_0}\D_1M_0.}
\]
The only part of the diagram that isn't immediate from properties of categories and multicategories
is the part involving $\theta$, but that follows from the following diagram,
in which the bottom row displays $\theta$ explicitly:
\[
\xymatrix@R+15pt{
&\D_0M_0\ar[dl]_-{I_\D I}\ar[d]_-{I_\D\eta_{\D_0}\!\!}\ar[dr]_-{I_\D\eta_{\D_1}\!\!\!\!\!\!}\ar[drr]^-{I_\D}
\\\D_1M_1\ar[r]_-{\D_1S}
&\D_1\D_0M_0\ar[r]_-{\D_1I_\D}
&\D_1^2M_0\ar[r]_-{\mu}
&\D_1M_0.}
\]
It now follows that the middle terms in the part of the counit just before 
the composition at the end,
\begin{gather*}
\D_0M_1\times_{\D_0M_0}\D_0M_1\times\D_1M_0\times_{\D_0M_0}\D_1M_0
\\\xymatrix{
\relax\ar[rr]^-{\D_0\gamma_M\times\gamma_\D}
&&\D_0M_1\times_{\D_0M_0}\D_1M_0,}
\end{gather*}
consist only of identity elements, and so contribute nothing after composition.
Since the terms on the end are given by identity maps in the total composition, 
we conclude that the second adjunction triangle commutes.

\section{The actual left adjoint: category structure}

In this section we begin the proof of Theorem \ref{maintheorem}.  Recall that 
the morphism set $LM_1$ of the actual left adjoint is given by a coequalizer of two arrows 
to the morphism set $\Lhat M_1$, while the objects are just the objects $\D_0M_0$
of $\Lhat M$.

We first give the category structure of $LM$.  The unit map $I:LM_0\to LM_1$
is just the unit map $\Lhat M_0\to\Lhat M_1$ composed with the quotient map
$\Lhat M_1\to LM_1$.  Source and target maps are also inherited from
$\Lhat M$, since the coequalization takes place in the middle of the 
two terms defining $\Lhat M$, while the source and target maps are on
the ``outside.''  The significant problem is verifying that the composition map
on $\Lhat M$ induces a map on $LM$.  For this, we need to see that
composition in $\Lhat M$ preserves equivalence classes.  Our strategy for
doing so is to lift the composition in $\Lhat M$ to the source of the coequalizer
defining $LM_1$ in two different ways so that the composition in $\Lhat M$
then descends to the coequalizer.  Here are the details.  

We abbreviate the coequalizer diagram as
\[
\xymatrix{
QM\ar@<.5ex>[r]^-{\psi_*}\ar@<-.5ex>[r]_-{\gamma_*}
&\Lhat M_1\ar[r]
&LM_1,}
\]
where we have written $QM$ for $\D_0M_1\times_{\D_0^2M_0}\D_0\D_1M_0\times_{\D_0M_0}\D_1M_0$,
$\psi_*$ for the map $\D_0\psi_1\times1$, and $\gamma_*$
for the map $1\times\gamma_\D((\mu\circ I_\D)\times1)$.
We will construct a map $\Gamma_R:QM\times_{\D_0M_0}\Lhat M_1\to QM$ for which 
both top and bottom choices of horizontal arrow in the following diagram commute:
\[
\xymatrix{
QM\times\Lhat M_1
\ar[d]_-{\Gamma_R}
\ar@<-.5ex>[r]_-{\gamma_*\times1}\ar@<.5ex>[r]^-{\psi_*\times1}
&\Lhat M_1\times\Lhat M_1
\ar[d]^-{\gamma_{\Lhat M}}
\\QM
\ar@<.5ex>[r]^-{\psi_*}\ar@<-.5ex>[r]_-{\gamma_*}
&\Lhat M_1.}
\]
This will show that equivalent morphisms in the left slot give rise to equivalent
composites.  We also construct a map $\Gamma_L:\Lhat M_1\times_{\D_0M_0}QM\to QM$
for which both top and bottom choices of horizontal arrow in the diagram
\[
\xymatrix{
\Lhat M_1\times QM
\ar[d]_-{\Gamma_L}
\ar@<-.5ex>[r]_-{1\times\gamma_*}\ar@<.5ex>[r]^-{1\times\psi_*}
&\Lhat M_1\times\Lhat M_1
\ar[d]^-{\gamma_{\Lhat M}}
\\QM
\ar@<.5ex>[r]^-{\psi_*}\ar@<-.5ex>[r]_-{\gamma_*}
&\Lhat M_1}
\]
commute.  This will show that equivalent morphisms in the right slot give 
rise to equivalent composites, and therefore the composition descends 
to $LM$.  

We define $\Gamma_R$
as the following composite, which requires some
explanation about the map $\tilde{\chi}$:
\[
\xymatrix{
\D_0M_1\times_{\D_0^2M_0}\D_0\D_1M_0\times_{\D_0M_0}\D_1M_0
\times_{\D_0M_0}\D_0M_1\times_{\D_0M_0}\D_1M_0
\ar[d]^-{1^2\times\chi\times1}
\\\D_0M_1\times_{\D_0^2M_0}\D_0\D_1M_0\times_{\D_0M_0}\D_0M_1
\times_{\D_0M_0}\D_1M_0\times_{\D_0M_0}\D_1M_0
\ar[d]^-{1\times\tilde{\chi}\times1^2}
\\\D_0M_1\times_{\D_0^2M_0}\D_0^2M_1\times_{\D_0^2M_0}\D_0\D_1M_0
\times_{\D_0M_0}\D_1M_0\times_{\D_0M_0}\D_1M_0
\ar[d]^-{\D_0\gamma_M\times1\times\gamma_\D}
\\\D_0M_1\times_{\D_0^2M_0}\D_0\D_1M_0\times_{\D_0M_0}\D_1M_0}
\]
To explain $\tilde{\chi}$, we observe that both squares in the diagram
\[
\xymatrix{
\D_0\D_1M_1\ar[r]^-{\D_0\D_1T}\ar[d]_-{\D_0S_\D}
&\D_0\D_1M_0\ar[d]^-{\D_0S_\D}
\\\D_0^2M_1\ar[r]^-{\D_0^2T}\ar[d]_-{\mu}
&\D_0^2M_0\ar[d]^-{\mu}
\\\D_0M_1\ar[r]_-{\D_0T}&\D_0M_0}
\]
are pullbacks, the bottom one since $\D_0$ is Cartesian, and the top one 
since $\D_0$ preserves pullbacks (being Cartesian.)  Consequently, we
can define $\tilde{\chi}$ as the composite
\[
\xymatrix@C+35pt{
\D_0\D_1M_0\times_{\D_0M_0}\D_0M_1
&\D_0\D_1M_1
\ar[l]^-{\cong}_-{(\D_0\D_1T,\mu\circ\D_0S_\D)}\ar[r]^-{(\D_0T_\D,\D_0\theta)}
&\D_0^2M_1\times_{\D_0^2M_0}\D_0\D_1M_0.}
\]

There are now two diagrams we wish to commute involving
$\Gamma_R$.  The first one, using the
arrows $\psi_*\times1$ and $\psi_*$ in the 
diagram given above, doesn't involve the final factor of $\D_1M_0$ before
it is multiplied, so can be omitted from the diagram.  With subscripts of
$\D_0M_0$ again suppressed, we wish the following to commute:
\[
\xymatrix{
\D_0M_1\times_{\D_0^2M_0}\D_0\D_1M_0\times\D_1M_0\times\D_0M_1
\ar[r]^-{\D_0\psi\times1^2}
\ar[d]_-{1^2\times\chi}
&\D_0M_1\times\D_1M_0\times\D_0M_1
\ar[d]^-{1\times\chi}
\\\D_0M_1\times_{\D_0^2M_0}\D_0\D_1M_0\times\D_0M_1\times\D_1M_0
\ar[r]_-{\D_0\psi\times1^2}
\ar[d]_-{1\times\tilde{\chi}\times1}
&\D_0M_1\times\D_0M_1\times\D_1M_0
\ar[dd]^-{\D_0\gamma_M\times1}
\\\D_0M_1\times_{\D_0^2M_0}\D_0^2M_1\times_{\D_0^2M_0}
\D_0\D_1M_0\times\D_1M_0
\ar[d]_-{\D_0\gamma_M\times1^2}
\\\D_0M_1\times_{\D_0^2M_0}\D_0\D_1M_0\times\D_1M_0
\ar[r]_-{\D_0\psi\times1}
&\D_0M_1\times\D_1M_0.}
\]
However, the top rectangle commutes by naturality, so we are left
with the bottom rectangle, in which the $\D_1M_0$ at the end plays no role.
We are left with wanting the rectangle
\[
\xymatrix{
\D_0M_1\times_{\D_0^2M_0}\D_0\D_1M_0\times\D_0M_1
\ar[r]^-{\D_0\psi\times1}
\ar[d]_-{1\times\tilde{\chi}}
&\D_0M_1\times\D_0M_1
\ar[dd]^-{\D_0\gamma_M}
\\\D_0M_1\times_{\D_0^2M_0}\D_0^2M_1\times_{\D_0^2M_0}\D_0\D_1M_0
\ar[d]_-{\D_0\gamma_M\times1}
\\\D_0M_1\times_{\D_0^2M_0}\D_0\D_1M_0
\ar[r]_-{\D_0\psi}
&\D_0M_1}
\]
to commute.  In order to verify this, we need some observations about
the presheaf actions on $M_1$ and $M_2$.

First, the target map $M_2\to M_1$ is supposed to be a map of
presheaves over the functor $\D\e:\D^2(*)\to\D(*)$.  In order to lift 
this to the actions by $\D^2M_0$ and $\D M_0$, we need the 
following diagram, in which the dotted arrow labeled $\e_*$ is induced as a map
of pullbacks from the left tall rectangle to the right one:
\[
\xymatrix{
&\D_0M_1
\ar[rr]^-{\D_0T}
\ar[dd]_(.7){\D_0S}
&&\D_0M_0
\ar[dddd]^-{\D_0\e}
\\\D_0M_1\times_{\D_0^2M_0}\D_1^2M_0
\ar@{.>}[rr]^(.7){\e_*}
\ar[dd]
\ar[ur]
&&\D_1M_0
\ar[ur]^-{T_\D}
\ar[dddd]^-{\D_1\e}
\\&\D_0^2M_0
\ar[dd]_-{\D_0^2\e}
\\\D_1^2M_0
\ar[ur]^-{T_\D^2}
\ar[dd]_-{\D_1^2\e}
\\&\D_0^2(*)
\ar[rr]^(.3){\D_0\e}
&&\D_0(*).
\\\D_1^2(*)
\ar[ur]^-{T_\D^2}
\ar[rr]_-{\D_1\e}
&&\D_1(*)
\ar[ur]_-{T_\D}}
\]
We can glue the following pullback cube to the top of the diagram,
where the map $1\times\e_*$ is so named since it really expresses
the composite
\[
\xymatrix{
M_2\times_{\D_0^2M_0}\D_1^2M_0
\ar[r]^-{\cong}
&M_1\times_{\D_0M_0}\D_0M_1\times_{\D_0^2M_0}\D_1^2M_0
\ar[r]^-{1\times\e_*}
&M_1\times_{\D_0M_0}\D_1M_0:
}
\]
\[
\xymatrix{
&M_2\ar[rr]^-{T}\ar[dd]_(.3){S}
&&M_1\ar[dd]^-{S}
\\M_2\times_{\D_0^2M_0}\D_1^2M_0
\ar[ur]^-{p_1}
\ar@{.>}[rr]^(.65){1\times\e_*}
\ar[dd]_-{S\times1}
&&M_1\times\D_1M_0
\ar[ur]^-{p_1}
\ar[dd]^(.3){p_2}
\\&\D_0M_1\ar[rr]^(.3){D_0T}
&&\D_0M_0.
\\\D_0M_1\times_{\D_0^2M_0}\D_1^2M_0
\ar@{.>}[rr]_-{\e_*}
\ar[ur]^-{p_1}
&&\D_1M_0\ar[ur]_-{T_\D}}
\]
We can express the fact that the target map $T:M_2\to M_1$ is
a map of presheaves over $\D\e$ by the diagram
\[
\xymatrix@C+15pt{
M_2\times_{\D_0^2(*)}\D_1^2(*)
\ar[r]^-{T\times_{\D_0\e}\D_1\e}
\ar[d]_-{\psi_2}
&M_1\times_{\D_0(*)}\D_1(*)
\ar[d]^-{\psi_1}
\\M_2\ar[r]_-{T}
&M_1.}
\]
But the previous two cubes allow us to rewrite the top
arrow to give the equivalent diagram
\[
\xymatrix{
M_2\times_{\D_0^2M_0}\D_1^2M_0
\ar[r]^-{\cong}
\ar[d]_-{\psi_2}
&M_1\times_{\D_0M_0}\D_0M_1\times_{\D_0^2M_0}\D_1^2M_0
\ar[r]^-{1\times\e_*}
&M_1\times_{\D_0M_0}\D_1M_0
\ar[d]^-{\psi_1}
\\M_2
\ar[rr]_-{T}
&&M_1.
}
\]
Meanwhile, the equivariance of $S:M_2\to\D_0M_1$ as a presheaf over $\D^2M_0$
can be expressed by the commutative rectangle
\[
\xymatrix{
M_2\times_{\D_0^2M_0}\D_1^2M_0
\ar[r]^-{S\times1}
\ar[d]_-{\psi_2}
&\D_0M_1\times_{\D_0^2M_0}\D_1^2M_0
\ar[d]^-{\D\psi_1}
\\M_2\ar[r]_-{S}&\D_0M_1,}
\]
and the equivariance of $\gamma_M:M_2\to M_1$ as a map over $\mu:\D^2(*)\to\D(*)$
lifts directly to $\mu:\D^2M_0\to\D M_0$ in the rectangle
\[
\xymatrix{
M_2\times_{\D_0^2M_0}\D_1^2M_0
\ar[r]^-{\gamma_M\times_\mu\mu}
\ar[d]_-{\psi_2}
&M_1\times_{\D_0M_0}\D_1M_0
\ar[d]^-{\psi_1}
\\M_2\ar[r]_-{\gamma_M}
&M_1.}
\]

We return now to our previous reduction of the first diagram, which
we expand and augment as follows:
\[
\xymatrix@C-45pt{
\D_0M_1\times_{\D_0^2M_0}\D_0\D_1M_0\times\D_0M_1
\ar[rr]^-{\D_0\psi_1\times1}
&&\D_0M_1\times\D_0M_1
\\\D_0M_1\times_{\D_0^2M_0}\D_0\D_1M_0\times_{\D_0^2M_0}\D_0^2M_1
\ar[u]^-{1\times_\mu\mu}_-{\cong}
\ar[rr]^-{\D_0\psi\times1}
&&\D_0M_1\times_{\D_0^2M_0}\D_0^2M_1
\ar[u]_-{1\times_\mu\mu}^-{\cong}
\ar[dddd]^-{\D_0\gamma_M}
\\\D_0M_1\times_{\D_0^2M_0}\D_0\D_1M_1
\ar[u]^-{1\times(\D_0\D_1T,\D_0S_\D)}_-{\cong}
\ar[dr]^-{\,\,\,\,\,\,\,\,\,\,\,\,\,\,\,\,\,1\times(\D_0T_\D,\D_0\D_1(I_\D\circ S))}
\ar[dd]_-{1\times(\D_0T_\D,\D_0\theta)}
\\&\D_0M_1\times_{\D_0^2M_0}\D_0^2M_1\times_{\D_0^3M_0}\D_0\D_1^2M_0
\ar[uur]^-{\D_0\psi_2}
\ar[dl]^-{\,\,\,\,\,\,\,\,1^2\times_{\D_0\mu}\D_0\mu}
\\\D_0M_1\times_{\D_0^2M_0}\D_0^2M_1\times_{\D_0^2M_0}\D_0\D_1M_0
\ar[d]_-{\D_0\gamma_M\times1}
\\\D_0M_1\times_{\D_0^2M_0}\D_0\D_1M_0
\ar[rr]_-{\D_0\psi_1}
&&\D_0M_1.}
\]
All the sub-diagrams here already commute, with the exception of the 
second one down: the one with the second $\D_0\psi_1\times1$ as its top arrow.  This
sub-diagram can be turned upside down, and have a $\D_0$ stripped off, to become
the following diagram of our desire, in which, as usual, we suppress subscript $\D_0M_0$'s
on products:
\[
\xymatrix@C+15pt{
M_1\times\D_1M_1
\ar[r]^-{1\times(T_\D,\D_1(I_\D\circ S))}\ar[d]_-{1\times(\D_1T,S_\D)}
&M_1\times\D_0M_1\times_{\D_0^2M_0}\D_1^2M_0
\ar[d]^-{\psi_2}
\\M_1\times\D_1M_0\times\D_0M_1
\ar[r]_-{\psi_1\times1}
&M_1\times\D_0M_1.}
\]
This diagram commutes precisely when it does so after projecting onto each of the
factors in the target $M_1\times_{\D_0M_0}\D_0M_1$
(which is the same thing as $M_2$), to which we can then apply
the equivariance diagrams we already know about $\psi_2$, namely
\begin{gather*}
\xymatrix{
M_1\times\D_0M_1\times_{\D_0^2M_0}\D_1^2M_0
\ar[r]^-{1\times\e_*}\ar[d]_-{\psi_2}
&M_1\times\D_1M_0
\ar[d]^-{\psi_1}
\\M_1\times\D_0M_1
\ar[r]_-{T=p_1}
&M_1}
\\
\text{and}
\\
\xymatrix{
M_1\times\D_0M_1\times_{\D_0^2M_0}\D_1^2M_0
\ar[r]^-{p_{23}}\ar[d]_-{\psi_2}
&\D_0M_1\times_{\D_0^2M_0}\D_1^2M_0
\ar[d]^-{\D\psi_1}
\\M_1\times\D_0M_1
\ar[r]_-{S=p_2}
&\D_0M_1,}
\end{gather*}
where $\D\psi_1$ indicates the action induced by the $\D^2(*)$ action
on $\D_0M_1$.  Combining these with the previous diagram, we find
that the diagram for $\Gamma_R$ with the top arrows commutes as
a consequence of the following two lemmas.

\begin{lemma}
The diagram
\[
\xymatrix@C+40pt{
M\times\D_1M_1
\ar[r]^-{1\times\e_*(T_\D,\D_1(I_\D\circ S))}
\ar[d]_-{1\times\D_1 T}
&M\times\D_1M_0
\ar[d]^-{\psi_1}
\\M\times\D_1M_0
\ar[r]_-{\psi_1}
&M_1}
\]
commutes.
\end{lemma}

\begin{proof}
It suffices to show that
\[
\D_1T=\e_*(T_\D,\D_1(I_\D\circ S)):\D_1M_1\to\D_1M_0.
\]
From the diagram defining $\e_*$, since its target is the pullback
\[
\D_1M_0\cong\D_0M_0\times_{\D_0(*)}\D_1(*),
\] 
we see that
$\e_*$ is completely determined by the two composites
\[
T_\D\circ\e_*=\D_0T\circ p_1
\quad\text{and}\quad
(\D_1\e)\circ\e_*=\D_1\e\circ\D_1^2\e\circ p_2.
\]
We can therefore compose our desired equality with $T_\D$
and $\D_1\e$ to see if it is true.  Composing with $T_\D$, we find that
\begin{gather*}
T_\D\circ\e_*(T_\D,\D_1(I_\D\circ S))
\\=\D_0T\circ p_1(T_\D,\D_1(I_\D\circ S))
\\=\D_0T\circ T_\D=T_\D\circ\D_1T.
\end{gather*}
Composing with $\D_1\e$, we have the diagram
\[
\xymatrix{
\D_1M_1\ar[r]^-{\D_1S}\ar[dd]_-{\D_1T}\ar[ddrr]^-{\D_1\e}
&\D_1\D_0M_0\ar[r]^-{\D_1I_\D}
&\D_1^2M_0\ar[d]^-{\D_1^2\e}
\\&&\D_1^2(*)\ar[d]^-{\D_1\e}
\\\D_1M_0\ar[rr]_-{\D_1\e}
&&\D_1(*),}
\]
which commutes since $\e:M_1\to*$ is terminal.  We may conclude
that the diagram commutes.
\end{proof}

\begin{lemma}
The diagram
\[
\xymatrix@C+30pt{
\D_1M_1
\ar[r]^-{(T_\D,\D_1(I_\D\circ S))}
\ar[dr]_-{S_\D}
&\D_0M_1\times_{\D_0^2M_0}\D_1^2M_0
\ar[d]^-{\D\psi_1}
\\&\D_0M_1}
\]
commutes.
\end{lemma}

\begin{proof}
We need some generalities about how the action
of $\D^2(M_0^\delta)$ on $\D_0M_1$ can be expressed.  Whenever we have
a presheaf $X$ over a category $\C$, we know that $\D_0X$ is a presheaf
over $\D\C$.  This is because $\D$ preserves target covers, which are equivalent
to presheaf structures on the objects of the source category of a target cover.
In particular, we may conclude that 
\[
\xymatrix{
[\D(\C\int X)]_1\ar[r]^-{T}\ar[d]
&[\D(\C\int X)]_0\ar[d]
\\(\D\C)_1\ar[r]_-{T}
&(\D\C)_0}
\]
is a pullback diagram, which with a little unpacking becomes
\[
\xymatrix{
\D_1X\times_{\D_1\C_0}\D_1\C_1\ar[r]^-{T}\ar[d]
&\D_0X\ar[d]^-{\D_0\e}
\\\D_1\C_1\ar[r]_-{T_\D T}
&\D_0\C_0.}
\]
This in turn expands to
\[
\xymatrix{
\D_1X\times_{\D_1\C_0}\D_1\C_1\ar[r]^-{T_\D}\ar[d]
&\D_0X\times_{\D_0\C_0}\D_0\C_1\ar[r]^-{\D_0 T}\ar[d]
&\D_0X\ar[d]^-{\D_0\e}
\\\D_1\C_1\ar[r]_-{T_\D}
&\D_0\C_1\ar[r]_-{\D_0 T}
&\D_0\C_0,}
\]
both of whose squares are pullbacks, the left one because of the general
principle that whenever we have a function $f:X\to Y$, there is a 
canonical isomorphism
\[
\D_0X\times_{\D_0Y}\D_1Y\cong\D_1X.
\]
(This is a consequence of the fact that $f^\delta:X^\delta\to Y^\delta$ 
is a cover, and therefore so is $\D(f^\delta)$.)  Now tracing definitions shows
that the action of $\D_1\C_1$ on $\D_0X$ is given by the maps
\begin{align*}
\D_0X\times_{\D_0\C_0}\D_1\C_1
&\cong\D_0X\times_{\D_0\C_0}\D_0\C_1\times_{\D_0\C_1}\D_1\C_1
\\&\cong\D_1X\times_{\D_1\C_0}\D_1\C_1\cong\D_1(X\times_{\C_0}\C_1)
\\
&\xymatrix{
\relax
\ar[r]^-{\D_1\xi}&\D_1X\ar[r]^-{S_\D}&\D_0X.}
\end{align*}
In our case, we have $M_1$ in the role of $X$ and $\D(M_0^\delta)$ in the
role of $\C$, so we can express the action of $\D_1^2M_0$ on $\D_0M_1$
as the composite
\begin{align*}
\D_0M_1\times_{\D_0^2M_0}\D_1^2M_0
&\cong\D_0M_1\times_{\D_0^2M_0}\D_0\D_1M_0\times_{\D_0\D_1M_0}\D_1^2M_0
\\&\cong\D_1M_1\times_{\D_1\D_0M_0}\D_1^2M_0\cong\D_1(M_1\times_{\D_0M_0}\D_1M_0)
\\&\xymatrix{
\relax\ar[r]^-{\D_1\psi_1}&\D_1M_1\ar[r]^-{S_\D}&\D_0M_1,}
\end{align*}
which is the map denoted $\D\psi_1$ in the diagram 
we are currently trying to verify.  We can now do so by means of the following diagram:
\[
\xymatrix{
\D_1M_1\ar[r]^-{(T_\D,\D_1S)}_-{\cong}\ar[dr]_-{\cong}
&\D_0M_1\times_{\D_0^2M_0}\D_1\D_0M_0
\ar[r]^-{1\times\D_1I_\D}\ar[d]^-{\cong}
&\D_0M_1\times_{\D_0^2M_0}\D_1^2M_0
\ar[d]^-{\cong}
\\&\D_1M_1\times_{\D_1\D_0M_0}\D_1\D_0M_0
\ar[r]^-{1\times\D_1I_\D}\ar[dr]_-{\cong}
&\D_1M_1\times_{\D_1\D_0M_0}\D_1^2M_0
\ar[d]^-{\D_1\psi_1}
\\&&\D_1M_1\ar[d]^-{S_\D}
\\&&\D_0M_1.}
\]
The two diagonal arrows are the canonical isomorphisms, so compose to the identity,
and the only nontrivial sub-diagram is the lower triangle, which commutes 
because of the unit property of the action $\psi_1$.  We have completed
verifying coherence of the first of our four directions for descent of the 
product on $\Lhat M$ to $LM$.
\end{proof}

The 
diagram involving $\Gamma_R$ with the bottom choice of horizontal arrows
again doesn't involve 
the last factor of $\D_1M_0$, but omitting it still results in 
a diagram that is too wide to fit on the page.  The left half is
the following,
\[
\xymatrix@C+10pt{
\D_0M_1\times_{\D_0^2M_0}\D_0\D_1M_0
\ar[r]^-{1\times_\mu(\mu\circ I_\D)\times1^2}
\ar[d]_-{1^2\times\chi}
&\D_0M_1\times(\D_1M_0)^2\times\D_0M_1
\ar[d]^-{1^2\times\chi}
\\\D_0M_1\times_{\D_0^2M_0}\D_0\D_1M_0\times\D_0M_1\times\D_1M_0
\ar[r]^-{1\times_\mu(\mu\circ I_\D)\times1^2}
\ar[d]_-{1\times\tilde\chi\times1}
&\D_0M_1\times\D_1M_0\times\D_0M_1\times\D_1M_0
\ar[d]^-{1\times\chi\times1}
\\\D_0M_1\times_{\D_0^2M_0}\D_0^2M_1\times_{\D_0^2M_0}\D_0\D_1M_0\times\D_1M_0
\ar[r]^-{1\times_\mu\mu\times_\mu(\mu\circ I_\D)\times1}
\ar[d]_-{\D_0\gamma_M\times1^2}
&(\D_0M_1)^2\times(\D_1M_0)^2
\ar[d]^-{\D_0\gamma_M\times1}
\\\D_0M_1\times_{\D_0^2M_0}\D_0\D_1M_0\times\D_1M_0
\ar[r]_-{1\times_\mu(\mu\circ I_\D)\times1}
&\D_0M_1\times(\D_1M_0)^2,}
\]
and the right half, to be pasted to the right of the above part, is
\[
\xymatrix{
\D_0M_1\times(\D_1M_0)^2\times\D_0M_1
\ar[r]^-{1\times\gamma_\D\times1}
\ar[d]^-{1^2\times\chi}
&\D_0M_1\times\D_1M_0\times\D_0M_1
\ar[dd]^-{1\times\chi}
\\\D_0M_1\times\D_1M_0\times\D_0M_1\times\D_1M_0
\ar[d]^-{1\times\chi\times1}
\\(\D_0M_1)^2\times(\D_1M_0)^2
\ar[r]^-{1^2\times\gamma_\D}
\ar[d]^-{\D_0\gamma_M\times1}
&(\D_0M_1)^2\times\D_1M_0
\ar[d]^-{\D_0\gamma_M\times1}
\\\D_0M_1\times(\D_1M_0)^2
\ar[r]_-{1\times\gamma_\D}
&\D_1M_1\times\D_1M_0.}
\]
The top part of the right half is the product on the right with $\D_0M_1$ of a
diagram we verified as part of showing that $\gamma_{\Lhat M}$ was
associative, and the bottom part is a simple naturality diagram.  
The top part of the left half is also a naturality diagram, and the rest
of the left half doesn't involve the last factor of $\D_1M_0$.  Deleting it,
we are left with verifying the following diagram:
\[
\xymatrix@C+35pt{
\D_0M_1\times_{\D_0^2M_0}\D_0\D_1M_0\times\D_0M_1
\ar[r]^-{1\times_\mu(\mu\circ I_\D)\times1}
\ar[d]_-{1\times\tilde\chi}
&\D_0M_1\times\D_1M_0\times\D_1M_1
\ar[d]^-{1\times\chi}
\\\D_0M_1\times_{\D_0^2M_0}\D_0^2M_1\times_{\D_0^2M_0}\D_0\D_1M_0
\ar[r]^-{1\times_\mu\mu\times_\mu(\mu\circ I_\D)}
\ar[d]_-{\D_0\gamma_M\times1}
&\D_0M_1\times\D_0M_1\times\D_1M_0
\ar[d]^-{\D_0\gamma_M\times1}
\\\D_0M_1\times_{\D_0^2M_0}\D_0\D_1M_0
\ar[r]_-{1\times_\mu(\mu\circ I_\D)}
&\D_0M_1\times\D_1M_0.}
\]
The bottom half of the diagram follows from the fact that the map
\[
\xymatrix{
\D_0M_1\times_{\D_0^2M_0}\D_0^2M_1
\ar[r]^-{1\times_\mu\mu}
&\D_0M_1\times\D_1M_0}
\]
is a canonical isomorphism between two expressions for $\D_0M_2$:
the relevant diagram is as follows, and consists of two stacked pullbacks:
\[
\xymatrix{
\D_0M_2
\ar[r]^-{\D_0T}
\ar[d]_-{\D_0S}
&\D_0M_1
\ar[d]^-{\D_0S}
\\\D_0^2M_1
\ar[r]^-{\D_0^2T}
\ar[d]_-{\mu}
&\D_0^2M_0
\ar[d]^-{\mu}
\\\D_0M_1
\ar[r]_-{\D_0T}
&\D_0M_0.}
\]

Expanding the top half of the previous diagram 
by deleting the initial $\D_0M_1$, which plays no role, and
using the definitions of $\tilde\chi$
and $\chi$, we get the following fill:
\[
\xymatrix{
\D_0\D_1M_0
\ar[r]^-{I_\D\times1}
&\D_1^2M_0\times\D_0M_1
\ar[r]^-{\mu\times1}
&\D_1M_0\times\D_0M_1
\\\D_0\D_1M_1
\ar[r]^-{I_\D}
\ar[u]^-{(\D_0\D_1T,\D_0S_\D)}_-{\cong}
\ar[d]_-{(\D_0T_\D,\D_0\theta)}
&\D_1^2M_1
\ar[r]^-{\mu}
\ar[u]_-{(\D_1^2T,S_\D\circ\mu)}
\ar[d]^-{(T_\D^2,\D_1\theta)}
&\D_1M_1
\ar[u]_-{(\D_1T,S_\D)}^-{\cong}
\ar[d]^-{(T_\D,\theta)}
\\\D_0^2M_1\times_{\D_0^2M_0}\D_0\D_1M_0
\ar[r]_-{1\times I_\D}
&\D_0^2M_1\times_{\D_0^2M_0}\D_1^2M_0
\ar[r]_-{\mu\times_\mu\mu}
&\D_0M_1\times\D_1M_0.}
\]
Here the projection of the top left square to the factor of $\D_1^2M_0$
is just a naturality diagram, and the projection to $\D_0M_1$ follows from the
following:
\[
\xymatrix{
\D_0\D_1M_1
\ar[r]^-{I_\D}
\ar[dd]_-{D_0S_\D}
&\D_1^2M_1
\ar[ddl]_-{S_\D^2}
\ar[d]^-{\mu}
\\&\D_1M_1
\ar[d]^-{S_\D}
\\\D_0^2M_1
\ar[r]_-{\mu}
&\D_0M_1.}
\]
The upper right square, projected to $\D_1M_0$, is again a naturality diagram,
and the projection to $\D_0M_1$ is trivial: both composites are $S_\D\circ\mu$.
The lower left square, projected to $\D_0^2M_1$, follows from the fact that
$T_\D\circ I_\D=1$, and the projection to $\D_1^2M_0$ follows from the fact that
$I_\D$ is a natural map from $\D_0$ to $\D_1$; this is part of the category structure
on $\{\D_0,\D_1\}$.  In the lower right square, projection to $\D_0M_1$ is again
a naturality diagram, while projection to $\D_1M_0$ uses the expansion of $\theta$
as the composite
\[
\xymatrix{
\D_1M_1
\ar[r]^-{\D_1S}
&\D_1\D_0M_0
\ar[r]^-{\D_1I_\D}
&\D_1^2M_0
\ar[r]^-{\mu}
&\D_1M_0}
\]
to give the following fill:
\[
\xymatrix{
\D_1^2M_1
\ar[r]^-{\D_1^2S}
\ar[d]_-{\mu}
&\D_1^2\D_0M_0
\ar[r]^-{\D_1^2I_\D}
\ar[d]^-{\mu}
&\D_1^3M_0
\ar[r]^-{\D_1\mu}
\ar[d]^-{\mu}
&\D_1^2M_0
\ar[d]^-{\mu}
\\\D_1M_1
\ar[r]^-{\D_1S}
&\D_1\D_0M_0
\ar[r]^-{\D_1I_\D}
&\D_1^2M_0
\ar[r]^-{\mu}
&\D_1M_0.}
\]
This completes the verification of the diagram involving $\Gamma_R$
with the bottom choice of horizontal arrows.  It follows that
composition in $\Lhat M$ preserves equivalence classes in
the left slot.

\begin{figure}[b]
\hrule
\[
\xymatrix{
\D_0M_1\times\D_1M_0\times\D_0M_1
\times_{\D_0^2M_0}\D_0\D_1M_0\times\D_1M_0
\\\D_0M_1\times\D_1M_1\times_{\D_0^2M_0}
\D_0\D_1M_0\times\D_1M_0
\ar[u]^-{\cong}_-{1\times(\D_1T,S_\D)\times1^2}
\ar[d]_-{\cong}^-{1\times(T_\D,\D_1S)\times1^2}
\\\D_0M_1\times\D_0M_1\times_{\D_0^2M_0}\D_1\D_0M_0
\times_{\D_0^2M_0}\D_0\D_1M_0\times
\D_1M_0
\\\D_0M_1\times\D_0M_1\times_{\D_0^2M_0}
\D_1^2M_0\times\D_1M_0
\ar[u]^-{\cong}_-{1^2\times(\D_1T_\D,S_\D)\times1}
\ar[d]_-{\cong}^-{1^2\times(T_\D,\D_1S_\D)\times1}
\\\D_0M_1\times\D_0M_1\times_{\D_0^2M_0}
\D_0\D_1M_0\times_{\D_0^2M_0}\D_1\D_0M_0\times\D_1M_0
\\\D_0M_1\times_{\D_0^2M_0}\D_0^2M_1
\times_{\D_0^3M_0}\D_0^2\D_1M_0\times_{\D_0^2M_0}
\D_1\D_0M_0\times\D_1M_0
\ar[u]_-{1\times_\mu\mu\times_\mu\mu\times1^2}^-{\cong}
\ar[d]^-{\D_0\gamma_M\times_{\D_0\mu}1^3}
\\\D_0M_1\times_{\D_0^2M_0}\D_0^2\D_1M_0
\times_{\D_0^2M_0}\D_1\D_0M_0\times\D_1M_0
\ar[d]^-{1\times\D_0(\mu\circ I_\D)\times_\mu(\mu\circ\D_1I_\D)\times1}
\\\D_0M_1\times_{\D_0^2M_0}\D_0\D_1M_0\times\D_1M_0\times\D_1M_0
\ar[d]^-{1^2\times\gamma_\D}
\\\D_0M_1\times_{\D_0^2M_0}\D_0\D_1M_0\times\D_1M_0.}
\]
\caption{}
\label{Fi:secondrearrangement}
\end{figure}

We display $\Gamma_L$ as the composite
in figure \ref{Fi:secondrearrangement}, 
where the unadorned products have the usual suppressed subscript
$\D_0M_0$.  
Some care is necessary in reading this picture; in particular, the
temptation to use the composition $\D_0\gamma_M:\D_0M_1\times\D_0M_1\to\D_0M_1$, 
suggested by the fact that $\D_0M_1\times\D_0M_1\cong\D_0M_2$, must be avoided,
since the map does not preserve the structure map of source type to $\D_0^2M_0$.
Instead, we must use the somewhat more complicated formalism displayed.  The
preservation of structure maps along the $\mu$ in the middle of the next to the last arrow of the figure is
a consequence of the following diagram:
\[
\xymatrix{
&\D_0^2\D_1M_0
\ar[r]^-{\D_0I_\D}
\ar[d]_-{\D_0^2S_\D}
&\D_0\D_1^2M_0
\ar[r]^-{\D_0\mu}
\ar[dl]^-{\D_0S_\D^2}
&\D_0\D_1M_0
\ar[d]^-{S_\D}
\\&\D_0^3M_0\ar[rr]^-{\D_0\mu}
\ar[d]_-{\mu}
&&\D_0^2M_0
\ar[d]^-{\mu}
\\\D_1\D_0M_0
\ar[r]^-{T_\D}
\ar[dr]_-{\D_0I_\D}
&\D_0^2M_0
\ar[rr]^-{\mu}
&&\D_0M_0.
\\&\D_1^2M_0
\ar[rr]_-{\mu}
\ar[u]_-{T_\D^2}
&&\D_1M_0
\ar[u]_-{T_\D}}
\]

The 
diagram with $\Gamma_L$ and the top choice of horizontal arrows 
now consists of a left column as
in figure \ref{Fi:secondrearrangement}, connected to a right column as below
by an arrow $1^2\times\D_0\psi_1\times1$ on top and one on the bottom labeled
$\D_0\psi_1\times1$.  The right column expresses the multiplication
in $\Lhat M$, as in the first diagram, but we expand it a bit
for the purposes of filling in; it appears as follows:
\[
\xymatrix{
\D_0M_1\times\D_1M_0\times\D_0M_1\times\D_1M_0
\\\D_0M_1\times\D_1M_1\times\D_1M_0
\ar[u]^-{\cong}_-{1\times(\D_1T,S_\D)\times1}
\ar[d]^-{1\times(T_\D,\D_1S)\times1}_-{\cong}
\\\D_0M_1\times\D_0M_1\times_{\D_0^2M_0}\D_1\D_0M_0\times\D_1M_0
\\\D_0M_1\times_{\D_0^2M_0}\D_0^2M_1\times_{\D_0^2M_0}\D_1\D_0M_0\times\D_1M_0
\ar[u]^-{\cong}_-{1\times_\mu\mu\times1^2}
\ar[d]^-{\D_0\gamma_M\times_\mu(\mu\circ I_\D)\times1}
\\\D_0M_1\times\D_1M_0\times\D_1M_0
\ar[d]^-{1\times\gamma_\D}
\\\D_0M_1\times\D_1M_0}
\]

We now proceed to fill in 
the diagram with $\Gamma_L$ and the top choice of horizontal arrows,
and start at the top, where the first
part of the fill doesn't involve either the first factor of $\D_0M_1$ or the last factor
of $\D_1M_0$, so we omit them.  The claim is now that the diagram
\[
\xymatrix{
\D_1M_1\times_{\D_0^2M_0}\D_0\D_1M_0\ar[d]_-{\cong}^-{(\D_1T,S_\D)}
&\D_1M_1\times_{\D_1\D_0M_0}\D_1^2M_0
\ar[l]_-{1\times_{\S_\D}S_\D}^-{\cong}
\ar[r]^-{\D_1\psi_1}
&\D_1M_1
\ar[d]^-{(\D_1T,S_\D)}_-{\cong}
\\\D_1M_0\times\D_0M_1\times_{\D_0^2M_0}\D_0\D_1M_0
\ar[rr]_-{1\times\D_0\psi_1}
&&\D_1M_0\times\D_0M_1}
\]
commutes; it has to be turned upside down in order to fit into the larger diagram
(which we haven't actually written down completely, for reasons of size.)  The 
projection to the first factor of the lower right corner commutes since it's $\D_1$
applied to the following diagram, which expresses the fact that $\psi_1$
preserves targets:
\[
\xymatrix{
M_1\times\D_1M_0\ar[r]^-{\psi_1}\ar[dr]_-{T\circ p_1}
&M_1\ar[d]^-{T}
\\&M_0.}
\]
Projection to the second factor amounts just to commutativity of the naturality
diagram
\[
\xymatrix{
\D_1(M_1\times\D_1M_0)\ar[r]^-{\D_1\psi_1}\ar[d]_-{S_\D}
&\D_1M_1\ar[d]^-{S_\D}
\\\D_0(M_1\times\D_1M_0)\ar[r]_-{\D_0\psi_1}
&\D_0M_1.}
\]
We may conclude that the first part of the fill commutes.

The next part of the fill also doesn't need the first or last factors, and consists of
the following diagram:
\[
\xymatrix@C-30pt{
\D_1M_1\times_{\D_0^2M_0}\D_0\D_1M_0
\ar[d]_-{(T_\D,\D_1S)\times1}^-{\cong}
&\D_1M_1\times_{\D_1\D_0M_0}\D_1^2M_0
\ar[l]_-{1\times_{S_\D}S_\D}^-{\cong}
\ar[r]^-{\D_1\psi_1}
\ar@/^3.5pc/[ddl]^-{T_\D\times_{T_\D}1}_-{\cong}
&\D_1M_1
\ar[ddd]^-{(T_\D,\D_1S)}
\\\D_0M_1\times_{\D_0^2M_0}\D_1\D_0M_0\times_{\D_0^2M_0}\D_0\D_1M_0
\\\D_0M_1\times_{\D_0^2M_0}\D_1^2M_0
\ar[u]^-{1\times(\D_1T_\D,S_\D)}_-{\cong}
\ar[d]^-{1\times(T_\D,\D_1S_\D)}_-{\cong}
\\\D_0M_1\times_{\D_0^2M_0}\D_0\D_1M_0\times_{\D_0^2M_0}\D_1\D_0M_0
\ar[rr]_-{\D_0\psi_1\times1}
&&\D_0M_1\times_{\D_0^2M_0}\D_1\D_0M_0.}
\]
The upper left portion consists entirely of isomorphisms, and commutes
by projecting to each of the three factors of the target, which is second down
on the left column.  The projections to the outer two factors commute trivially,
and the one to the central $\D_1\D_0M_0$ commutes since they coincide
with the structure maps in the pullback $\D_1M_1\times_{\D_1\D_0M_0}\D_1^2M_0$
that gives the source of the sub-diagram.  We are left for this portion of the fill
with the diagram
\[
\xymatrix@C+10pt{
\D_1M_1\times_{\D_1\D_0M_0}\D_1^2M_0
\ar[r]^-{\D_1\psi_1}
\ar[d]_-{T_\D\times_{T_\D}1}^-{\cong}
&\D_1M_1
\ar[dd]^-{(T_\D,\D_1S)}_-{\cong}
\\\D_0M_1\times_{\D_0^2M_0}\D_1^2M_0
\ar[d]_-{1\times(T_\D,\D_1S_\D)}^-{\cong}
\\\D_0M_1\times_{\D_0^2M_0}\D_0\D_1M_0\times_{\D_0^2M_0}\D_1\D_0M_0
\ar[r]_-{\D_0\psi_1\times1}
&\D_0M_1\times_{\D_0^2M_0}\D_1\D_0M_0.}
\]
Projecting to the factor of $\D_0M_1$ in the target, we see that
the diagram reduces to the naturality diagram
\[
\xymatrix{
\D_1(M_1\times\D_1M_0)\ar[r]^-{\D_1\psi_1}\ar[d]_-{T_\D}
&\D_1M_1\ar[d]^-{T_\D}
\\\D_0(M_1\times\D_1M_0)\ar[r]_-{\D_0\psi_1}
&\D_0M_1.}
\]
And projection to the factor of $\D_1\D_0M_0$ consists of $\D_1$ applied
to the diagram
\[
\xymatrix{
M_1\times\D_1M_0\ar[r]^-{\psi_1}\ar[d]_-{p_2}
&M_1\ar[d]^-{S}
\\\D_1M_0\ar[r]_-{S_\D}
&\D_0M_0,}
\]
which records the fact that $\psi_1$ preserves source structure maps.  This
part of the fill is now complete.

The next part of the fill is just the naturality diagram
\[
\xymatrix{
\D_0^2M_1\times_{\D_0^3M_0}\D_0^2\D_1M_0
\ar[r]^-{\D_0^2\psi_1}
\ar[d]_-{\mu\times_\mu\mu}
&\D_0^2M_1\ar[d]^-{\mu}
\\\D_0M_1\times_{\D_0^2M_0}\D_0\D_1M_0
\ar[r]_-{\D_0\psi_1}
&\D_0M_1}
\]
turned upside down and with $\D_0M_1$ attached at the front,
and $\D_1\D_0M_0\times\D_1M_0$ at the back.  

We come next to the following diagram, in which a $\D_1M_0$ at
the back has been suppressed:
\[
\xymatrix@C-55pt{
\D_0M_1\times_{\D_0^2M_0}\D_0^2M_1\times_{\D_0^3M_0}
\D_0^2\D_1M_0\times_{\D_0^2M_0}\D_1\D_0M_0
\ar[dr]^-{1\times\D_0^2\psi_1\times1}
\ar[dd]_-{1^2\times\D_0I_\D\times1}
\\&\D_0M_1\times_{\D_0^2M_0}\D_0^2M_1\times_{\D_0^2M_0}\D_1\D_0M_0
\ar[dd]^-{\D_0\gamma_M\times_\mu(\mu\circ\D_1I_\D)}
\\\D_0M_1\times_{\D_0^2M_0}\D_0^2M_1\times_{\D_0^3M_0}
\D_0\D_1^2M_0\times_{\D_0^2M_0}\D_1\D_0M_0
\ar[d]_-{\D_0\gamma_M\times_{\D_0\mu}\D_0\mu\times_\mu(\mu\circ\D_1I_\D)}
\\\D_0M_1\times_{\D_0^2M_0}\D_0\D_1M_0\times\D_1M_0
\ar[r]_-{\D_0\psi_1\times1}
&\D_0M_1\times\D_1M_0.}
\]
While it is clear that the part involving the last factor commutes, since it is just the
composite
\[
\xymatrix{
\D_1\D_0M_0\ar[r]^-{\D_1I_\D}&\D_1^2M_0\ar[r]^-{\mu}&\D_1M_0}
\]
either way around the diagram, it is displayed here to emphasize that
the structure maps are preserved by the arrows in the diagram, especially
the second one in the left column.  In the source of the arrow, the structure
map of source type for $\D_0\D_1^2M_0$ is given, by examining the previous
fill diagram, by the composite
\[
\xymatrix{
\D_0\D_1^2M_0\ar[r]^-{\D_0S_\D^2}&\D_0^3M_0\ar[r]^-{\D_0\mu}&\D_0^2M_0.}
\]
This does map to the structure map for $\D_0\D_1M_0$ in the next term 
because of the commutativity of
\[
\xymatrix{
\D_0\D_1^2M_0
\ar[r]^-{\D_0S_\D^2}
\ar[d]_-{\D_0\mu}
&\D_0^3M_0
\ar[r]^-{\mu}
\ar[d]^-{\D_0\mu}
&\D_0^2M_0
\ar[d]^-{\mu}
\\\D_0\D_1M_0
\ar[r]_-{\D_0S_\D}
&\D_0^2M_0
\ar[r]_-{\mu}
&\D_0M_0,
}
\]
in which the structure map for $\D_0\D_1^2M_0$ now appears at a right angle.  

Stripping off the last term, we wish to see that
\[
\xymatrix@C+10pt{
\D_0M_1\times_{\D_0^2M_0}\D_0^2M_1\times_{\D_0^3M_0}
\D_0^2\D_1M_0
\ar[r]^-{1\times\D_0^2\psi_1}
\ar[d]_-{1^2\times\D_0I_\D}
&\D_0M_1\times_{\D_0^2M_0}\D_0^2M_1
\ar[dd]^-{\D_0\gamma_M}
\\\D_0M_1\times_{\D_0^2M_0}\D_0^2M_1\times_{\D_0^3M_0}
\D_0\D_1^2M_0
\ar[d]_-{\D_0\gamma_M\times_{\D_0\mu}\D_0\mu}
\\\D_0M_1\times_{\D_0^2M_0}\D_0\D_1M_0
\ar[r]_-{\D_0\psi_1}
&\D_0M_1
}
\]
commutes.  However, this can now have a $\D_0$ stripped off, resulting in
the diagram of our desire looking like
\[
\xymatrix@C+10pt{
M_1\times\D_0M_1\times_{\D_0^2M_0}
\D_0\D_1M_0
\ar[r]^-{1\times\D_0\psi_1}
\ar[d]_-{1^2\times I_\D}
&M_1\times\D_0M_1
\ar[dd]^-{\gamma_M}
\\M_1\times\D_0M_1\times_{\D_0^2M_0}
\D_1^2M_0
\ar[d]_-{\gamma_M\times_{\mu}\mu}
\\M_1\times\D_1M_0
\ar[r]_-{\psi_1}
&M_1.
}
\]
Filling this diagram requires the following lemma connecting
the presheaf actions on $M_2$ and $M_1$.

\begin{lemma}
The following diagram commutes.
\[
\xymatrix@C+7pt{
M_2\times_{\D_0^2M_0}\D_0\D_1M_0
\ar[r]^-{(T,S)\times1}_-{\cong}
\ar[d]_-{1\times I_\D}
&M_1\times\D_0M_1\times_{\D_0^2M_0}\D_0\D_1M_0
\ar[d]^-{\cong}
\\M_2\times_{\D_0^2M_0}\D_1^2M_0
\ar[d]_-{\psi_2}
&M_1\times\D_0(M_1\times\D_1M_0)
\ar[d]^-{1\times\D_0\psi_1}
\\M_2\ar[r]_-{(T,S)}^-{\cong}
&M_1\times\D_0M_1.}
\]
\end{lemma}

\begin{proof}
If we project onto the second factor, the resulting diagram commutes as a 
consequence of the following one:
\[
\xymatrix{
M_2\times_{\D_0^2M_0}\D_0\D_1M_0
\ar[rr]^-{S\times1}
\ar[d]_-{1\times I_\D}
&&\D_0M_1\times_{\D_0^2M_0}\D_0\D_1M_0
\ar[dl]^-{1\times I_\D}
\ar[dd]^-{\cong}
\\M_2\times_{\D_0^2M_0}\D_1^2M_0
\ar[r]^-{S\times1}
\ar[ddd]_-{\psi_2}
&\D_0M_1\times_{\D_0^2M_0}\D_1^2M_0
\ar[d]^-{\cong}
\\&\D_1(M_1\times\D_1M_0)
\ar[d]^-{\D_1\psi_1}
&\D_0(M_1\times\D_1M_0)
\ar[l]_-{I_\D}
\ar[d]^-{\D_0\psi_1}
\\&\D_1M_1
\ar[d]_-{S_\D}
&\D_0M_1
\ar[l]_-{I_\D}
\ar[dl]^-{=}
\\M_2\ar[r]_-{S}
&\D_0M_1.}
\]
Here the lower right rectangle commutes because of the requirement that
$S:M_2\to\D_0M_1$ be a map of $\D^2(M_0^\delta)$ presheaves, along
with our previous identification of the presheaf action on $\D_0M_1$.  The only
other part of the diagram that isn't immediate is the upper left (somewhat distorted)
square, but that is a consequence of the fact that $I_\D$ splits $T_\D$ in the
following naturality diagram, in which the right square is a pullback:
\[
\xymatrix{
\D_0(M_1\times\D_1M_0)
\ar[r]^-{I_\D}
\ar[d]_-{\D_0p_2}
&\D_1(M_1\times\D_1M_0)
\ar[r]^-{T_\D}
\ar[d]^-{\D_1p_2}
&\D_0(M_1\times\D_1M_0)
\ar[d]_-{\D_0p_2}
\\\D_0\D_1M_0
\ar[r]_-{I_\D}
&\D_1^2M_0
\ar[r]_-{T_\D}
&\D_0\D_1M_0.}
\]

Projection onto the first factor requires us to see that the diagram
\[
\xymatrix{
M_2\times_{\D_0^2M_0}\D_0\D_1M_0
\ar[rr]^-{p_1}
\ar[d]_-{1\times I_\D}
&&M_2
\ar[d]^-{T}
\\M_2\times_{\D_0^2M_0}\D_1^2M_0
\ar[r]_-{\psi_2}
&M_2\ar[r]_-{T}&M_1}
\]
commutes.  However, this follows as a result of the following larger
expansion of the diagram:
\[
\xymatrix{
M_2\times_{\D_0^2M_0}\D_0\D_1M_0
\ar[rr]^-{(T,S)\times1}_-{\cong}
\ar[d]_-{1\times I_\D}
&&M_1\times\D_0M_1\times_{\D_0^2M_0}\D_0\D_1M_0
\ar[dl]_-{1^2\times I_\D}
\ar[d]^-{p_{12}}
\ar@/^4pc/[dd]^-{1\times\e_*}
\\M_2\times_{\D_0^2M_0}\D_1^2M_0
\ar[r]^-{(T,S)\times1}_-{\cong}
\ar[dd]_-{\psi_2}
&M_1\times\D_0M_1\times_{\D_0^2M_0}\D_1^2M_0
\ar[d]_-{1\times\e_*}
&M_1\times\D_0M_1
\ar[d]^-{1\times\D_0T}
\\&M_1\times\D_1M_0
\ar[d]_-{\psi_1}
&M_1\times\D_0M_0
\ar[l]_-{1\times I_\D}
\ar[dl]^-{\cong}
\\M_2\ar[r]_-{T}&M_1.}
\]
The only part of this expansion that hasn't been previously established is
the pentagon incorporated in the right hand column.  However, we can 
expand the diagram defining $\e_*$ to include $I_\D$'s as follows:
\[
\xymatrix@C-20pt{
\D_0M_1\times_{\D_0^2M_0}\D_0\D_1M_0
\ar[rr]^-{1\times I_\D}
\ar@{.>}[dr]^-{\e_*}
\ar[dd]_-{p_2}
&&\D_0M_1\times_{\D_0^2M_0}\D_1^2M_0
\ar[rr]^-{p_1}
\ar@{.>}[dr]^-{\e_*}
\ar[dd]_(.7){p_2}
&&\D_0M_1
\ar[dr]^-{\D_0T}
\ar[dd]^(.7){\D_0S}
\\&\D_0M_0
\ar[rr]^(.3){I_\D}
\ar[ddd]_-{\D_0\e}
&&\D_1M_0
\ar[rr]^(.3){T_\D}
\ar[ddd]_-{\D_1\e}
&&\D_0M_0
\ar[ddd]^-{\D_0\e}
\\\D_0\D_1M_0
\ar[rr]_(.3){I_\D}
\ar[d]_-{\D_0\D_1\e}
&&\D_1^2M_0
\ar[rr]_(.3){T_\D^2}
\ar[d]_-{\D_1^2\e}
&&\D_0^2M_0
\ar[d]^-{\D_0^2\e}
\\\D_0\D_1(*)
\ar[rr]^(.3){I_\D}
\ar[dr]_-{\D_0\e}
&&\D_1^2(*)
\ar[rr]^(.3){T_\D^2}
\ar[dr]_-{\D_1\e}
&&\D_0^2(*)
\ar[dr]_-{\D_0\e}
\\&\D_0(*)
\ar[rr]_-{I_\D}
&&\D_1(*)
\ar[rr]_-{T_\D}
&&\D_0(*).}
\]
In this diagram, the left $\e_*$ does have $\D_0M_0$ as its target, since
both the front squares are pullbacks: the right one since $\e:M_0^\delta\to*$
is a target cover, and the left one since $T_\D\circ I_\D=\id$.  Consequently,
the top of the diagram tells us that $\e_*=\D_0T\circ p_1$, as necessary
in the previous diagram.  This establishes projection onto the first factor
in the diagram of the lemma, concluding its proof.
\end{proof}

We can now proceed with the current fill diagram, above the statement
of the lemma, which follows from its expansion as follows:
\[
\xymatrix{
M_1\times\D_0M_1\times_{\D_0^2M_0}\D_0\D_1M_0
\ar[rrr]^-{1\times\D_0\psi_1}
\ar[dd]_-{1^2\times I_\D}
&&&M_1\times\D_0M_1
\ar[ddd]^-{\gamma_M}
\\&M_2\times_{\D_0^2M_0}\D_0\D_1M_0
\ar[ul]_-{(T,S)\times1}^-{\cong}
\ar[d]^-{1\times I_\D}
\\M_1\times\D_0M_1\times_{\D_0^2M_0}\D_1^2M_0
\ar[d]_-{\gamma_M\times_\mu\mu}
&M_2\times_{\D_0^2M_0}\D_1^2M_0
\ar[l]_-{(T,S)\times1}^-{\cong}
\ar[dl]^-{\gamma_M\times_\mu\mu}
\ar[r]_-{\psi_2}
&M_2
\ar[uur]^-{(T,S)}_-{\cong}
\ar[dr]_-{\gamma_M}
\\M_1\times\D_1M_0
\ar[rrr]_-{\psi_1}
&&&M_1.}
\]
Here the irregular pentagon in the upper right is a consequence of
the previous lemma, and the distorted square at the bottom expresses
the equivariance of $\gamma:M_2\to M_1$ as a presheaf map over
the monad multiplication $\mu:\D^2\to\D$.

Now the final part of the fill is just the naturality diagram
\[
\xymatrix@C+10pt{
\D_0M_1\times_{\D_0^2M_0}\D_0\D_1M_0\times\D_1M_0\times\D_1M_0
\ar[r]^-{\D_0\psi_1\times1^2}
\ar[d]_-{1^2\times\gamma_\D}
&\D_0M_1\times\D_1M_0\times\D_1M_0
\ar[d]^-{1\times\gamma_\D}
\\\D_0M_1\times_{\D_0^2M_0}\D_0\D_1M_0\times\D_1M_0
\ar[r]_-{\D_0\psi_1\times1}
&\D_0M_1\times\D_1M_0.}
\]
We may conclude that the 
diagram with $\Gamma_L$ and the top choice of horizontal arrows
does in fact commute.

The diagram with $\Gamma_L$ and the bottom choices
of horizontal arrows is as with the top choices, but with the two 
horizontal arrows replaced with $1^2\times\gamma_\D\circ((\mu\circ I_\D)\times1)$
on top, and similarly on the bottom without the square on the 1.
In pursuit of this diagram,
which amounts to a glorified
associativity diagram, we need a few lemmas.

\begin{lemma}
We have a natural isomorphism of natural transformations
\[
\xymatrix{
\D_2^2
\ar[rr]^-{(T_\D^2,S_\D^2)}
&&\D_1^2\times_{\D_0^2}\D_1^2.}
\]
\end{lemma}

\begin{proof}
There are at least three diagrams that will demonstrate this; the one we
want is the following one:
\[
\xymatrix{
\D_2^2
\ar[r]^-{T_\D}
\ar[d]_-{\D_2S_\D}
&\D_1\D_2
\ar[r]^-{\D_1T_\D}
\ar[d]_-{\D_1S_\D}
&\D_1^2
\ar[d]^-{\D_1S_\D}
\\\D_2\D_1
\ar[r]^-{T_\D}
\ar[d]_-{S_\D}
&\D_1^2
\ar[r]^-{\D_1T_\D}
\ar[d]_-{S_\D}
&\D_1\D_0
\ar[d]^-{S_\D}
\\\D_1^2
\ar[r]_-{T_\D}
&\D_0\D_1
\ar[r]_-{\D_0T_\D}
&\D_0^2.}
\]
The upper right and lower left squares are pullbacks by definition.
The upper left square is a pullback since $S_\D:\D_2X^\delta\to\D_1X^\delta$
is a target cover for any set $X$, which $\D_2$ preserves, and the lower right square is a 
pullback since $T_\D:\D_1X^\delta\to\D_0X^\delta$ is a source cover,
which $\D_1$ preserves.  The total diagram is therefore a pullback.
\end{proof}

From this, we can now proceed to the following lemma.

\begin{lemma}
The diagram
\[
\xymatrix@C+20pt{
\D_1^2
\ar[r]^-{(\D_1T,S_\D)}_-{\cong}
\ar[ddrr]_-{=}
&\D_1\D_0\times_{\D_0^2}\D_0\D_1
\ar[r]^-{\D_1I_\D\times I_\D}
&\D_1^2\times_{\D_0^2}\D_1^2
\\&&\D_2^2
\ar[u]_-{(T_\D^2,S_\D^2)}^-{\cong}
\ar[d]^-{\gamma_\D^2}
\\&&\D_1^2}
\]
commutes.
\end{lemma}

\begin{proof}
We can fill the diagram as follows.
\[
\xymatrix@C+20pt{
\D_1^2
\ar[r]^-{(\D_1T,S_\D)}
\ar[drr]_-{(I_R)_\D}
\ar[ddrr]_-{=}
&\D_1\D_0\times_{\D_0^2}\D_0\D_1
\ar[r]^-{1\times I_\D}
&\D_1\D_0\times_{\D_0^2}\D_1^2
\ar[r]^-{\D_1I_\D\times1}
&\D_1^2\times_{\D_0^2}\D_1^2
\\&&\D_2\D_1
\ar[u]_-{(T_\D^2,S_\D)}^-{\cong}
\ar[r]^-{\D_2(I_L)_\D}
\ar[dr]^-{=}
\ar[d]^-{\gamma_\D}
&\D_2^2
\ar[u]_-{(T_\D^2,S_\D^2)}^-{\cong}
\ar[d]^-{\D_2\gamma_\D}
\\&&\D_1^2
&\D_2\D_2.
\ar[l]^-{\gamma_\D}}
\]
The upper left triangle commutes since $I_R$ can be expressed as
the composite
\[
\xymatrix{
\D_1
\ar[r]^-{\cong}
&\D_0\times_{\D_0}\D_1
\ar[r]^-{I_\D\times1}
&\D_1\times_{\D_0}\D_1
\ar[r]^-{\cong}
&\D_2,}
\]
and similarly the upper right square commutes since $I_L$ can be expressed as
the composite
\[
\xymatrix{
\D_1
\ar[r]^-{\cong}
&\D_1\times_{\D_0}\D_0
\ar[r]^-{1\times I_\D}
&\D_1\times_{\D_0}\D_1
\ar[r]^-{\cong}
&\D_2.}
\]
The rest of the diagram expresses the unit conditions for $\gamma_\D$.
\end{proof}

The next lemma is the one we will really need for the diagram with $\Gamma_L$
and the bottom choice of horizontal arrows:

\begin{lemma}\label{muandgamma}
The following diagram commutes:
\[
\xymatrix@C+10pt{
\D_1^2
\ar[r]^-{(\D_1T_\D,S_\D)}
\ar[ddr]_-{\mu}
&\D_1\D_0\times_{\D_0^2}\D_0\D_1
\ar[d]^-{(\mu\circ\D_1I_\D)\times_\mu(\mu\circ I_\D)}
\\&\D_1\times_{\D_0}\D_1
\ar[d]^-{\gamma_\D}
\\&\D_1.}
\]
\end{lemma}

\begin{proof}
We can expand and fill in the diagram as follows:
\[
\xymatrix@C+10pt{
\D_1^2
\ar[r]^-{(\D_1T_\D,S_\D)}_-{\cong}
\ar[ddrr]_-{=}
&\D_1\D_0\times_{\D_0^2}\D_0\D_1
\ar[r]^-{\D_1I_\D\times I_\D}
&\D_1^2\times_{\D_0^2}\D_1^2
\ar[r]^-{\mu\times_\mu\mu}
&\D_1\times_{\D_0}\D_1
\\&&\D_2^2
\ar[u]_-{(T_\D^2,S_\D^2)}^-{\cong}
\ar[r]^-{\mu}
\ar[d]^-{\gamma_\D^2}
&\D_2
\ar[u]_-{(T_\D,S_\D)}^-{\cong}
\ar[d]^-{\gamma_\D}
\\&&D_1^2
\ar[r]_-{\mu}
&\D_1.}
\]
The left triangle commutes by the previous lemma, the upper right
square because $\mu$ commutes with $T_\D$ and $S_\D$, 
and the lower right square because $\mu$ commutes with $\gamma_\D$;
all of these are aspects of the fact that $\D$ is a category object
in monads on $\Set$.
\end{proof}

We can now begin filling in the 
diagram with $\Gamma_L$ and the bottom choice of horizontal arrows,
which has its left
column the 
expression for $\Gamma_L$
displayed in figure \ref{Fi:secondrearrangement},
connected to a right column we may need to display after filling, and 
connected across the top using the action $\phi:\D_0\D_1M_0\times\D_1M_0
\to\D_1M_0$ given by the composite 
\[
\xymatrix{
\D_0\D_1M_0\times\D_1M_0
\ar[r]^-{I_\D\times1}
&\D_1^2M_0\times\D_1M_0
\ar[r]^-{\mu\times1}
&\D_1M_0\times\D_1M_0
\ar[r]^-{\gamma_\D}
&\D_1M_0.}
\]
The first part of the fill consists of the following diagram,
from which an initial $\D_0\D_1$ has been suppressed, since all its morphisms
are the identity:
\[
\xymatrix@C-45pt{
\D_1M_0\times\D_0M_1\times_{\D_0^2M_0}\D_0\D_1M_0\times\D_1M_0
\ar[r]^-{1^2\times_\mu\phi}
&\D_1M_0\times\D_0M_1\times\D_1M_0
\\\D_1M_1\times_{\D_0^2M_0}\D_0\D_1M_0\times\D_1M_0
\ar[u]^-{(\D_1T,S_\D)\times1^2}_-{\cong}
\ar[r]^-{1\times_\mu\phi}
\ar[d]_-{(T_\D,\D_1S)\times1}^-{\cong}
&\D_1M_1\times\D_1M_0
\ar[u]_-{(D_1T,S_\D)\times1}^-{\cong}
\ar[dd]^-{(T_\D,\D_1S)\times1}_-{\cong}
\\\D_0M_1\times_{\D_0^2M_0}\D_1\D_0M_0\times_{\D_0^2M_0}\D_0\D_1M_0\times\D_1M_0
\ar[dr]^-{1^2\times_\mu\phi}
\\&\D_0M_1\times_{\D_0^2M_0}\D_1\D_0M_0\times\D_1M_0\times\D_1M_0.}
\]
For the rest of the fill, we can delete all terms involving $M_1$, since they just get
multiplied together and have no effect on the rest of the diagram.  Since the rest of
the diagram involves only $M_0$, we can suppress it, and now wish to verify the 
following diagram:
\[
\xymatrix{
\D_1\D_0\times_{\D_0^2}\D_0\D_1\times\D_1
\ar[dr]^-{1\times_\mu(\mu\circ I_\D)\times1}
\\&\D_1\D_0\times\D_1\times\D_1
\ar[r]^-{1\times\gamma_\D}
\ar[d]_-{(\mu\circ\D_1I_\D)\times1^2}
&\D_1\D_0\times\D_1
\ar[d]^-{(\mu\circ\D_1I_\D)\times1}
\\&\D_1\times\D_1\times\D_1
\ar[r]^-{1\times\gamma_\D}
\ar[d]^-{\gamma_\D\times1}
&\D_1\times\D_1
\ar[ddd]^-{\gamma_\D}
\\\D_1^2\times\D_1
\ar[uuu]^-{(\D_1T_\D,S_\D)\times1}_-{\cong}
\ar[r]^-{\mu\times1}
\ar[d]_-{(T_\D,\D_1S_\D)\times1}^-{\cong} 
&\D_1\times\D_1
\ar[ddr]^-{\gamma_\D}
\\\D_0\D_1\times_{\D_0^2}\D_1\D_0\times\D_1
\ar[dr]^(.44){\phantom{Mjj}(\mu\circ I_\D)\times_\mu(\mu\circ\D_1I_\D)\times1}
\\\D_0^2\D_1\times_{\D_0^2}\D_1\D_0\times\D_1
\ar[u]^-{\mu\times1^2}
\ar[d]_-{\D_0(\mu\circ I_\D)\times_\mu(\mu\circ\D_1I_\D)\times1}
&\D_1\times\D_1\times\D_1
\ar[uu]_-{\gamma_\D\times1}
\ar[dd]_-{1\times\gamma_\D}
&\D_1
\\\D_0\D_1\times\D_1\times\D_1
\ar[ur]^-{(\mu\circ I_\D)\times1^2}
\ar[d]_-{1\times\gamma_\D}
\\\D_0\D_1\times\D_1
\ar[r]_-{(\mu\circ I_\D)\times1}
&\D_1\times\D_1
\ar[uur]_-{\gamma_\D}}
\]
The upper left part commutes because of lemma \ref{muandgamma}, 
the part below it by its analogue switching $T_\D$ and $S_\D$,
and
the only other part of the diagram requiring verification is the third part down
on the left.  That diagram, however, easily commutes on projection to the 
second and third factors of $\D_1$ in the target, leaving us only with
projection to the first factor.  That in turn falls to commutativity of the
following diagram:
\[
\xymatrix{
\D_0^2\D_1
\ar[rr]^-{\mu}
\ar[d]_-{\D_0I_\D}
&&\D_0\D_1
\ar[d]^-{I_\D}
\\\D_0\D_1^2
\ar[r]^-{I_\D}
\ar[d]_-{\D_0\mu}
&\D_1^3
\ar[r]^-{\mu}
\ar[d]_-{\D_0\mu}
&\D_1^2
\ar[d]^-{\mu}
\\\D_0\D_1
\ar[r]_-{I_\D}
&\D_1^2
\ar[r]_-{\mu}
&\D_1.}
\]
We have verified all four diagrams we needed in order to see that composition
in $\Lhat M$ descends to composition in $LM$.  It follows that $LM$ inherits
a category structure from $\Lhat M$.

\section{The actual left adjoint: $\D$-algebra structure}

We next need to show that $LM$ inherits a $\D$-algebra structure from $\Lhat M$.
For this purpose, it is convenient to recognize that the composite
\[
\xymatrix{
\D_0M_1\times_{\D_0M_0}\D_1M_0
\ar[d]^-{\cong}
\\\D_0M_1\times_{\D_0^2M_0}\D_0^2M_0\times_{\D_0M_0}\D_1M_0
\ar[d]^-{1\times\D_0I_\D}
\\\D_0M_1\times_{\D_0^2M_0}\D_0\D_1M_0\times_{\D_0M_0}\D_1M_0}
\]
provides a common right inverse for the two maps coequalized by the
map from $\Lhat M_1$ to $LM_1$.  Consequently, the coequalizer is 
reflexive, and so is preserved by products.  This is most of the proof of
the following lemma.

\begin{lemma}
Let $\D$ be the monad associated to an operad.  Then $\D$ preserves
reflexive coequalizers.
\end{lemma}

\begin{proof}
Let $
\xymatrix@1{
A\ar@<1ex>[r]^-{f}\ar@<-1ex>[r]_-{g}&B\ar[r]^-{h}\ar@{.>}[l]&C}
$
be a reflexive coequalizer.  Then since products preserve reflexive
coequalizers, 
\[
\xymatrix{
A^n\ar@<1ex>[r]^-{f^n}\ar@<-1ex>[r]_-{g^n}&B^n\ar[r]^-{h^n}\ar@{.>}[l]&C^n}
\]
is also a reflexive coequalizer.  The identity coequalizer on $\DD_n$ is
also reflexive, so 
\[
\xymatrix{
\DD_n\times A^n\ar@<1ex>[r]^-{1\times f^n}\ar@<-1ex>[r]_-{1\times g^n}
&\DD_n\times B^n\ar[r]^-{1\times h^n}\ar@{.>}[l]&\DD_n\times C^n}
\]
is also a reflexive coequalizer.  Now passage to orbits, being a colimit, 
commutes with coequalizers, and the section map is also preserved.
\end{proof}

We wish to show that the $\D$-algebra structure on $\Lhat M$ descends
to one on $LM$, and since the objects of $\Lhat M$ and of $LM$ coincide
(both are $\D_0M_0$),
and we have shown that the quotient map preserves the category structure,
it suffices to show that the $\D_1$-algebra structure
on $\Lhat M_1$ descends to $LM_1$.  Since $LM_1$ is defined as the coequalizer
of a reflexive fork, we can use the previous lemma to express $\D_1LM_1$
as the coequalizer of $\D_1$ applied to that same fork.  Now we want the 
action map
\[
\xymatrix{
\D_1\Lhat M_1\ar[r]^-{\xi_1}&\Lhat M_1}
\]
to descend to one on $LM_1$, that is, we want to define an action map
on $LM_1$ so that the diagram
\[
\xymatrix{
\D_1\Lhat M_1\ar[r]^-{\xi_1}\ar[d]&\Lhat M_1\ar[d]
\\\D_1LM_1\ar[r]_-{\xi_1}&LM_1}
\]
commutes, where the vertical arrows are given by the descent map from
$\Lhat M_1$ to $LM_1$.  In order to do so, we recall that the action map
on $\Lhat M_1$ is given by the map
\[
\xymatrix{
\D_1(\D_0M_1\times\D_1M_0)\cong\D_1\D_0M_1\times_{\D_1\D_0M_0}\D_1^2M_0
\ar[rr]^-{(\mu\circ I_\D)\times_{(\mu\circ I_\D)}\mu}
&&\D_0M_1\times\D_1M_0.}
\]
What we will use is an action map $\hat\xi$ on the source of the coequalizer, and show
that the resulting two diagrams
\[
\xymatrix{
\D_1(\D_0M_1\times_{\D_0^2M_0}\D_0\D_1M_0\times\D_1M_0)
\ar[r]^-{\hat\xi}
\ar@<.5ex>[d]_-{\D_1(\D_0\psi_1\times1)\phantom{j}}
\ar@<-.5ex>[d]^-{\phantom{j}\D_1(1\times\phi)}
&\D_0M_1\times_{\D_0^2M_0}\D_0\D_1M_0\times\D_1M_0
\ar@<.5ex>[d]_-{\D_0\psi_1\times1\phantom{j}}
\ar@<-.5ex>[d]^-{\phantom{j}1\times\phi}
\\\D_1(\D_0M_1\times\D_1M_0)
\ar[r]_-{\xi_1}
&\D_0M_1\times\D_1M_0,}
\]
one each for choice of left vertical arrows or right vertical arrows, commute.  
It will then follow that $\xi_1$ induces an action map on $LM_1$, as desired.

For $\hat\xi$ we choose the map
\[
\xymatrix{
\D_1\D_0M_1\times_{\D_1\D_0^2M_0}\D_1\D_0\D_1M_0\times_{\D_1\D_0M_0}\D_1^2M_0
\ar[d]_-{(\mu\circ I_\D)\times_{\mu\circ I_\D}(\mu\circ I_\D)\times_{\mu\circ I_\D}\mu}
\\\D_0M_1\times_{\D_0^2M_0}\D_0\D_1M_0\times\D_1M_0.}
\]
In order to see that this map is well-defined, we
call the first two factors in the source $P_{12}$ and the second two $P_{23}$,
and display the commuting diagrams
\[
\xymatrix{
P_{12}\ar[r]^-{p_1}
\ar[d]_-{p_2}
&\D_1\D_0M_1
\ar[r]^-{T_\D}
\ar[d]_-{\D_1\D_0S_\D}
&\D_0^2M_1
\ar[r]^-{\mu}
\ar[d]_-{\D_0^2S_\D}
&\D_0M_1
\ar[d]^-{\D_0S_\D}
\\\D_1\D_0\D_1M_0
\ar[r]^-{\D_1\D_0T_\D}
\ar[dr]_-{T_\D}
&\D_1\D_0^2M_0
\ar[r]^-{T_\D}
&\D_0^3M_0
\ar[r]^-{\mu}
&\D_0^2M_0
\\&\D_0^2\D_1M_0
\ar[ur]^-{\D_0^2T_\D}
\ar[r]_-{\mu}
&\D_0\D_1M_0
\ar[ur]_-{\D_0T_\D}}
\]
and
\[
\xymatrix{
P_{23}
\ar[r]^-{p_1}
\ar[dd]_-{p_2}
&\D_1\D_0\D_1M_0
\ar[r]^-{T_\D}
\ar[d]_-{\D_1\D_0S_\D}
&\D_0^2\D_1M_0
\ar[r]^-{\mu}
\ar[d]_-{\D_0^2S_\D}
&\D_0\D_1M_0
\ar[d]^-{\D_0S_\D}
\\&\D_1\D_0^2M_0
\ar[r]^-{T_\D}
\ar[d]_-{\D_1\mu}
&\D_0^3M_0
\ar[r]^-{\mu}
\ar[d]_-{\D_0\mu}
&\D_0^2M_0
\ar[d]^-{\mu}
\\\D_1^2M_0
\ar[r]^-{\D_1T_\D}
\ar[dr]_-{\mu}
&\D_1\D_0M_0
\ar[r]^-{T_\D}
&\D_0^2M_0
\ar[r]^-{\mu}
&\D_0M_0.
\\&\D_1M_0
\ar[urr]_-{T_\D}}
\]
Now the diagram we want with the left choice of vertical arrows
doesn't involve the terms $\xymatrix@1{\D_1^2M_0\ar[r]^-{\mu}&\D_1M_0}$,
but the rest of the diagram now reduces to the naturality diagram
\[
\xymatrix{
\D_1\D_0(M_1\times_{\D_0M_0}\D_1M_0)
\ar[r]^-{\mu\circ T_\D}
\ar[d]_-{\D_1\D_0\psi_1}
&\D_0(M_1\times_{\D_0M_0}\D_1M_0)
\ar[d]^-{\D_0\psi_1}
\\\D_1\D_0M_1
\ar[r]_-{\mu\circ T_\D}
&\D_0M_1.}
\]
The diagram we want with the right choice of vertical arrows doesn't involve the terms
$\xymatrix@1{\D_1\D_0M_1\ar[r]^-{\mu\circ T_\D}&\D_0M_1}$, and so
the diagram reduces to the following one, which is still to be verified:
\[
\xymatrix@C+40pt{
\D_1\D_0\D_1M_0\times_{\D_1\D_0M_0}\D_1^2M_0
\ar[r]^-{(\mu\circ T_\D)\times_{\mu\circ T_\D}\mu}
\ar[d]_-{\D_1\phi}
&\D_0\D_1M_0\times_{\D_0M_0}\D_1M_0
\ar[d]^-{\phi}
\\\D_1^2M_0
\ar[r]_-{\mu}&\D_1M_0.}
\]
However, this can be filled in as follows, expanding the definition
of $\phi=\gamma_\D((\mu\circ I_\D)\times1)$, and suppressing the
$M_0$, since it plays no role:
\[
\xymatrix@C+7pt{
\D_1\D_0\D_1\times_{\D_1\D_0}\D_1^2
\ar[r]^-{T_\D\times_{T_\D}1}
\ar[d]_-{\D_1I_\D\times1}
&\D_0^2\D_1\times_{\D_0^2}\D_1^2
\ar[r]^-{\mu\times_\mu\mu}
\ar[d]_-{I_\D^2\times1}
&\D_0\D_1\times_{\D_0}\D_1
\ar[d]^-{I_\D\times1}
\\\D_1^3\times_{\D_1\D_0}\D_1^2
\ar[r]^-{(I_\D\circ T_\D)\times_{T_\D}1}
\ar[d]_-{\D_1\mu\times1}
&\D_1^3\times_{\D_0^2}\D_1^2
\ar[r]^-{\mu\times_\mu\mu}
\ar[d]_-{\D_1\mu\times1}
&\D_1^2\times_{\D_0}\D_1
\ar[d]^-{\mu\times1}
\\\D_1^2\times_{\D_1\D_0}\D_1^2
\ar[r]^-{(I_\D\circ T_\D)\times_{T_\D}1}
\ar[d]_-{\D_1\gamma_\D}
&\D_1^2\times_{\D_0^2}\D_1^2
\ar[r]^-{\mu\times_\mu\mu}
\ar[dl]^-{\gamma_\D^2}
&\D_1\times_{\D_0}\D_1
\ar[d]^-{\gamma_\D}
\\\D_1^2
\ar[rr]_-{\mu}
&&\D_1.}
\]
It follows that the $\D$-algebra structure on $\Lhat M$ descends to one
on $LM$.

\section{The actual left adjoint: the counit}

We need to show that the counit $\e:\Lhat U\C\to\C$ factors through
$LU\C$.  Since the quotient map $\Lhat U\C\to LU\C$ is the identity
on objects, this amounts to checking that the map on morphisms factors.
Since the quotient map is a coequalizer on morphisms, this happens if and only if the
counit $\e:\Lhat U\C_1\to\C_1$ coequalizes the maps defining $LU\C$.
We begin by recalling that the action $\psi_1:U\C_1\times_{\D_0\C_0}\D_1\C_0$
can be expressed as the composite
\begin{gather*}
\xymatrix{
U\C_1\times_{\D_0\C_0}\D_1\C_0
\ar[r]^-{\kappa_1\times_{\xi_0}1}_-{\cong}
&\C_1\times_{\C_0}\D_1\C_0
\ar[r]^-{1\times\D_1I}
&\C_1\times_{\C_0}\D_1\C_1}
\\\xymatrix{
\relax
\ar[rr]^-{1\times(\xi_1,S_\D S)}
&&\C_1\times_{\C_0}\C_1\times_{C_0}\D_0\C_0
\ar[r]^-{\gamma_\C\times1}
&\C_1\times_{\C_0}\D_0\C_0\cong U\C_1.}
\end{gather*}
Now the counit $\e_1:\Lhat U\C_1\to\C_1$ expands as the composite
\begin{gather*}
\xymatrix@C+10pt{
\D_0U\C_1\times_{D_0\C_0}\D_1\C_0
\ar[r]^-{\D_0\kappa_1\times_{\xi_0}1}
&\D_0\C_1\times_{\C_0}\D_1\C_0}
\\\xymatrix{
\relax
\ar[r]^-{I_\D\times\D_1I}
&\D_1\C_1\times_{\C_0}\D_1\C_1
\ar[r]^-{\xi_1\times\xi_1}
&\C_1\times_{\C_0}\C_1
\ar[r]^-{\gamma_\C}
&\C_1.}
\end{gather*}
Since the first step of the counit is $\D_0\kappa_1$ on the factor of 
$\D_0U\C_1\cong\D_0\C_1\times_{\D_0\C_0}\D_0^2\C_0$,
and $\D_0\kappa_1$ just discards the $\D_0^2\C_0$, we can rewrite
the composite $\e_1\circ(\D_0\psi_1\times1)$ as follows, 
in which it should be noted carefully that the structure maps from
$\D_0\D_1\C_0$ to $\D_0\C_0$ are not the same: to the left it's 
$\D_0\xi_0(T_\D)$, and to the right it's $\mu\circ\D_0S_\D$:
\begin{gather*}
\xymatrix@C+20pt{
\D_0U\C_1\times_{\D_0^2\C_0}\D_0\D_1\C_0\times_{\D_0\C_0}\D_1\C_0
\ar[r]^-{\D_0\kappa_1\times_{\D_0\xi_0}1^2}
&\D_0\C_1\times_{\D_0\C_0}\D_0\D_1\C_0\times_{\D_0\C_0}\D_1\C_0}
\\\xymatrix@C+20pt{
\relax
\ar[r]^-{1\times\D_0\D_1I\times1}
&\D_0\C_1\times_{\D_0\C_0}\D_0\D_1\C_1\times_{\D_0\C_0}\D_1\C_0
\ar[r]^-{1\times\D_0\xi_1\times_{\xi_0}1}
&\D_0\C_1\times_{\D_0\C_0}\D_0\C_1\times_{\C_0}\D_1\C_0}
\\\xymatrix@C+10pt{
\relax
\ar[r]^-{\D_0\gamma_\C\times1}
&\D_0\C_1\times_{\C_0}\D_1\C_0
\ar[r]^-{I_\D\times\D_1I}
&\D_1\C_1\times_{\C_0}\D_1\C_1
\ar[r]^-{\xi_1\times\xi_1}
&\C_1\times_{\C_0}\C_1
\ar[r]^-{\gamma_\C}
&\C_1.}
\end{gather*}
We can now rewrite this using the following diagram, which starts
at the second term of the above composite:
\[
\xymatrix@C+35pt{
\D_0\C_1\times\D_0\D_1\C_0\times\D_1\C_0
\ar[r]^-{I_\D\times I_\D\D_1I\times\D_1I}
\ar[d]_-{1\times\D_0\D_1I\times1}
&\D_1\C_1\times\D_1^2\C_1\times\D_1\C_1
\ar[dd]^-{1\times\D_1\xi_1\times_{\xi_0}1}
\\\D_0\C_1\times\D_0\D_1\C_1\times\D_1\C_0
\ar[ur]_-{I_\D\times I_\D\times\D_1I}
\ar[d]_-{1\times\D_0\xi_1\times_{\xi_0}1}
\\\D_0\C_1\times\D_0\C_1\times_{\C_0}\D_1\C_0
\ar[r]^-{I_\D\times I_\D\times\D_1I}
\ar[d]_-{\D_0\gamma_\C\times1}
&\D_1\C_1\times\D_1\C_1\times_{\C_0}\D_1\C_1
\ar[d]^-{\gamma_{\D\C}\times1}
\\\D_0\C_1\times_{\C_0}\D_1\C_0
\ar[r]_-{I_\D\times\D_1I}
&\D_1\C_1.}
\]
The bottom rectangle simply records the fact that the
composition of identities is the identity in the category structure
on $\D$.  We can now express the composite $\e_1\circ(\D_0\psi_1\times1)$ 
as follows:
\begin{gather*}
\xymatrix@C+20pt{
\D_0U\C_1\times_{\D_0^2\C_0}\D_0\D_1\C_0\times_{\D_0\C_0}\D_1\C_0
\ar[r]^-{\D_0\kappa_1\times_{\D_0\xi_0}1^2}
&\D_0\C_1\times_{\D_0\C_0}\D_0\D_1\C_0\times_{\D_0\C_0}\D_1\C_0}
\\\xymatrix@C+30pt{
\relax
\ar[r]^-{I_\D\times I_\D\D_1I\times\D_1I}
&\D_1\C_1\times_{\D_0\C_0}\D_1^2\C_1\times_{\D_0\C_0}\D_1\C_1
\ar[r]^-{1\times\D_1\xi_1\times_{\xi_0}1}
&\D_1\C_1\times_{\D_0\C_0}\D_1\C_1\times_{\C_0}\D_1\C_1}
\\\xymatrix@C+10pt{
\relax
\ar[r]^-{\gamma_{\D\C}\times1}
&\D_1\C_1\times_{\C_0}\D_1\C_1
\ar[r]^-{\xi_1\times\xi_1}
&\C_1\times_{\C_0}\C_1
\ar[r]^-{\gamma_\C}
&\C_1.}
\end{gather*}
Next, we appeal to the diagram
\[
\xymatrix@C+7pt{
\D_1\C_1\times\D_1^2\C_1\times\D_1\C_1
\ar[r]^-{1\times_{\xi_0}\mu\times1}
\ar[d]_-{1\times\D_1\xi_1\times_{\xi_0}1}
&\D_1\C_1\times_{\C_0}\D_1\C_1\times\D_1\C_1
\ar[r]^-{1\times\gamma_{\D\C}}
\ar[d]^-{\xi_1\times\xi_1\times_{\xi_0}\xi_1}
&\D_1\C_1\times_{\C_0}\D_1\C_1
\ar[d]^-{\xi_1\times\xi_1}
\\\D_1\C_1\times\D_1\C_1\times_{\C_0}\D_1\C_1
\ar[r]^-{\xi_1\times_{\xi_0}\xi_1\times\xi_1}
\ar[d]_-{\gamma_{\D\C}\times1}
&\C_1\times_{\C_0}\C_1\times_{\C_0}\C_1
\ar[r]^-{1\times\gamma_\C}
\ar[d]^-{\gamma_\C\times1}
&\C_1\times_{\C_0}\C_1
\ar[d]^-{\gamma_\C}
\\\D_1\C_1\times_{\C_0}\D_1\C_1
\ar[r]_-{\xi_1\times\xi_1}
&\C_1\times_{\C_0}\C_1
\ar[r]_-{\gamma_\C}
&\C_1,}
\]
in which the upper left and lower right rectangles are associativity
diagrams, and the upper right and lower left follow from the requirement
that $\xi:\D\C\to\C$ be a functor; in particular, the diagram
\[
\xymatrix{
(\D\C)_2
\ar[r]^-{\xi_2}
\ar[d]_-{\gamma_{\D\C}}
&\C_2
\ar[d]^-{\gamma_\C}
\\(\D\C)_1
\ar[r]_-{\xi_1}
&\C_1}
\]
must commute.  Note that the action map $\xi_2$ can be rewritten
as $\xi_1\times_{\xi_0}\xi_1$; the rectangles in the previous diagram
now follow.  We can now rewrite our composite:  $\e_1\circ(\D_0\psi_1\times1)$
coincides with
\begin{gather*}
\xymatrix@C+20pt{
\D_0U\C_1\times_{\D_0^2\C_0}\D_0\D_1\C_0\times_{\D_0\C_0}\D_1\C_0
\ar[r]^-{\D_0\kappa_1\times_{\D_0\xi_0}1^2}
&\D_0\C_1\times_{\D_0\C_0}\D_0\D_1\C_0\times_{\D_0\C_0}\D_1\C_0}
\\\xymatrix@C+30pt{
\relax
\ar[r]^-{I_\D\times I_\D\D_1I\times\D_1I}
&\D_1\C_1\times_{\D_0\C_0}\D_1^2\C_1\times_{\D_0\C_0}\D_1\C_1
\ar[r]^-{1\times_{\xi_0}\mu\times1}
&\D_1\C_1\times_{\C_0}\D_1\C_1\times_{\D_0\C_0}\D_1\C_1}
\\\xymatrix@C+10pt{
\relax
\ar[r]^-{1\times\gamma_{\D\C}}
&\D_1\C_1\times_{\C_0}\D_1\C_1
\ar[r]^-{\xi_1\times\xi_1}
&\C_1\times_{\C_0}\C_1
\ar[r]^-{\gamma_\C}
&\C_1.}
\end{gather*}
We now appeal to the diagram
\[
\xymatrix@C+25pt{
\D_0\C_1\times\D_0\D_1\C_0\times\D_1\C_0
\ar[r]^-{I_\D\times I_\D\D_1I\times\D_1I}
\ar[d]_-{1\times I_\D\times1}
&\D_1\C_1\times\D_1^2\C_1\times\C_1\C_1
\ar[dd]^-{1\times_{\xi_0}\mu\times1}
\\\D_0\C_1\times\D_1^2\C_0\times\D_1\C_0
\ar[ur]_-{I_\D\times\D_1^2I\times\D_1I}
\ar[d]_-{1\times_{\xi_0}\mu\times1}
\\\D_0\C_1\times_{\C_0}\D_1\C_0\times\D_1\C_0
\ar[r]^-{I_\D\times\D_1I\times\D_1I}
\ar[d]_-{1\times\gamma_\D}
&\D_1\C_1\times_{\C_0}\D_1\C_1\times\D_1\C_1
\ar[d]^-{1\times\gamma_{\D\C}}
\\\D_0\C_1\times_{\C_0}\D_1\C_0
\ar[r]_-{I_\D\times\D_1I}
&\D_1\C_1\times_{\C_0}\D_1\C_1,}
\]
in which the lower rectangle commutes since composition
of identities in $\C$ give identities, and we can now rewrite
 $\e_1\circ(\D_0\psi_1\times1)$ as follows:
 \begin{gather*}
\xymatrix@C+20pt{
\D_0U\C_1\times_{\D_0^2\C_0}\D_0\D_1\C_0\times_{\D_0\C_0}\D_1\C_0
\ar[r]^-{\D_0\kappa_1\times_{\D_0\xi_0}1^2}
&\D_0\C_1\times_{\D_0\C_0}\D_0\D_1\C_0\times_{\D_0\C_0}\D_1\C_0}
\\\xymatrix@C+30pt{
\relax
\ar[r]^-{1\times I_\D\times1}
&\D_0\C_1\times_{\D_0\C_0}\D_1^2\C_0\times_{\D_0\C_0}\D_1\C_0
\ar[r]^-{1\times_{\xi_0}\mu\times1}
&\D_0\C_1\times_{\C_0}\D_1\C_0\times_{\D_0\C_0}\D_1\C_0}
\\\xymatrix@C+10pt{
\relax
\ar[r]^-{1\times\gamma_{\D}}
&\D_0\C_1\times_{\C_0}\D_1\C_0
\ar[r]^-{I_\D\times\D_1I}
&\D_1\C_1\times_{\C_0}\D_1\C_1
\ar[r]^-{\xi_1\times\xi_1}
&\C_1\times_{\C_0}\C_1
\ar[r]^-{\gamma_\C}
&\C_1.}
\end{gather*}
However, we can now appeal to the diagram
\[
\xymatrix@C+20pt{
\D_0U\C_1\times_{\D_0^2\C_0}\D_0\D_1\C_0\times\D_1\C_0
\ar[r]^-{\D_0\kappa_1\times_{\D_0\xi_0}1^2}
\ar[d]_-{1\times I_\D\times1}
&\D_0\C_1\times\D_0\D_1\C_0\times\D_1\C_0
\ar[d]^-{1\times I_\D\times1}
\\\D_0U\C_1\times_{\D_0^2\C_0}\D_1^2\C_0\times\D_1\C_0
\ar[r]^-{\D_0\kappa_1\times_{\D_0\xi_0}1^2}
\ar[d]_-{1\times_{\mu}\mu\times1}
&\D_0\C_1\times\D_1^2\C_0\times\D_1\C_0
\ar[d]^-{1\times_{\xi_0}\mu\times1}
\\\D_0U\C_1\times\D_1\C_0\times\D_1\C_0
\ar[r]^-{\D_0\kappa_1\times_{\xi_0}1^2}
\ar[d]_-{1\times\gamma_\D}
&\D_0\C_1\times_{\C_0}\D_1\C_0\times\D_1\C_0
\ar[d]^-{1\times\gamma_\D}
\\\D_0U\C_1\times\D_1\C_0
\ar[r]_-{\D_0\kappa_1\times_{\xi_0}1}
&\D_0\C_1\times_{\C_0}\D_1\C_0}
\]
to conclude that the composite can be rewritten as
 \begin{gather*}
\xymatrix@C+20pt{
\D_0U\C_1\times_{\D_0^2\C_0}\D_0\D_1\C_0\times_{\D_0\C_0}\D_1\C_0
\ar[r]^-{1\times I_\D\times1}
&\D_0U\C_1\times_{\D_0^2\C_0}\D_1^2\C_0\times_{\D_0\C_0}\D_1\C_0}
\\\xymatrix@C+30pt{
\relax
\ar[r]^-{1\times_\mu\mu\times1}
&\D_0U\C_1\times_{\D_0\C_0}\D_1C_0\times_{\D_0\C_0}\D_1\C_0
\ar[r]^-{1\times\gamma_\D}
&\D_0U\C_1\times_{\D_0\C_0}\D_1\C_0}
\\\xymatrix@C+10pt{
\relax
\ar[r]^-{\D_0\kappa_1\times_{\xi_0}1}
&\D_0\C_1\times_{\C_0}\D_1\C_0
\ar[r]^-{I_\D\times\D_1I}
&\D_1\C_1\times_{\C_0}\D_1\C_1
\ar[r]^-{\xi_1\times\xi_1}
&\C_1\times_{\C_0}\C_1
\ar[r]^-{\gamma_\C}
&\C_1,}
\end{gather*}
which coincides with $\e_1\circ(1\times\phi)$.  We may conclude that
the counit descends to $LU\C\to\C$.

\section{The actual left adjoint: the unit}

Our last piece of business is to verify that the unit map preserves
the presheaf structure on morphism sets; this was not done for the
unit to $\Lhat M$, since that is the one piece of the adjunction framework
that fails, and therefore $\Lhat$ is \emph{not} the actual left adjoint.
If we write $\pi:\Lhat M_1\to LM_1$ for the map on morphisms given
by the coequalizer construction, the unit map is given by the 
composite
\[
\xymatrix{
M_1\ar[r]^-{\eta_1}&\Lhat M_1\ar[r]^-{\pi}&LM_1.}
\]
Given a map $f:M\to N$ of $\D$-multicategories, the condition for
preservation of presheaf structure can be expressed by the requirement
that the following diagram commutes:
\[
\xymatrix@C+20pt{
M_1\times_{\D_0M_0}\D_1M_0
\ar[r]^-{f_1\times_{\D_0f_0}\D_1f_0}
\ar[d]_-{\psi_M}
&N_1\times_{\D_0N_0}\D_1N_0
\ar[d]^-{\psi_N}
\\M_1
\ar[r]_-{f_1}
&N_1.}
\]
In our case, what we want is for the perimeter of the diagram
\[
\xymatrix@C+20pt{
M_1\times_{\D_0M_0}\D_1M_0
\ar[r]^-{\eta_1\times_{\D_0\eta_0}\D_1\eta_0}
\ar[d]_-{\psi_1}
&U\Lhat M_1\times_{\D_0^2M_0}\D_1\D_0M_0
\ar[r]^-{U\pi\times1}\
\ar[d]_-{\psi_{U\Lhat M}}
&ULM_1\times_{\D_0^2M_0}\D_1\D_0M_0
\ar[d]^-{\psi_{ULM}}
\\M_1\ar[r]_-{\eta_1}
&U\Lhat M_1\ar[r]_-{U\pi}
&ULM_1}
\]
to commute.  However, although the right square does commute since
we've shown that the $\D$-algebra structure on $\Lhat M$ descends to one
on $LM$, the left square emphatically does \emph{not} commute.

Our strategy is first to express $U\pi$ as a coequalizer, and then show that
the two ways around the left square factor through the two maps that are
coequalized by $U\pi$.  We need the following elementary lemma.

\begin{lemma}
Suppose 
\[
\xymatrix{
A\ar@<.5ex>[r]\ar@<-.5ex>[r]&B\ar[r]&C}
\]
is a coequalizer diagram in $\Set/Z$ for some set $Z$, and let $X$ be a set 
over $Z$.  Then the induced diagram
\[
\xymatrix{
A\times_ZX\ar@<.5ex>[r]\ar@<-.5ex>[r]&B\times_ZX\ar[r]&C\times_ZX}
\]
is also a coequalizer diagram.
\end{lemma}

\begin{proof}
The coequalizer diagram can be decomposed as a coproduct of coequalizer
diagrams of
fibers over the elements of $Z$,
\[
\xymatrix{
\coprod_{z\in Z}A_z\ar@<.5ex>[r]\ar@<-.5ex>[r]
&\coprod_{z\in Z}B_z\ar[r]
&\coprod_{z\in Zz}C_z.}
\]
The conclusion now follows easily.
\end{proof}

In our application, we have $LM_1$ displayed in the coequalizer diagram
\[
\D_0M_1\times_{\D_0^2M_0}\D_0\D_1M_0\times_{\D_0M_0}\D_1M_0
\rightrightarrows\D_0M_1\times_{\D_0M_0}\D_1M_0\to LM_1,
\]
and using the structure maps of source type throughout, all maps are 
over $\D_0M_0$.  Since in general for a $\D$-algebra $\C$ we have
\[
U\C_1\cong\C_1\times_{\C_0}\D_0\C_0,
\]
and in our case $\C_0=LM_0=\D_0M_0$, we get
\[
ULM_1\cong LM_1\times_{\D_0M_0}\D_0^2M_0.
\]
Now the lemma tells us that $ULM_1$ arises as the coequalizer in
the following diagram, where in the interest of space we have adopted our usual convention
of suppressing $\D_0M_0$'s in the subscripts on pullbacks:
\[
\D_0M_1\times_{\D_0^2M_0}\D_0\D_1M_0\times\D_1M_0\times\D_0^2M_0
\rightrightarrows\D_0M_1\times\D_1M_0\times\D_0^2M_0\to ULM_1.
\]
Note that the middle term is $U\Lhat M_1$, and therefore the coequalizer
map is $U\pi$.  The two maps for which $U\pi$ is the coequalizer can be
written as $\D_0\psi_1\times1^2$ and $1\times\phi\times1$, since they also
are induced from the maps defining $LM_1$ as a coequalizer.  

We next introduce a map 
\[
\xymatrix{
M_1\times_{\D_0M_0}\D_1M_0\ar[r]^-{\etahat}
&\D_0M_1\times_{\D_0^2M_0}\D_0\D_1M_0\times_{\D_0M_0}\D_1M_0\times_{\D_0M_0}\D_0^2M_0}
\]
given as $\eta\times_\eta(\eta,I_\D\circ S_\D,\D_0\eta\circ S_\D)$.  We claim that
the diagrams
\[
\xymatrix@C+10pt{
M_1\times_{\D_0M_0}\D_1M_0
\ar[r]^-{\psi_1}
\ar[d]_-{\etahat}
&M_1
\ar[d]^-{\eta_1}
\\\D_0M_1\times_{\D_0^2M_0}\D_0\D_1M_0\times\D_1M_0\times\D_0^2M_0
\ar[r]_-{\D_0\psi_1\times1^2}
&\D_0M_1\times\D_1M_0\times\D_0^2M_0}
\]
and
\[
\xymatrix@C+10pt{
M_1\times_{\D_0M_0}\D_1M_0
\ar[r]^-{\eta\times_{\D_0\eta}\D_1\eta}
\ar[d]_-{\etahat}
&U\Lhat M_1\times_{\D_0^2M_0}\D_1\D_0M_0
\ar[d]^-{\psi_{U\Lhat M}}
\\\D_0M_1\times_{\D_0^2M_0}\D_0\D_1M_0\times\D_1M_0\times\D_0^2M_0
\ar[r]_-{1\times\phi\times1}
&\D_0M_1\times\D_1M_0\times\D_0^2M_0}
\]
both commute, that is, that both ways around the left square
above that does \emph{not} commute instead factor through
the two maps that are coequalized by $U\pi$, with the factorization
provided by the map $\etahat$.  This will establish that the perimeter
of the rectangle does commute, verifying that the unit is in fact a map
of $\D$-multicategories.  

We must still verify the two claimed diagrams.  

For the first diagram, we first expand the right vertical arrow
as $(\eta,I_\D\circ S,\D_0\eta\circ S)$, from the definition of
the unit map to $U\Lhat M_1$.  Now
projecting to the first factor $\D_0M_1$ of the target gives
us the diagram
\[
\xymatrix{
M_1\times_{\D_0M_0}\D_1M_0
\ar[r]^-{\psi_1}
\ar[d]_-{\eta\times_{\eta}\eta}
&M_1
\ar[d]^-{\eta}
\\\D_0M_1\times_{\D_0^2M_0}\D_0\D_1M_0
\ar[r]_-{\D_0\psi_1}
&\D_0M_1,}
\]
which is just a naturality diagram for $\eta:1\to\D_0$, since the lower
left corner is isomorphic to $\D_0(M_1\times_{\D_0M_0}\D_1M_0)$.  Projecting
to the second factor $\D_1M_0$ of the target gives us the diagram
\[
\xymatrix{
M_1\times_{\D_0M_0}\D_1M_0
\ar[r]^-{\psi_1}
\ar[d]_-{p_2}
&M_1
\ar[d]^-{S}
\\\D_1M_0
\ar[r]_-{S_\D}
&\D_0M_0
\ar[r]_-{I_\D}
&\D_1M_0,}
\]
which is a consequence of the requirement that the presheaf action $\psi_1$
preserve the structure maps of source type.  Projection to the third factor
is almost the same diagram: we just compose with $\D_0\eta$ instead of $I_\D$:
\[
\xymatrix{
M_1\times_{\D_0M_0}\D_1M_0
\ar[r]^-{\psi_1}
\ar[d]_-{p_2}
&M_1
\ar[d]^-{S}
\\\D_1M_0
\ar[r]_-{S_\D}
&\D_0M_0
\ar[r]_-{\D_0\eta}
&\D_0^2M_0.}
\]
This completes verification of the first claimed diagram.  

For the second one, we claim that both ways around the square coincide with
the map $\eta\times(1,\D_0\eta\circ S_\D)$.  For the lower left composite,
unpacking $1\times\phi\times1$ results in the claim that
\[
\xymatrix{
M_1\times_{\D_0M_0}\D_1M_0
\ar@/^3pc/[dddr]^-{\eta\times(1,\D_0\eta\circ S_\D)}
\ar[d]_-{\eta\times_\eta(\eta,I_\D\circ S_\D,\D_0\eta\circ S_\D)}
\\\D_0M_1\times_{\D_0^2M_0}\D_0\D_1M_0\times\D_1M_0\times\D_0^2M_0
\ar[d]_-{1\times I_\D\times1^2}
\\\D_0M_1\times_{\D_0^2M_0}\D_1^2M_0\times\D_1M_0\times\D_0^2M_0
\ar[d]_-{1\times_{\mu}\mu\times1^2}
\\\D_0M_1\times\D_1M_0\times\D_1M_0\times\D_0^2M_0
\ar[r]_-{1\times\gamma_\D\times1}
&\D_0M_1\times\D_1M_0\times\D_0^2M_0}
\]
commutes.  To check this, we first simplify the left column by checking 
the projections to each factor in the target.  The first $\D_0M_1$ clearly
is the result of just $\eta_0:M_1\to\D_0M_1$ from the first factor of the
source.  Everything else arises from the factor $\D_1M_0$ in the source.
The map to the second factor $\D_1M_0$ in the bottom left corner is the
identity, because of the diagram
\[
\xymatrix{
&\D_0\D_1M_0
\ar[d]^-{I_\D}
\\\D_1M_0
\ar[ur]^-{\eta_0}
\ar[r]^-{\eta_1}
\ar[dr]_-{=}
&\D_1^2M_0
\ar[d]^-{\mu}
\\&\D_1M_0.}
\]
The map to the third factor, also $\D_1M_0$, is simply $I_\D\circ S_\D$, 
and therefore the composite with the multiplication $\gamma_\D$ in
the bottom arrow gives us the identity because of the commuting triangle
\[
\xymatrix@C+10pt{
\D_1M_0
\ar[dr]_-{=}
\ar[r]^-{(1,I_\D\circ S_\D)}
&\D_1M_0\times_{\D_0M_0}\D_1M_0
\ar[d]^-{\gamma_\D}
\\&\D_1M_0.}
\]
We conclude that the left column simplifies as claimed.

What remains is to show that the upper right composite also simplifies 
to $\eta\times(1,\D_0\eta\circ S_\D)$.  Here we begin by simplifying the
top arrow as follows:
\[
\xymatrix@C+30pt{
M_1\times_{\D_0M_0}\D_1M_0
\ar[r]^-{\eta\times_{\D_0\eta}\D_1\eta}
\ar[dr]^-{\phantom{MMMMM}(\eta,I_\D\circ S,\D_0\eta\circ S)\times\D_1\eta}
\ar@/_1.5pc/[ddr]_-{(\eta,I_\D\circ S)\times\D_1\eta\phantom{MM}}
&U\Lhat M_1\times_{\D_0^2M_0}\D_1\D_0M_0
\ar[d]^-{=}
\\&\D_0M_1\times\D_1M_0\times\D_0^2M_0\times_{\D_0^2M_0}\D_1\D_0M_0
\ar[d]^-{\kappa_1\times_\mu1}_-{\cong}
\\&\D_0M_1\times\D_1M_0\times\D_1\D_0M_0
}
\]
Now we need to discuss the presheaf action map on $U\Lhat M_1$.  The presheaf
action on an underlying $\D$-multicategory $U\C$ is given by the composite
\begin{gather*}
\xymatrix{
U\C_1\times_{\D_0U\C_0}\D_1U\C_0=U\C_1\times_{\D_0\C_0}\D_1\C_0
\cong\C_1\times_{\C_0}\D_0\C_0\times_{\D_0\C_0}\D_1\C_0
\ar[r]^-{\kappa_1\times_{\xi_0}1}_-{\cong}
&\C_1\times_{\C_0}\D_1\C_0}
\\\xymatrix@C+10pt{
\relax
\ar[r]^-{1\times\D_1I}
&\C_1\times_{\C_0}\D_1\C_1
\ar[r]^-{1\times(\xi_1,S_\D S)}
&\C_1\times_{\C_0}\C_1\times_{\C_0}\D_0\C_0
\ar[r]^-{\gamma_\C\times1}
&\C_1\times_{\C_0}\D_0\C_0\cong U\C_1.}
\end{gather*}
Next, replacing $\C$ with $\Lhat M$, we can replace $\C_1$ with $\D_0M_1\times_{\D_0M_0}\D_1M_0$,
and $\C_0$ with $\D_0M_0$.  Furthermore, the identity arrow 
$I_{\Lhat M}:\D_0M_0\to\D_0M_1\times_{\D_0M_0}\D_1M_0$ is given
by the components $(\D_0I,I_\D)$.  Starting with the term
\[
\C_1\times_{\C_0}\D_1\C_0\cong\D_0M_1\times_{\D_0M_0}\D_1M_0\times_{\D_0M_0}\D_1\D_0M_0, 
\]
and suppressing
the subscript $\D_0M_0$'s
as usual, we can express the presheaf action on $U\Lhat M_1$ as the composite
\begin{gather*}
\xymatrix@C+40pt{
\D_0M_1\times\D_1M_0\times\D_1\D_0M_0
\ar[r]^-{1^2\times(\D_1\D_0I,\D_1I_\D)}
&\D_0M_1\times\D_1M_0\times\D_1\D_0M_1\times_{\D_1\D_0M_0}\D_1^2M_0}
\\\xymatrix@C+25pt{
\relax
\ar[rr]^-{1^2\times(\mu\circ T_\D)\times_{(\mu\circ T_\D)}(\mu, S_\D^2)}
&&\D_0M_1\times\D_1M_0\times\D_0M_1\times\D_1M_0\times\D_0^2M_0}
\\\xymatrix@C+30pt{
\relax
&\D_0M_1\times\D_1M_1\times\D_1M_0\times\D_0^2M_0
\ar[l]_-{1\times(\D_1T,S_\D)\times1^2}^-{\cong}
}
\\\xymatrix@C+30pt{
\relax
\ar[r]^-{1\times(T_\D,\theta)\times1^2}
&\D_0M_1\times\D_0M_1\times\D_1M_0\times\D_1M_0\times\D_0^2M_0
}
\\\xymatrix@C+15pt{
\relax
\ar[r]^-{\gamma_M\times\gamma_\D\times1}
&\D_0M_1\times\D_1M_0\times\D_0^2M_0.}
\end{gather*}
Precomposing this with
\[
\xymatrix@C+30pt{
M_1\times\D_1M_0
\ar[r]^-{(\eta,I_\D\circ S)\times\D_1\eta}
&\D_0M_1\times\D_1M_0\times\D_1\D_0M_0,}
\]
we wish to show that the entire composite ends up being $\eta\times(1,\D_0\eta\circ S_\D)$.
We proceed in steps, since the overall diagram is a bit large.
First, we have
\[
\xymatrix@C+30pt{
M_1\times\D_1M_0
\ar[r]^-{(\eta,I_\D\circ S)\times\D_1\eta}
\ar[dr]_-{(\eta,I_\D\circ S)\times(\D_1(\D_0I\circ\eta),\D_1\eta_1)\phantom{MMMMMM}}
&\D_0M_1\times\D_1M_0\times\D_1\D_0M_0
\ar[d]^-{1^2\times(\D_1\D_0I,\D_1I_\D)}
\\&\D_0M_1\times\D_1M_0\times\D_1\D_0M_1\times_{\D_1\D_0M_0}\D_1^2M_0.}
\]
This diagram commutes by projecting to each factor of the target,
and the only one that isn't immediate is the last one, which is a 
consequence of the category structure on $\D$, and in particular 
the diagram
\[
\xymatrix{
M_0\ar[r]^-{\eta_0}\ar[dr]_-{\eta_1}
&\D_0M_0\ar[d]^-{I_\D}
\\&\D_1M_0.}
\]
Next, we have the following diagram, which can be pasted to the previous step,
\[
\xymatrix{
&\D_0M_1\times\D_1M_0\times\D_1\D_0M_1\times_{\D_1\D_0M_0}\D_1^2M_0
\ar[dd]^-{1^2\times(\mu\circ T_\D)\times_{(\mu\circ T_\D)}(\mu,S_\D^2)}
\\M_1\times\D_1M_0
\ar[ur]^-{(\eta,I_\D\circ S)\times(\D_1(\D_0I\circ\eta),\D_1\eta_1)\phantom{MMMMMM}}
\ar[dr]_-{(\eta,I_\D\circ S)\times(\D_0I\circ T_\D,1,\D_0\eta\circ S_\D)\phantom{MMMMMM}}
\\&\D_0M_1\times\D_1M_0\times\D_0M_1\times\D_1M_0\times\D_0^2M_0.}
\]
Again, this commutes by projecting to each factor of the target, and the first two
are immediate.  The third one, to the second $\D_0M_1$, commutes because of
the following diagram:
\[
\xymatrix{
\D_1M_0
\ar[r]^-{\D_1\eta}
\ar[d]_-{T_\D}
&\D_1\D_0M_0
\ar[r]^-{\D_1\D_0I}
\ar[d]^-{T_\D}
&\D_1\D_0M_1
\ar[d]^-{T_\D}
\\\D_0M_0
\ar[r]^-{\D_0\eta}
\ar[dr]_-{=}
&\D_0^2M_0
\ar[r]^-{\D_0^2I}
\ar[d]^-{\mu}
&\D_0^2M_1
\ar[d]^-{\mu}
\\&\D_0M_0
\ar[r]_-{\D_0I}
&\D_0M_1.}
\]
Projection to the fourth factor, the second $\D_1M_0$, merely expresses
the monad identity
\[
\xymatrix{
\D_1M_0
\ar[r]^-{\D_1\eta}
\ar[dr]_-{=}
&\D_1^2M_0
\ar[d]^-{\mu}
\\&\D_1M_0.}
\]
And the last one, to $\D_0^2M_0$, is the elementary diagram
\[
\xymatrix{
\D_1M_0
\ar[r]^-{\D_1\eta_1}
\ar[d]_-{S_\D}
&\D_1^2M_0
\ar[dr]^-{S_\D^2}
\ar[d]_-{S_\D}
\\\D_0M_0
\ar[r]_-{\D_0\eta_1}
&\D_0\D_1M_0
\ar[r]_-{\D_0S}
&\D_0^2M_0.}
\]
Next, to be pasted to the previous step, is the diagram
\[
\xymatrix{
&\D_0M_1\times\D_1M_0\times\D_0M_1\times\D_1M_0\times\D_0^2M_0
\\M_1\times\D_1M_0
\ar[ur]^-{(\eta,I_\D\circ S)\times(\D_0I\circ T_\D,1,\D_0\eta\circ S_\D)\phantom{MMMMMM}}
\ar[dr]_-{(\eta,I_\D I\circ S)\times(1,\D_0\eta\circ S_\D)\phantom{MMMMMM}}
\\&\D_0M_1\times\D_1M_1\times\D_1M_0\times\D_0^2M_0.
\ar[uu]_-{1\times(\D_1T,S_\D)\times1^2}^-{\cong}}
\]
Projection to the first factor is clear, and to the second factor, $\D_1M_0$, 
is a consequence of 
\[
\xymatrix{
M_1\ar[r]^-{S}
&\D_0M_0
\ar[r]^-{I_\D}
\ar[dr]_-{I_\D I}
&\D_1M_0
\ar[dr]^-{=}
\ar[d]^-{\D_1I}
\\&&\D_1M_1
\ar[r]_-{\D_1T}
&\D_1M_0.}
\]
Projection to the third factor relies on the fact that the map factors through the
common projection to $\D_0M_0$ in the pullback $M_1\times_{\D_0M_0}\D_1M_0$
from which all these maps emanate.  The relevant diagram is as follows:
\[
\xymatrix{
M_1\times\D_1M_0\ar[r]^-{p_1}
\ar[d]_-{p_2}
&M_1
\ar[d]^-{S}
\\\D_1M_0
\ar[r]_-{T_\D}
&\D_0M_0
\ar[r]^-{\D_0I}
\ar[dr]_-{I_\D I}
&\D_0M_1
\ar[dr]^-{=}
\ar[d]^-{I_\D}
\\&&\D_1M_1
\ar[r]_-{S_\D}
&\D_0M_1.}
\]
The other two factors are clear.  We next paste on the diagram
\[
\xymatrix{
&\D_0M_1\times\D_1M_1\times\D_1M_0\times\D_0^2M_0
\ar[dd]^-{1\times(T_\D,\theta)\times1^2}
\\M_1\times\D_1M_0
\ar[ur]^-{(\eta,I_\D I\circ S)\times(1,\D_0\eta\circ S_\D)\phantom{MMMMMM}}
\ar[dr]_-{(\eta,\D_0I\circ S)\times(I_\D\circ T_\D,1,\D_0\eta\circ S_\D)\phantom{MMMMMM}}
\\&\D_0M_1\times\D_0M_1\times\D_1M_0\times\D_1M_0\times\D_0^2M_0.}
\]
Here projection to the first factor is clear, and to the second is a consequence
of 
\[
\xymatrix{
\D_0M_0
\ar[r]^-{\D_0I}
\ar[dr]_-{I_\D I}
&\D_0M_1
\ar[dr]^-{=}
\ar[d]^-{I_\D}
\\&\D_1M_1
\ar[r]_-{T_\D}
&\D_0M_1.}
\]
For the third projection, to the first factor of $\D_1M_0$, we appeal to the
fact that the map factors through $S:M_1\to\D_0M_0$, and so can just as well
be expressed factoring through $T_\D:\D_1M_0\to\D_0M_0$, since the source
is the pullback along these two maps.  The upper map can therefore be considered
$I_\D I\circ T_\D:\D_1M_0\to\D_1M_1$.  We then expand the definition of $\theta$
as the composite
\[
\xymatrix{
\D_1M_1
\ar[r]^-{\D_1S}
&\D_1\D_0M_0
\ar[r]^-{\D_1I_\D}
&\D_1^2M_0
\ar[r]^-{\mu}
&\D_1M_0,}
\]
and appeal to the following diagram:
\[
\xymatrix{
\D_1M_0
\ar[r]^-{T_\D}
&\D_0M_0
\ar[r]^-{I_\D}
\ar[d]_-{I_\D I}
&\D_1M_0
\ar[dl]_-{\D_0I}
\ar[d]_-{\D_1\eta_0}
\ar[dr]_-{\phantom{MMM}\D_1\eta_1}
\ar@/^.3pc/[drr]^-{=}
\\&\D_1M_1
\ar[r]_-{\D_1S}
&\D_1\D_0M_0
\ar[r]_-{\D_1I_\D}
&\D_1^2M_0
\ar[r]_-{\mu}
&\D_1M_0.}
\]
Projection to the last two factors is trivial.  Finally, we paste in the diagram
\[
\xymatrix{
&\D_0M_1\times\D_0M_1\times\D_1M_0\times\D_1M_0\times\D_0^2M_0
\ar[dd]^-{\D_0\gamma_M\times\gamma_\D\times1}
\\M_1\times\D_1M_0
\ar[ur]^-{(\eta,\D_0I\circ S)\times(I_\D\circ T_\D,1,\D_0\eta\circ S_\D)\phantom{MMMMMM}}
\ar[dr]_-{\eta\times(1,\D_0\eta\circ S_\D)\phantom{MMM}}
\\&\D_0M_1\times\D_1M_0\times\D_0^2M_0.
}
\]
Here projection to the first factor is a consequence of the right unital property
of $\gamma_M$, and projection to the second follows from the left unital
property for $\gamma_\D$.  This concludes the proof that the unit map
preserves the presheaf structure, and therefore that we have, in fact,
constructed a left adjoint to the forgetful functor.

\end{document}